\documentclass[]{article}
\usepackage[utf8]{inputenc}
\usepackage{makeidx}
\usepackage{setspace}
\usepackage{amssymb , amsmath, amsthm,graphicx, float,enumerate}
\usepackage{xcolor}
\usepackage{comment}
\usepackage{todonotes}
\usepackage{float}
\usepackage{caption}
\usepackage{subcaption}
\usepackage{authblk}
\newtheorem{thm}{Theorem}[subsection]

\newtheorem{rem}{Remark}[subsection]
 \newtheorem{prop}{Proposition}[subsection]

\newtheorem{lemma}{Lemma}[subsection]
\newtheorem{cor}{Corollary}[subsection]

\title{Semiclassical estimates for eigenvalue means of Laplacians on spheres}
\author{$^\dagger$Davide Buoso, $^\star$Paolo Luzzini, $^\ddag$Luigi Provenzano, and $^{\star\star}$Joachim Stubbe}%

\date{\today}

\begin{document}

\maketitle

\noindent
{\bf Abstract:}
We compute three-term semiclassical asymptotic expansions of counting functions and Riesz-means of the eigenvalues of the Laplacian on spheres and hemispheres, for both Dirichlet and Neumann boundary conditions. Specifically for Riesz-means we prove upper and lower bounds involving asymptotically sharp shift terms, and we extend them to domains of $\mathbb S^d$. We also prove a Berezin-Li-Yau inequality for domains contained in the hemisphere $\mathbb S^2_+$. Moreover, we consider polyharmonic operators for which we prove analogous results that highlight the role of dimension for P\'olya-type inequalities. Finally, we provide sum rules for Laplacian eigenvalues on spheres and compact two-point homogeneous spaces.

\vspace{11pt}

\noindent
{\bf Keywords:} eigenvalues; P\'olya's conjecture; spheres and hemispheres; compact homogeneous spaces; polyharmonic operators; Riesz-means; Berezin-Li-Yau inequality; Kr\"oger inequality; averaged variational principle; semiclassical expansions; asymptotically sharp estimates.

\vspace{6pt}
\noindent
{\bf 2020 Mathematics Subject Classification:} 58C40, 35P15, 35J25, 35J40, 53A05.


\section{Introduction}
In 1954, in the first edition of its monograph \cite{Po},  P\'olya stated his celebrated conjecture that the leading term in Weyl's law separates the spectrum of the Dirichlet Laplacian from that of the Neumann Laplacian. More precisely, given $\Omega$ an open bounded set in $\mathbb R^2$, and given $\lambda_k(\Omega), \mu_k(\Omega)$ the $k$-th eigenvalue of the Laplace operator $-\Delta$ with respectively Dirichlet and Neumann boundary conditions, Pólya conjectured that
\begin{equation}
\label{polconj}
\mu_{k+1}(\Omega)\le \frac{4\pi k}{|\Omega|}\le \lambda_k(\Omega)
\end{equation}
for any $k\in\mathbb N$ (the inequality for $\lambda_k(\Omega)$ is understood for $k\geq 1$). The same conjecture has been then formulated also in higher dimensions, and for $\Omega\subset\mathbb R^d$, $d\geq 2$ reads
\begin{equation}
\label{polconj_d}
\mu_{k+1}(\Omega)\le \frac{4\pi^2}{\omega_d^{2/d}}\left(\frac{ k}{|\Omega|}\right)^{2/d}\le \lambda_k(\Omega)
\end{equation}
for any $k\in\mathbb N$, where $\omega_d$ denotes the volume of the unit ball in $\mathbb R^d$. The quantity in the middle of \eqref{polconj_d} is the semiclassical approximation of the eigenvalues $\lambda_k(\Omega),\mu_k(\Omega)$ since $\lambda_k(\Omega),\mu_k(\Omega)\sim\frac{4\pi^2}{\omega_d^{2/d}}\left(\frac{k}{|\Omega|}\right)^{2/d}$ for $k\to\infty$,  as already proved by Weyl \cite{W}. For this reason we say that these inequalities are asymptotically sharp (in leading order).


Pólya himself proved inequalities \eqref{polconj} when $\Omega$ is  a tiling domain \cite{Po2} and, while there have been some developments in two and in higher dimensions (see e.g. \cite{ filopolya,frei_d, FrSa22,Lap97, LePoSh22}), the general case remains  at the moment an open problem.

On the other hand, inequalities \eqref{polconj_d} in an averaged (weaker) version,
\begin{equation}
\label{blyk}
\frac 1 k \sum_{j=1}^k\mu_j(\Omega)\le \frac{d}{d+2}\frac{4\pi^2}{\omega_d^{2/d}}\left(\frac{ k}{|\Omega|}\right)^{2/d}\le \frac 1 k \sum_{j=1}^k \lambda_j(\Omega),
\end{equation}
were actually proven for any $\Omega\subseteq \mathbb R^d$ by Berezin \cite{berezin} and Li and Yau \cite{LY} for 
the Dirichlet eigenvalues and by Kr\"oger \cite{Kr} for Neumann eigenvalues. For this reason, the first inequality in \eqref{blyk} is  now known as  Kr\"oger inequality while the second one as Berezin-Li-Yau inequality. Note that inequalities \eqref{blyk} are asymptotically sharp, therefore P\'olya's conjecture holds in an averaged sense.
 We remark that inequalities \eqref{polconj_d} on each eigenvalue can be rephrased as (reversed) bounds on the counting functions, i.e.,
 \[
 N^D(z)=\#\{\lambda_k(\Omega)\leq z\}, \quad N^N(z)=\#\{\mu_k(\Omega)\leq z\} \qquad \forall z\,\geq 0,
 \]
  while inequalities \eqref{blyk} on eigenvalue averages as (reversed) bounds on the first Riesz-means, i.e.
  \[
  R_1^D(z)=\sum_j(z-\lambda_j(\Omega))_+, \quad R_1^N(z)=\sum_j(z-\mu_j(\Omega))_+ \qquad \forall z\,\geq 0.
  \]
 Here $\displaystyle{\sum_j}$ denotes the summation over all $j\in\mathbb N$. P\'olya's conjecture \eqref{polconj_d} is in some sense justified also by the semiclassical asymptotic expansions of the eigenvalues, which is equivalent to the expansion of the counting functions as $z\to+\infty$:
\begin{equation}\label{two-terms generic}
\begin{split}
  &N^D(z)\sim L_{0,d}^{class}|\Omega|\,z^{\frac{d}{2}}-\frac{1}{4}\,L_{0,d-1}^{class}|\partial\Omega|\,z^{\frac{d-1}{2}}\\
  &N^N(z)\sim L_{0,d}^{class}|\Omega|\,z^{\frac{d}{2}}+\frac{1}{4}\,L_{0,d-1}^{class}|\partial\Omega|\,z^{\frac{d-1}{2}}.
  \end{split}
\end{equation}
Here $L_{0,d}^{class},L_{0,d-1}^{class}$ are the semiclassical constants, which depend only on $d$ (see \eqref{semiclassical-constants-Laplacian} for the precise definition). The first term in the expansions \eqref{two-terms generic} was proved by Weyl \cite{W}, who later conjectured this two-term expansion (see \cite{W2}) that was proved only much later \cite{Iv,melrose} under suitable geometric conditions. More precisely, the set of periodic points of the geodesic billiard needs to have Lebesgue measure zero. We refer to \cite{SV} for a more exhaustive discussion on the history of semiclassical expansions, as well as two-term  asymptotics for more general elliptic operators. 
Let us mention that the analogous expansions for the (more regular) first Riesz-mean hold under much weaker assumptions on the domain (see e.g., \cite{frgei,FrLa20}).

 From this, the natural question arises whether the semiclassical expansions with the leading order term, or more terms as in \eqref{two-terms generic},  yield upper or lower bounds for all finite $z$.  For the counting functions $N^D(z), N^N(z)$ this amounts to P\'olya's conjecture; for the Riesz-means $R_1^D(z), R_1^N(z)$ this corresponds to the well-known bounds by Li-Yau \cite{LY} and Kr\"oger \cite{Kr}.

In order to further understand the asymptotic behavior of the eigenvalues, several authors  investigated Weyl-sharp inequalities for Riesz-means (or equivalently for  eigenvalue averages)   improving \eqref{blyk} with lower order  terms and also  reversed inequalities. In this regard, we mention the works \cite{fr_aver,Ge11,HaPrSt19,HaSt18,Me03,We08}. 

While the above discussion concerns inequalities and expansions in the Euclidean setting, a natural extension is to investigate similar questions for the Laplacian on different manifolds, for example when $\Omega$ is a domain of the sphere $\mathbb S^d=\{x\in\mathbb{R}^{d+1}:|x|=1\}$. There has been a growing interest in the study of the properties of domains in manifolds for which Pólya's conjecture does not hold, the emblematic cases being the sphere $\mathbb S^d$ and the hemisphere $\mathbb S^d_+=\mathbb S^d \cap \{x_{d+1} > 0\}$, $d\geq 2$. The hemisphere $\mathbb S^2_+$ is an exception: Bérard and Besson \cite{BeBe80} showed that Pólya's conjecture is true for the hemisphere, the quarter-sphere, and the eighth-sphere in dimension $2$. Also, Freitas, Mao, and Salvessa \cite{FrMaSa22} carry out a careful analysis of Pólya's conjecture on $d$-dimensional spheres and hemispheres, identifying precise subsequences of eigenvalues for which Pólya's conjecture fails, and others for which it is valid. Moreover, they prove a Pólya-type inequality with an asymptotically sharp correction term measuring how far the eigenvalues are from the leading term in Weyl's law. They also deduce Pólya-type inequalities for the eigenvalues on the whole sphere in the same spirit. These bounds suggest that a second term for the counting function $N(z)$ for $\mathbb S^d$ should be oscillating, but unbounded (which is expected due to the order of the remainder, which is known, see \cite{canzani,SV}). Notice that, as the sphere violates the necessary geometrical conditions, an expansion like \eqref{two-terms generic} cannot hold and is actually known to be false (cf. \cite[Example 1.2.5]{SV}). Because of this, these bounds are fundamental to provide a better understanding of the remainder. The same conclusion can be deduced also for hemispheres, where again the expansion  \eqref{two-terms generic} cannot be inferred from classical arguments (cf.     \cite[Section 1.3.1]{SV}).


Considering sharp estimates of Riesz-means on domains of compact homogeneous manifolds (in particular spheres), Strichartz \cite{stri} proved a series of asymptotically sharp inequalities. We remark that the starting point is an observation due to Colin de Verdière and Gallot \cite{Ga} relating the Riesz-mean in a domain with that in the manifold containing the domain. Improvements in the case of the sphere have been proved by Ilyin and Laptev \cite{ilyin_laptev}. We also mention El Soufi, Harrell, Ilias, Stubbe \cite{EHIS} where the authors present the so-called {\it averaged variational principle}, which is an efficient way to recover the result of Colin de Verdière and Gallot, and apply it to bound Riesz-means on general Riemannian manifolds, also for other types of operators \cite{bpls_buckling, bps}. Related bounds for eigenvalue averages on domains of Riemannian manifolds can be found in \cite{chengyang}. However we remark that for domains in a general Riemannian manifold sharp upper or lower bounds for Riesz means are not available.

In this paper we consider specifically the Laplacian on the sphere and on the hemisphere and derive asymptotics and Weyl-sharp upper and lower bound for Riesz-means and counting functions that complement and improve those already present in the literature.

Our first aim  is to investigate further terms in the asymptotic expansions for $N(z)$ in the case $\mathbb S^d$ and $N^D(z),N^N(z)$ in the case of $\mathbb S^d_+$ in order to clarify the behavior highlighted in \cite{FrMaSa22}.
We will also consider subsequent terms in the expansion of the more regular Riesz-mean $R_1(z)$ for $\mathbb S^d$ and $R_1^D(z),R_1^N(z)$ for $\mathbb S^d_+$. For example, for $R_1(z)$, the second term has a sign, but contains an oscillating part (see Theorem \ref{R1-d-sphere}). We highlight the relation with the results in \cite{stri}, where the lim-inf and lim-sup of the remainder term for $R_1(z)$ are computed.  On the other hand, on $\mathbb S^d_+$, we derive three-term expansions both for the counting functions and the Riesz-means (see Theorems \ref{ND}, \ref{R1-d-hemi-N}, \ref{NN},  \ref{R1-d-hemi-NNeu}).  Note that $R_1^D(z),R_1^N(z)$ have a second  term which is coherent with the expansion \eqref{two-terms generic} even though $N^D(z),N^N(z)$ do not. This behavior is the expected one since Riesz-means present a higher regularity than counting functions. Our results on asymptotic expansions for counting functions in some sense complete the study of \cite[\S 1.7]{SV} where the authors consider non-classical two-terms expansions for the eigenvalues of the degree operator on spheres and hemispheres, namely the operator $-\Delta+\frac{(d-1)^2}{4}$. Note that the analysis for this operator is somehow easier since the energy levels are given by $\left(l+\frac{d-1}{2}\right)^2$. With some effort, it is possible to recover two-terms expansions for the eigenvalues of the Laplacian from the expansions in \cite[\S 1.7]{SV}. However, up to our knowledge, the three-terms expansions esablished in our paper were not known.


Once precise spectral asymptotics are established, one naturally asks whether it is possible to obtain bounds, at least for the more regular Riesz-means. For example,  in $\mathbb{S}^2$ we are able to improve the lower bounds for $R_1(z)$  present in \cite{ilyin_laptev, stri} and derive sharp bounds with lower order terms also for $\mathbb{S}^2_+$.  As for the higher dimensional case, we  prove  upper and lower bounds for $R_1(z)$, containing an asymptotically sharp shift (see Theorems \ref{lo-shift} and \ref{thm:ubR1dshift}). Note that the upper bound with a shift in the form of Theorem \ref{thm:ubR1dshift} implies a Berezin-Li-Yau inequality for the shifted eigenvalues. The case of $\mathbb S^1$ is intrinsically different, since the leading term in Weyl's law is neither an upper bound nor a lower bound. Nevertheless, we provide an upper bound containing an asymptotically sharp shift (see Theorem \ref{R1-S1}).


The second aim of the present paper is to consider Berezin-Li-Yau bounds for Dirichlet eigenvalues on domains. It is well-known that it is not possible to bound from above with the leading term in Weyl's law the Riesz-mean $R_1^D(z)$. The natural counterexample is a domain which is invading the whole sphere: its Dirichlet spectrum is converging to the spectrum on the whole sphere, for which Berezin-Li-Yau inequality does not hold. However, bounds of Berezin-Li-Yau-type with a correction term can be obtained in the spirit of  \cite{Ga}, as done in  \cite{ilyin_laptev,stri}. In this paper we observe that if we restrict to domains on the hemisphere $\mathbb S^2_+$, then Berezin-Li-Yau bounds do hold (see Theorems \ref{thm:blyS2+} and \ref{thm:blyS2+imp}). Moreover, they hold for $\mathbb S^d_+$ when $d=3,4,5$ (see Theorem \ref{BLY-345}). For domains in $\mathbb S^d$ we complete the picture of \cite{ilyin_laptev,stri} by establishing Berezin-Li-Yau bounds with a shift term, which is asymptotically sharp when the domain is $\mathbb S^d$ (see Theorem \ref{bounds_domains_Sd}).


In order to provide a complete picture, we also consider the case of polyharmonic operators $(-\Delta)^p$ on $\mathbb S^d$. We prove asymptotically sharp upper and lower bounds for $R_1(z)$, and we highlight that the leading term in Weyl's law is in general neither an upper nor a lower bound for $R_1(z)$. We believe that it should be a lower bound exclusively for  $p<d$, while it is not for $p\geq d$, and we prove this statement for $p=2$ (see Remark \ref{rem_pd} and Theorem \ref{lower-R-1-bil-sphere}). This is an extension of the analogous results for the case of the Laplacian $p=1$. We then focus on the biharmonic operator ($p=2$),  for which we provide Berezin-Li-Yau inequalities for the Dirichlet problem on domains of $\mathbb S^2_+$ (Theorem \ref{b-l-y-poly-hemi}) and for the buckling problem on domains of $\mathbb S^2$ (Theorem \ref{BLY-buckling}), and Kr\"oger bounds for the Neumann eigenvalues on domains of $\mathbb S^d$ for all $d>3$ (Theorem \ref{kr-gen-poly}). 

Finally, we provide some sum rules for Laplacian eigenvalues on the sphere $\mathbb{S}^d$, which allow to get upper and lower bounds on the second Riesz-mean $R_2$ bypassing any variational approach, and which can be readily generalized to compact two-point homogeneous spaces. To this regard, we point out that having the explicit forms for the eigenvalues and their multiplicities for all the compact two-point homogeneous spaces (Proposition \ref{prop:EigHS}) would allow to extend the previously discussed results in all these cases (see also \cite{stri} where the real projective space case is discussed). 


 For what concerns the techniques used, to obtain the results in $\mathbb{S}^d$ and   $\mathbb{S}^d_+$ we mainly exploit the fact that in both cases the eigenvalues and their multiplicities are completely known and easily described, and this in turn allows for a somewhat explicit but rather complicated representation of any related quantity. Careful manipulations then permit to recover in a new manner  known bounds and  to derive new ones. The results for domains instead take advantage of what we proved for $\mathbb{S}^d$ together with the {\it averaged variational principle} of Harrell and Stubbe (see Theorem \ref{thm:AVP}, see also \cite{ha_st_1}).
  

The paper is organized as follows. In Section \ref{sec:pre} we introduce the notation and some preliminaries: we state the eigenvalue problems, the functional setting and the tools needed in our analysis. Section \ref{sec:2sphere} contains our results in dimension 2, that is for the sphere $\mathbb{S}^2$, the hemisphere $\mathbb{S}^2_+$, and domains of the hemisphere. Then in Section \ref{sec:dsphere} we consider the $d$-dimensional case of the sphere $\mathbb{S}^d$, its domains, and the 
hemisphere $\mathbb{S}^d_+$. In Section \ref{S1} we deal with the case of the circle $\mathbb{S}^1$. 
The case of higher order problems, such as polyharmonic operators and the buckling problem, is treated in Section 
\ref{higher_order}. Finally in Section \ref{compact-two-points} we present the sum rules for the Laplacian on spheres and other compact homogeneous spaces. For the sake of clarity, we have postponed some technical results to Appendix \ref{app_A}. In Appendix \ref{app_C} we discuss a duality principle for Riesz-means.

\section{Preliminaries and notation}\label{sec:pre}

Let $M^d$ be a $d$-dimensional, compact, Riemannian manifold, and let $\Omega$ be a domain in $M^d$ (possibly $\Omega=M^d$). We recall that $\Omega$ is called a domain if it is an open, bounded, connected set. By $L^2(\Omega)$ we denote the classical Lebesgue space of square integrable functions. By $H^m(\Omega)$ we denote the standard Sobolev space of functions in $L^2(\Omega)$ with all weak partial derivatives up to the order $m$ in $L^2(\Omega)$. By $H^m_0(\Omega)$ we denote the closure of $C^{\infty}_c(\Omega)$ in $H^m(\Omega)$ with respect to its standard norm. Throughout the paper, by $\mathbb N$ we denote the set of natural numbers including zero.

On $M^d$ we consider the (closed) eigenvalue problem for the Laplacian
\begin{equation}\label{lapM}
-\Delta u=\lambda u,
\end{equation}
and on domains $\Omega\subset M^d$ we consider the Dirichlet problem
\begin{equation}\label{dir_domain}
\begin{cases}
-\Delta u=\lambda u\,, & {\rm in\ }\Omega,\\
u=0\,, & {\rm on\ }\partial\Omega,
\end{cases}
\end{equation}
and the Neumann problem
\begin{equation}\label{neu_domain}
\begin{cases}
-\Delta u=\lambda u\,, & {\rm in\ }\Omega,\\
\partial_{\nu}u=0\,, & {\rm on\ }\partial\Omega.
\end{cases}
\end{equation}


We will understand problems \eqref{dir_domain} and \eqref{neu_domain} in their weak formulations. For problem \eqref{dir_domain} it amounts to finding a function $u\in H^1_0(\Omega)$ and a real number $\lambda\in\mathbb R$ such that
\begin{equation}\label{dir_domain_w}
\int_{\Omega}\nabla u\cdot\nabla\phi=\lambda\int_{\Omega} u\phi\,,\ \ \ \forall \phi\in H^1_0(\Omega).
\end{equation}
For problem \eqref{neu_domain}, the variational formulation is the same as \eqref{dir_domain_w} but with the energy space $H^1_0(\Omega)$ replaced by $H^1(\Omega)$. 


We denote the eigenvalues of \eqref{lapM} as
$$
0=\lambda_1<\lambda_2\leq\cdots\leq\lambda_j\leq\cdots\nearrow+\infty.
$$
As for the Dirichlet and Neumann problems \eqref{dir_domain}-\eqref{neu_domain} on domains of $ M^d$, we shall denote the eigenvalues by
$$
0<\lambda_1(\Omega)<\lambda_2(\Omega)\leq\cdots\leq\lambda_j(\Omega)\leq\cdots\nearrow+\infty
$$
and
$$
0=\mu_1(\Omega)<\mu_2(\Omega)\leq\cdots\leq\mu_j(\Omega)\leq\cdots\nearrow+\infty,
$$
respectively. In order to ease the notation, we will omit the explicit dependence on $\Omega$  when there is no possibility of confusion.


In many situations (e.g., $M^d=\mathbb S^d$, the round sphere), the eigenvalues of the Laplacian appear as {\it energy levels}, namely, the values assumed by the eigenvalues (without multiplicities) form a sequence which we shall denote by
$$
0=\lambda_{(0)}<\lambda_{(1)}<\lambda_{(2)}<\cdots<\lambda_{(l)}<\cdots\nearrow+\infty\,.
$$
Each eigenvalue corresponding to an energy level $\lambda_{(l)}$, $l\in\mathbb N$, has a certain multiplicity, which depends on $d$ and $l$, and which we shall denote by $m_{l,d}$. To clarify the situation, let us just consider the eigenvalues of the Laplacian on $\mathbb S^1$, which are given by the sequence
$$
0,1,1,4,4,9,9,\cdots,l^2,l^2,\cdots
$$
therefore $\lambda_{(l)}=l^2$, $l\in\mathbb N$, and $m_{0,1}=1$, $m_{l,1}=2$ for all $l\geq 1$. In general $m_{0,d}=1$ for all $d$. If we want to enumerate the eigenvalues of $\mathbb S^1$ in increasing order, counting multiplicities, we will denote them as $\lambda_1=0,\lambda_2=1,\lambda_3=1,\lambda_4=4,\lambda_5=4,...$.


In this paper we shall consider mainly the case of the $d$-dimensional round sphere $\mathbb S^d$, and its domains. The round sphere is a compact, two-point homogeneous space. In Section \ref{compact-two-points} we shall discuss the case of the others spaces in this family, for which we establish sum rules.


Concerning $\mathbb S^d$ and its domains, we will consider semiclassical estimates for {\it Riesz-means} of eigenvalues, namely for
\begin{equation*}
R_{\gamma}(z)=\sum_{j}(z-\lambda_j)_+^{\gamma} \qquad \forall z \geq 0
\end{equation*}
where $\gamma\geq 0$, $a_+$ denotes the positive part of a real number $a$, and the sum is taken over $j\in\mathbb N$.  As a convention, when the summation is over all $j \in \mathbb{N}$, we will just write the index $j$ at the bottom of the summation symbol. If the sum is over some subset $J\subset\mathbb N$ we will write $\sum_{j\in J}$; if the sum starts from some $k_0\in\mathbb N$ we will write $\sum_{j\geq k_0}$. When $\gamma=0$, then $R_0(z)$ is just the counting function $N(z)$ which counts the number of eigenvalues $\lambda_j$ below $z$. 


We will denote by $R_1(z)$ and $N(z)$ the Riesz-mean and the counting function for the whole manifold $M^d$, namely, for problem \eqref{lapM}.  Moreover, we shall denote by $R_1^D(z),N^D(z)$ and by $R_1^N(z),N^N(z)$ the Riesz-means and the counting functions for the Dirichlet \eqref{dir_domain} and Neumann \eqref{neu_domain} problems on domains of $M^d$, respectively.


In this paper we will be mainly interested in $\gamma=1$, i.e., the first Riesz-mean, since semiclassical estimates can be deduced in a very efficient way by means of the {\it averaged variational principle}, introduced by Harrell and  Stubbe in \cite{EHIS,ha_st_1}, generalizing a work of Kr\"oger \cite{Kr} by averaging over test functions which form a complete frame of the underlying Hilbert space. We shall state it here for sake of completeness.
\begin{thm}\label{thm:AVP}
Let $H$ be a self-adjoint operator in a Hilbert space $(\mathcal{H}, \langle\cdot,\cdot,\rangle_\mathcal{H})$, the spectrum of which is discrete at least in its lower portion, and we denote it by
\[
\omega_1 \leq \omega_2 \leq \cdots \leq \omega_j \leq \cdots
\]
 with corresponding orthonormalized eigenvectors $\{g_j\}_{j \in \mathbb{N}\setminus\{0\}}$.
The closed quadratic form corresponding to $H$ is denoted $Q(\varphi,\varphi)$ for any $\varphi$ in
the quadratic form domain $\mathcal{Q}(H) \subset \mathcal{H}$. Let $f_p \in \mathcal{Q}(H)$ be a family
of vectors indexed by a variable $p$ ranging over a measure space $(\mathfrak{M}, \Sigma, \sigma)$. Suppose  that $\mathfrak{M}_0$ is a subset of $\mathfrak{M}$.
Then for any $z \in \mathbb{R}$,
\begin{equation*}
\sum_j(z-\omega_j)_+ \int_{\mathfrak{M}}\left| \langle g_j, f_p \rangle_{\mathcal{H}}
\right|^2 \, d\sigma_p \geq \int_{\mathfrak{M}_0} \left( z\|f_p \|_{\mathcal{H}}^2-Q(f_p,f_p) \right)\,d\sigma_p,
\end{equation*}
provided that the integrals converge.
\end{thm}

In particular, in the situation of $\mathbb S^d$ the {\it averaged variational principle} turns out to be equivalent to generalizations of the
Berezin-Li-Yau method, which was first observed by Colin de Verdière and Gallot \cite{Ga}, and employed in various form in Ilyin and Laptev \cite{ilyin_laptev} and Strichartz \cite{stri}. This application of the averaged variational principle to recover in an efficient way the results of Strichartz \cite{stri} is contained in \cite{EHIS}, and it is employed not only for homogeneous spaces but for more general Riemannian manifolds.

We shall denote by $L_{\gamma,d}^{class}$ the semiclassical constant for Laplacian eigenvalues in dimension $d$, which is given by
\begin{equation}\label{semiclassical-constants-Laplacian}
  L_{\gamma,d}^{class}=(4\pi)^{-d/2}\frac{\Gamma(\gamma+1)}{\Gamma(\gamma+1+d/2)}.
\end{equation}
It is also convenient to recall that
$$
L_{0,d}^{class}|\mathbb S^d|=\frac{2}{\Gamma(d+1)}=\frac{2}{d!}\,,\ \ \ L_{1,d}^{class}|\mathbb S^d|=\frac{4}{(d+2)\Gamma(d+1)}=\frac{4}{(d+2)d!}.
$$

Finally, we introduce the {\it fluctuation function} $\psi$ defined by

\begin{equation}\label{fluct}
\psi(\eta) = \eta - \lfloor\eta\rfloor -\frac{1}{2} \qquad 
\forall \eta \geq 0,
\end{equation}
where $\lfloor\eta\rfloor$ denotes the integer part of $\eta$.

\section{The two-dimensional sphere $\mathbb S^2$ and the hemisphere $\mathbb S^2_+$}\label{sec:2sphere}

In this section we will consider semiclassical estimates for Laplacian eigenvalues in the exceptional case of the two-dimensional sphere. In particular, we shall consider the closed problem on $\mathbb S^2$, the Dirichlet and Neumann problems for the hemisphere $\mathbb S^2_+$, and on domains of $\mathbb S^2_+$. 

\subsection{The sphere $\mathbb S^2$}\label{2sphere}
As is well-known, the energy levels of the Laplacian on $\mathbb{S}^2$ are given by 
$\lambda_{(l)}= l(l+1)$ with corresponding multiplicities
$m_{l,2}=2l+1$, $l \in \mathbb{N}$, see e.g., \cite{berger}. It is well-known \cite{SV,W} that Weyl's law for the counting function of the Laplacian eigenvalues on $\mathbb{S}^2$ reads
\begin{equation*}
N(z) = L^{class}_{0,2}|\mathbb{S}^2|z + o(z) = z+ o(z) \quad \mbox{ as } z\to +\infty,
\end{equation*}
and, accordingly, the semiclassical limit for the first Riesz-mean $R_1$ is
\begin{equation*}
R_1(z) = L^{class}_{1,2}|\mathbb{S}^2|z^2 + o(z^2) = \frac{1}{2}z^2+ o(z^2)  \quad \mbox{ as } z\to +\infty.
\end{equation*}

Strichartz \cite[(3.11)-(3.13) p. 166]{stri} proves a Weyl sharp lower bound with a correction term and a Weyl sharp upper bound   for the eigenvalue means of the Laplacian eigenvalues on $\mathbb S^2$. These bounds are equivalent to a  Weyl sharp shifted upper bound and a Weyl sharp lower bound the first Riesz-mean $R_1$ which we show in the following proposition. The upper bound was also shown by Ilyin and Laptev \cite{ilyin_laptev}. Our technique will allow a more careful analysis, improving these bounds in  Theorem \ref{bound-R-1-2-improved}.  For a discussion of  $d>2$ and our significant improvements based on the techniques introduced here, see Section \ref{dsphere}, in particular Theorems \ref{lo-shift} and \ref{thm:ubR1dshift} and the subsequent remarks.


\begin{prop}\label{prop:R1S2bounds}
For all $z\geq 0$, the following bounds hold for the first Riesz-mean $R_1$ of the Laplacian eigenvalues on $\mathbb{S}^2$:
\[
\frac{1}{2}z^2 \leq R_1(z) \leq \frac{1}{2}\left( z+\frac{1}{2}\right)^2.
\]
Equality in the lower bound holds if and only if $z = \lambda_{(l)}$
for some $l \in \mathbb{N}$. For the upper bound,
equality holds if and only if $z=(l+1)^2-\frac{1}{2} \in [\lambda_{(l)},\lambda_{(l+1)}]$ for some $l \in \mathbb{N}$.
\end{prop}

\begin{proof}
For the sake of simplicity, we prove the bounds for $z=w(w+1)$ where $w\geq 0$.   We note that
\begin{multline}\label{proof_prop_2}
     R_1(w(w+1))
       = \sum_{l=0}^{\lfloor w\rfloor}(2l+1)(w(w+1)-l(l+1))\\
       =\frac{1}{2}\, (w+\lfloor w\rfloor+1)(w+\lfloor w\rfloor+2)(w-\lfloor w\rfloor)(\lfloor w\rfloor+1-w)+\frac{1}{2}\,w^2(w+1)^2.
\end{multline}
Since $\lfloor w\rfloor\leq w<\lfloor w\rfloor+1$ the first term in the right hand side of the above equation is non-negative. Moreover, it equals zero if and only if $w \in \mathbb{N}$,
that is when $w(w+1)$ equals  an energy level $\lambda_{(l)}$ for some $l \in \mathbb{N}$. For the upper bound we write $R_1$ as follows
\begin{equation}\label{R1S2bounds1}
   R_1(w(w+1))
         =-\,\frac{1}{8}\, \bigg(-2(w+\lfloor w\rfloor+1)(w-\lfloor w\rfloor)+2\lfloor w\rfloor+1\bigg)^2+\frac{1}{2}\,\left(w(w+1)+\frac{1}{2}\right)^2.
 \end{equation}
Since
\begin{equation*}
     -2(w+\lfloor w\rfloor+1)(w-\lfloor w\rfloor)+2\lfloor w\rfloor+1
        =2\left(\lfloor w\rfloor+\frac{1}{2}\right)\left(\lfloor w\rfloor+\frac{3}{2}\right)-2\left(w+\frac{1}{2}\right)^2,
\end{equation*}
the first term in the right hand side of equation 
\eqref{R1S2bounds1} is non-positive and equals zero if 
\[
w=-\frac{1}{2}+\sqrt{\left(\lfloor w\rfloor+\frac{1}{2}\right)\left(\lfloor w\rfloor+\frac{3}{2}\right)}
\]
which has a solution 
$w = -\frac{1}{2}+\sqrt{\left(l+\frac{1}{2}\right)\left(l+\frac{3}{2}\right)}$ in each interval $[l,l+1]$.
Hence, recalling  the substitution $z=w(w+1)$ we have that
the equality in the upper bound holds if and only if
$z=(l+1)^2-\frac{1}{2}$ which is in the interval $[\lambda_{(l)},\lambda_{(l+1)}]$.
\end{proof}

As anticipated, a careful inspection of the proof of Proposition \ref{prop:R1S2bounds} allows to establish an improved two-sided bound for the first Riesz-mean with a sharp first term and 
a second term  of order $z$ with an oscillating, but positive, coefficient, thus improving the results of \cite{ilyin_laptev,stri}.

\begin{thm}\label{bound-R-1-2-improved}
  For all $z\geq0$, the following bounds hold for the first Riesz-mean $R_1$ of the Laplacian eigenvalues on $\mathbb{S}^2$:
 \begin{equation*}
  \frac{1}{2}\,z^2+2\left(\frac{1}{4}-\psi(w)^2\right)\left(z-\frac{\sqrt{z}}{2}\right)
  \leq R_1(z)\leq
  \frac{1}{2}\,z^2+2\left(\frac{1}{4}-\psi(w)^2\right)\left(z+\frac{\sqrt{z}}{2}+\frac{1}{2}\right),
\end{equation*}
where $\psi$ is the fluctuation function \eqref{fluct} and $w$ is defined by the relation $w(w+1)=z$.
Consequently, for any $\epsilon>0$:
\begin{equation*}
 \underset{z\to+\infty}{\lim}z^{-\epsilon-1/2}\left( R_1(z)- \frac{1}{2}\,z^2-2\left(\frac{1}{4}-\psi(w)^2\right)z\right)=0.
\end{equation*}
\end{thm}
\begin{proof}
It is sufficient to consider the third line of \eqref{proof_prop_2},  substitute $\lfloor w\rfloor$ with $w-\psi(w)-\frac{1}{2}$,  and use the bounds $w\leq \sqrt{z}\leq w+1$ and $|\psi|\leq\frac{1}{2}$.
\end{proof}

We remark that the lower bound is given by the first two terms of the asymptotic expansion of $R_1(z)$ which we prove in general for $d\geq 2$ (see Theorem \ref{R1-d-sphere}), plus a term of negative sign of (lower) order $\sqrt{z}$, and the upper bound is given by the same two terms, plus a term of   positive sign of (lower) order $\sqrt{z}$, as it is expected. The results of Proposition \ref{prop:R1S2bounds} and Theorem \ref{bound-R-1-2-improved} are illustrated in Figure \ref{f1}.

\begin{figure}[ht]
\centering
\includegraphics[width=0.7\textwidth]{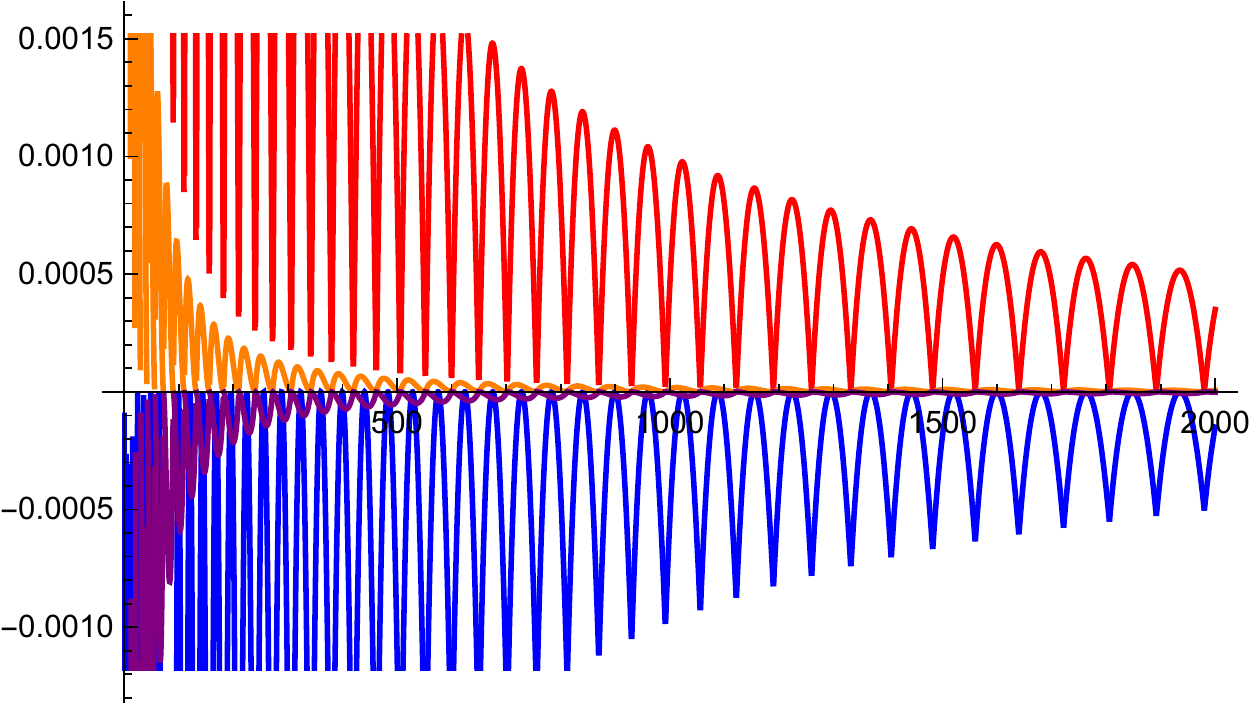}
\caption{In blue and red, the ratio (minus $1$) of $R_1$ and the upper and the lower bounds of Proposition \ref{prop:R1S2bounds}, respectively; in purple and orange the same quantity, but with the improved upper and lower bounds of Theorem \ref{bound-R-1-2-improved}}\label{f1}
\end{figure}

\subsection{The hemisphere $\mathbb S^2_+$}\label{2hemi}
Now we pass to consider the case of the two dimensional hemisphere $\mathbb S^2_+$. 
Since the hemisphere $\mathbb{S}^2_+$ has a non-empty boundary, to consider problems on $\mathbb{S}^2_+$ it is necessary to impose boundary conditions. We will consider both the cases of Dirichlet and Neumann  boundary conditions imposed on the equator, that is problems \eqref{dir_domain} and \eqref{neu_domain} with $M^d = \mathbb{S}^2$ and $\Omega =\mathbb{S}^2_+$. 

\subsubsection{Dirichlet Laplacian} 
We start with the case of Dirichlet boundary condition imposed on the equator. As is well-known, the energy levels of the Dirichlet Laplacian on $\mathbb{S}^2_+$  are the same
of the Laplacian on $\mathbb{S}^2$, that is  $\lambda_{(l)}=l(l+1)$, but with corresponding multiplicities $l$, where  $l\in\mathbb{N}\setminus\{0\}$.

Since the work by B\'erard and Besson \cite{BeBe80} it is known that the eigenvalues of the Dirichlet Laplacian on 
$\mathbb{S}^2_+$ satisfy  P\'olya's conjecture. The same result, together with many more on P\'olya's-type inequalities on spheres and hemisphere, is proved by Freitas, Mao and Salvessa \cite{FrMaSa22}. First, we provide another elementary proof of P\'olya's conjecture for $\mathbb{S}^2_+$.

\begin{prop}\label{prop:PoS+}
  For all $z\geq 0$, the counting function $N^D(z)$ for the Dirichlet Laplacian on $\mathbb{S}^2_+$ satisfies the following inequality:
  \begin{equation}\label{HS-2-N-bound}
    N^D(z)\leq \frac{1}{2}\,z.
  \end{equation}
\end{prop}
\begin{proof}
 As already done in previous proofs, we set $z=w(w+1)$ with $w\geq 0$. Then we have
  \begin{equation*}
    N^D(w(w+1))-\frac{w(w+1)}{2}=\sum_{l=1}^{\lfloor w\rfloor}l\, - \frac{w(w+1)}{2}=-\,\frac{(w-\lfloor w\rfloor)(w+1+\lfloor w\rfloor)}{2}\leq 0,
  \end{equation*}
  that clearly proves the bound.
\end{proof}

Actually, we are able to prove a two-sided bound for the counting function, where the upper bound improves the result of Proposition \ref{prop:PoS+}.
\begin{thm}\label{prop:PoS+imp}
  For all $z\geq 0$, the counting function $N^D(z)$ for the Dirichlet Laplacian on $\mathbb{S}^2_+$ satisfies the following inequality:
  \begin{equation*}
 \frac{z}{2}\left(1-\left(\psi(w)+\frac{1}{2}\right)z^{-\frac{1}{2}}\right)^2-\frac{1}{8}\left(\psi(w)+\frac{1}{2}\right)z^{-\frac{1}{2}} \leq   N^D(z)\leq \frac{z}{2}\left(1-\left(\psi(w)+\frac{1}{2}\right)z^{-\frac{1}{2}}\right)^2 ,
  \end{equation*}
  where $w$ is defined by the relation $w(w+1)=z$.
\end{thm}
\begin{proof}
We prove the inequalities for $z=w(w+1)$, $w \geq 0$. We have
\begin{multline*}
N(w(w+1)) = \frac{\lfloor w \rfloor(\lfloor w \rfloor+1)}{2} 
= \frac{(w+\lfloor w \rfloor-w)(w+1+\lfloor w \rfloor-w)}{2} 
\\=\frac{w(w+1)-\left(\psi(w)+\frac{1}{2}\right)(2w+1)+\left(\psi(w)+\frac{1}{2}\right)^2}{2}.
\end{multline*}
For the upper bound it suffices to note that $2w+1 \geq 2\sqrt{w(w+1)}$ and recall the substitution $z=w(w+1)$. In fact,
\begin{multline*}
\frac{w(w+1)-\left(\psi(w)+\frac{1}{2}\right)(2w+1)+\left(\psi(w)+\frac{1}{2}\right)^2}{2}\\
\leq \frac{w(w+1)}{2} - \left(\psi(w)+\frac{1}{2}\right)\sqrt{w(w+1)}+\frac{1}{2}\left(\psi(w)+\frac{1}{2}\right)^2\\
=\frac{w(w+1)}{2}\left(1-\left(\psi(w)+\frac{1}{2}\right)(w(w+1))^{-\frac{1}{2}}\right)^2.
\end{multline*}
The lower bound can be proved in the same way by noting that $2w+1 \leq 2\sqrt{w(w+1)}+\frac{1}{4\sqrt{w(w+1)}}$.
\end{proof}
The results of  Theorem \ref{prop:PoS+imp} are illustrated in Figure \ref{f2}.
\begin{rem}
We remark that the upper bound coincides with the expression given by the leading term in Weyl's law plus the second and the third terms found in the expansion \eqref{hemisphere-d-N-of-z-three-term-asymptotics} for $N^D(z)$ in Theorem \ref{ND} for all $d\geq 2$. The lower bound coincides with the same three terms, and the further term which one can find going further in the asymptotic expansion of $N^D(z)$ (which is not difficult in the case $d=2$).
\end{rem}

\begin{figure}[ht]
     \centering
     \begin{subfigure}[b]{0.4\textwidth}
         \centering
         \includegraphics[width=\textwidth]{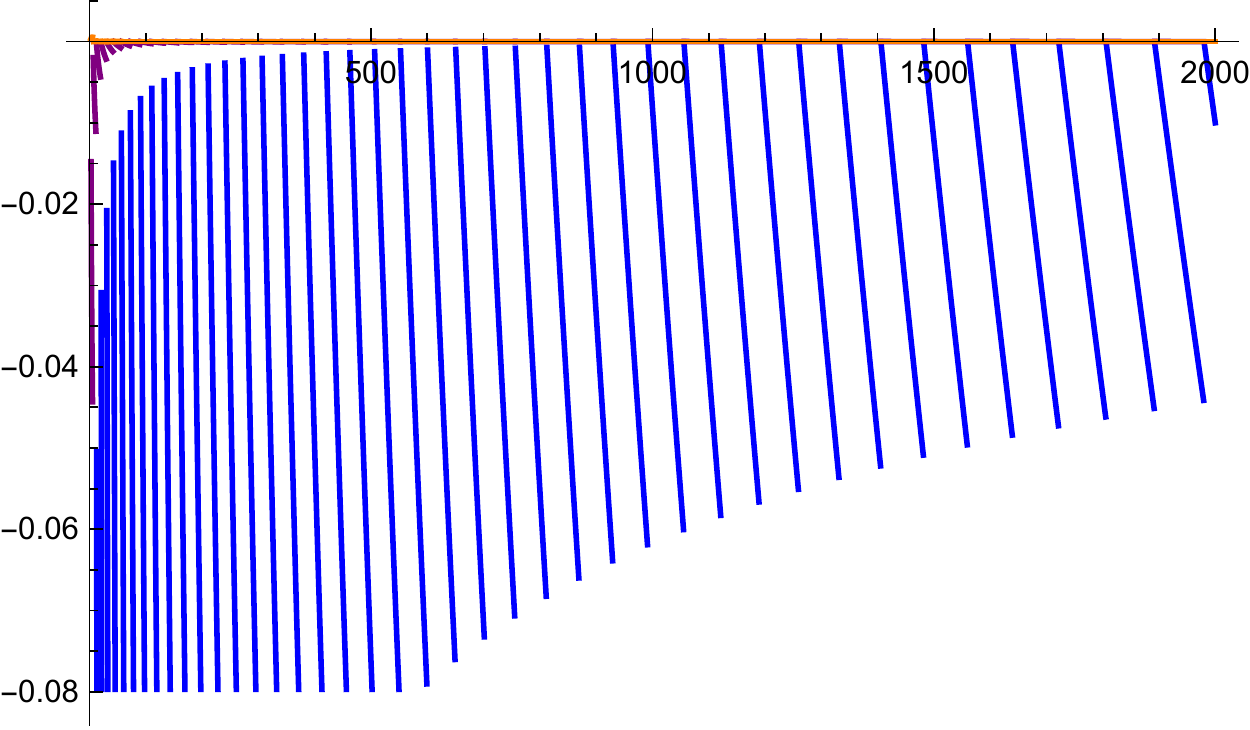}
     \end{subfigure}
     \hfill
     \begin{subfigure}[b]{0.4\textwidth}
         \centering
         \includegraphics[width=\textwidth]{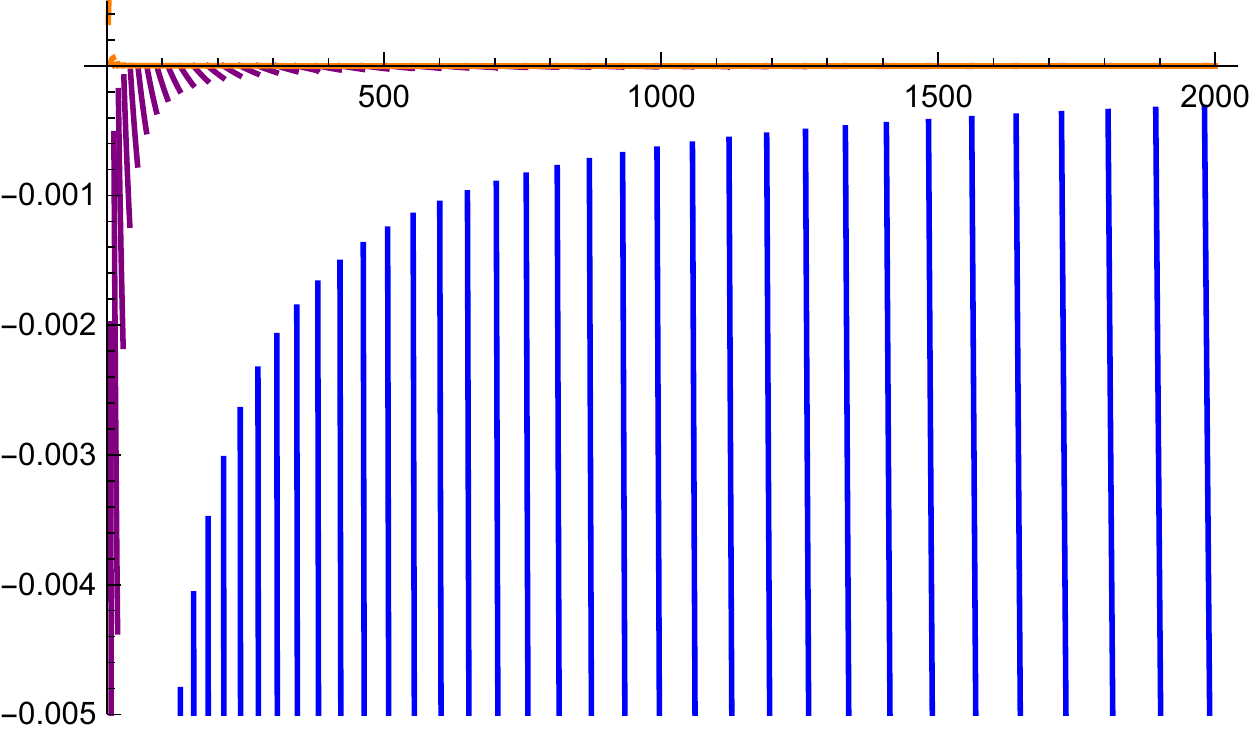}
     \end{subfigure}
     \hfill
     \begin{subfigure}[b]{0.4\textwidth}
         \centering
         \includegraphics[width=\textwidth]{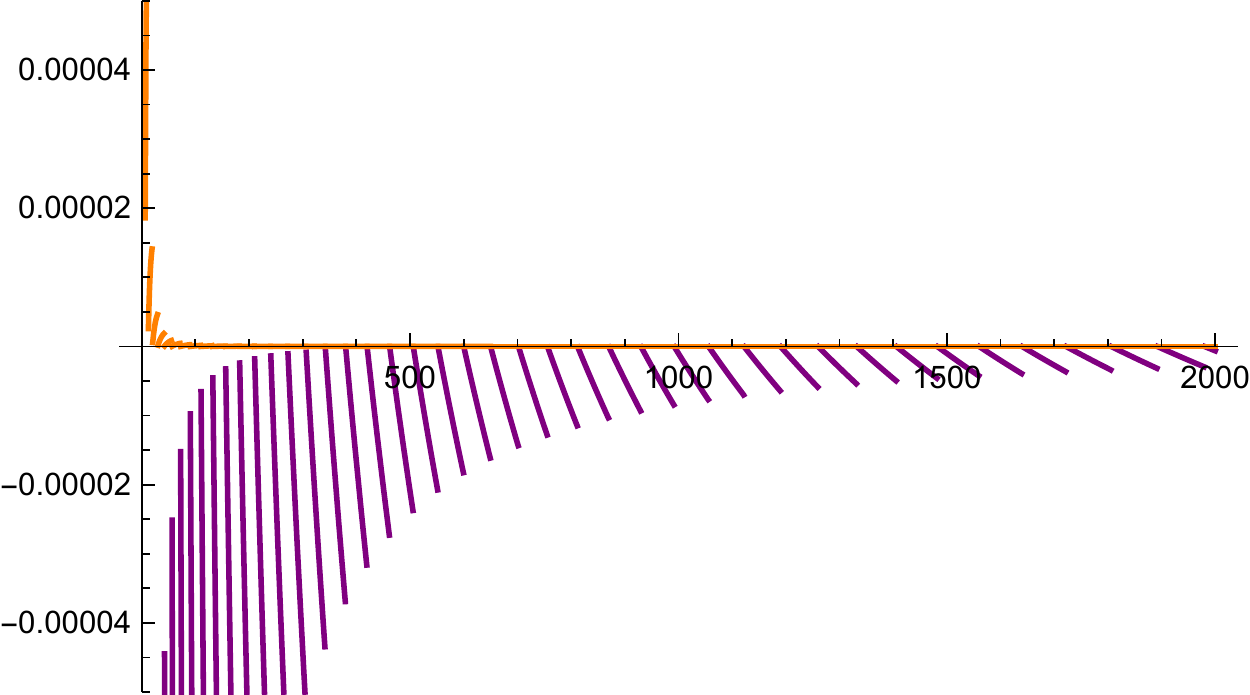}
     \end{subfigure}
     \hfill
     \begin{subfigure}[b]{0.4\textwidth}
         \centering
         \includegraphics[width=\textwidth]{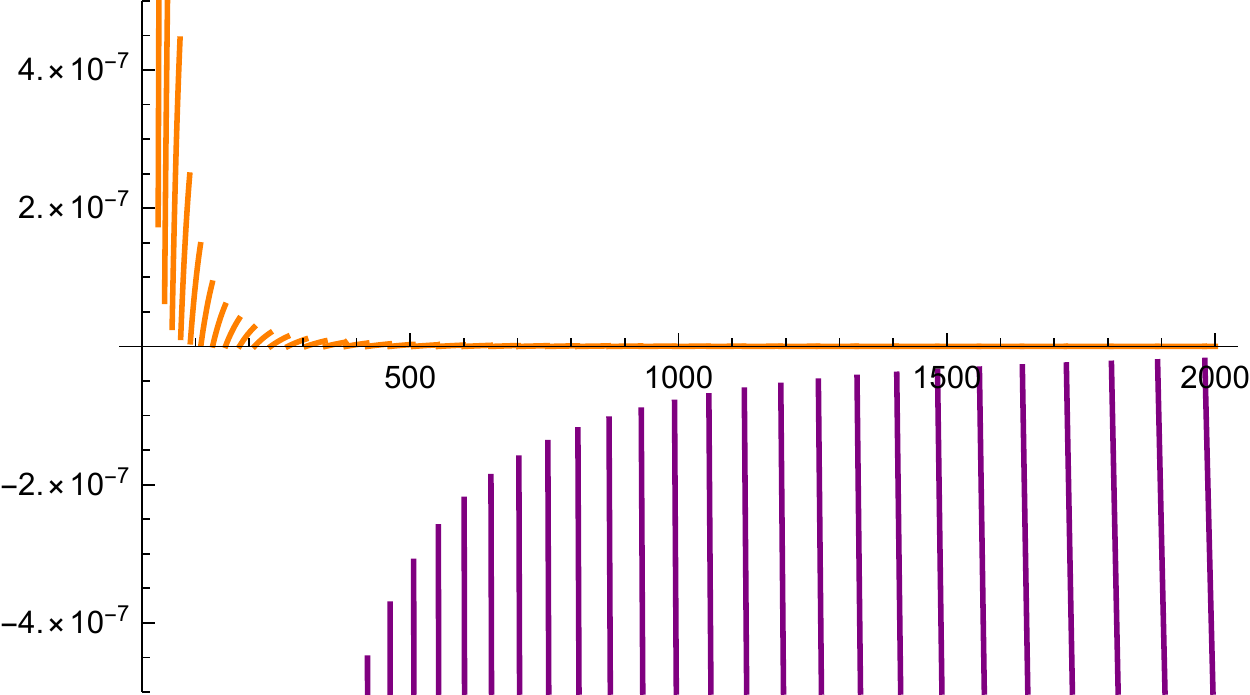}
     \end{subfigure}
        \caption{In blue, the ratio (minus $1$) of the function $N^D$ and the leading term in Weyl's law $z/2$, while in purple and orange the ratio (minus $1$) of $N^D$ and the improved upper and lower bound of Theorem \ref{prop:PoS+imp}, respectively. The four pictures represent the same ratios but at different scales.}\label{f2}
       
\end{figure}

We pass now to consider Weyl sharp upper and lower bounds for the first Riesz-mean $R_1^D$. The semiclassical expansion of $R_1^D$ reads
\begin{equation*}
  R_1^D(z) = L_{1,2}^{class}|\mathbb{S}_{+}^2|z^2-\frac{1}{4}L_{1,1}^{class}|\partial\mathbb{S}_{+}^2|z^{3/2}+O(z)
  =\frac{1}{4}\,z^2- \frac{1}{3}\,z^{3/2} +O(z) \quad \mbox{ as } z\to +\infty
\end{equation*}
where the $O(z)$ term is oscillatory and non-negative (see Theorem \ref{R1-d-hemi-N}; see \cite{FR_GR} for Euclidean domains). Note that $R_1^D$ admits a two-term ``standard'' expansion as in \eqref{two-terms generic} (the second term is a power-like function), contrarily to $N^D$  (see \eqref{hemisphere-d-N-of-z-three-term-asymptotics}).

 Note also that the leading term in Weyl's law is an upper bound for $R_1^D$ and this follows immediately from the validity of P\'olya's conjecture (Proposition \ref{prop:PoS+}).
 
 
 In the following theorem we derive upper and lower bounds for $R_1^D$ with also lower order terms (see Figure \ref{f34} for an illustration of the result).

\begin{thm}\label{prop:R1S+}
 For all $z\geq 0$ the following bounds hold for the first Riesz-mean $R_1^D$ of the Dirichlet Laplacian eigenvalues on $\mathbb{S}^2_+$:
 \begin{equation*}
  \frac{1}{4}\,z^2-\frac{1}{3}\,z\sqrt{z+\frac{1}{4}}\leq R_1^D(z)\leq   \frac{1}{4}\,z^2- \frac{1}{3}\,z\sqrt{z+\frac{1}{4}}+\frac{1}{4}\,z.
\end{equation*}
Moreover, equality in the lower   bound  occurs if and only if $z=\lambda_{(l)}$ for some $l \in \mathbb{N}\setminus\{0\}$.
\end{thm}

\begin{proof}
 As in the previous proofs, we derive the bounds for $z=w(w+1)$,  for all $w>0$. We first note that 
\begin{multline*}
     R_1^D(w(w+1))
      = \sum_{l=1}^{\lfloor w\rfloor}l(w(w+1)-l(l+1))\\
        =-\frac{1}{4}\, (w-\lfloor w\rfloor)^2(w+\lfloor w\rfloor+1)^2
       -\frac{1}{6}\,\lfloor w\rfloor(\lfloor w\rfloor+1)(2\lfloor w\rfloor+1)+\frac{1}{4}\,w^2(w+1)^2.
 \end{multline*}
For the lower bound we write $R_1^D$ as follows

\begin{multline*}
     R_1^D(w(w+1))-\frac{1}{4}\,w^2(w+1)^2+\frac{1}{6}\,w(w+1)(2w+1)\\
        =-\frac{1}{4}\, (w-\lfloor w\rfloor)^2(w+\lfloor w\rfloor+1)^2+\frac{1}{6}\bigg(w(w+1)(2w+1)- \lfloor w\rfloor(\lfloor w\rfloor+1)(2\lfloor w\rfloor+1)\bigg)\,.
\end{multline*}
We add and subtract
 $\displaystyle \frac{1}{4}\, (w-\lfloor w\rfloor)(w+\lfloor w\rfloor+1)^2$ to the right hand side of the previous equality.
First we note that
\begin{equation*}
  \begin{split}
     -\frac{1}{4}\, (w-\lfloor w\rfloor)^2(w+\lfloor w\rfloor+1)^2+\frac{1}{4}\, (w-\lfloor w\rfloor)(w+\lfloor w\rfloor+1)^2
        =\frac{1}{4}(w-\lfloor w\rfloor)(1+\lfloor w\rfloor -w)(w+\lfloor w\rfloor+1)^2.
  \end{split}
\end{equation*}
Moreover,
\begin{multline*}
     \frac{1}{6}\left(w(w+1)(2w+1)- \lfloor w\rfloor(\lfloor w\rfloor+1)(2\lfloor w\rfloor+1)\right)\\
        =\frac{1}{6}(w-\lfloor w\rfloor)\left(2\lfloor w\rfloor^2+2w^2+2w\lfloor w\rfloor+3w+3\lfloor w\rfloor+1\right)\\
       =\frac{1}{4}(w-\lfloor w\rfloor)\left(\frac{1}{3}\,(w-\lfloor w\rfloor)^2+(w+\lfloor w\rfloor)^2+2w+2\lfloor w\rfloor+\frac{2}{3}\right),
\end{multline*}
and therefore
\begin{multline*}
    -\frac{1}{4}\, (w-\lfloor w\rfloor)(w+\lfloor w\rfloor+1)^2+\frac{1}{6}\left(w(w+1)(2w+1)- \lfloor w\rfloor(\lfloor w\rfloor+1)(2\lfloor w\rfloor+1)\right)\\
        =\frac{1}{12}(w-\lfloor w\rfloor)\left((w-\lfloor w\rfloor)^2-1\right).
\end{multline*}
Combining both we get
\begin{multline*}
     R_1^D(w(w+1))-\frac{1}{4}\,w^2(w+1)^2+\frac{1}{6}\,w(w+1)(2w+1)\\
        =\frac{1}{12}\, (w-\lfloor w\rfloor)\bigg(3(1+\lfloor w\rfloor -w)(w+\lfloor w\rfloor+1)^2+(w-\lfloor w\rfloor)^2-1\bigg)\\
       =\frac{1}{12}\, (w-\lfloor w\rfloor) (1+\lfloor w\rfloor -w)\left(3(w+\lfloor w\rfloor+1)^2-(1+w-\lfloor w\rfloor)\right),
\end{multline*}
which is obviously non-negative. In particular, the right hand side of the previous equality vanishes if and only if  $w$ is a non-negative integer. Recalling the substitution $z= w(w+1)$ the statement for the lower bound is proved.

Next we pass to consider the upper bound. For the sake of simplicity from now up to the end of the proof we write 
$w= \lfloor w\rfloor + x$ where $x\in[0,1[$ denotes the fractional part of $w$. We rewrite $R_1^D$ in the following way 
\begin{multline*}
     R_1(w(w+1))-\frac{1}{4}\,w^2(w+1)^2+\frac{1}{6}\,w(w+1)(2w+1)\\
  =\frac{1}{12}\, x(1-x)\left(12\lfloor w\rfloor^2+12\lfloor w\rfloor x+12\lfloor w\rfloor+3x^2+5x+2\right)\\
     =\frac{1}{4}\,(\lfloor w\rfloor+x)(\lfloor w\rfloor+1+x)- \lfloor w\rfloor(\lfloor w\rfloor+1)\left(x-\frac{1}{2}\right)^2\\
       -\frac{x}{2}\left(1-2x(1-x)\right)\lfloor w\rfloor-\frac{x}{12}\left(3x^3+2x^2+1)\right)
     \leq \frac{1}{4}\,(\lfloor w\rfloor+x)(\lfloor w\rfloor+1+x)
\end{multline*}
which is the claimed upper bound.
\end{proof}
\begin{rem}
  In Theorem \ref{prop:R1S+} the lower bound is negative for $0\leq z\leq 2$ and since $R_1^D(z)=0$ for $0\leq z\leq 2$ one can clearly replace it by the trivial bound $0$. Moreover, the upper bound in Theorem \ref{prop:R1S+} also implies the Weyl-sharp upper bound $R_1^D(z)\leq \frac{z^2}{2}$ since  $-\frac{1}{3}z\sqrt{z+\frac{1}{4}}+\frac{1}{4}z \leq 0$ for all $z\geq 5/16$.
\end{rem}


\subsubsection{Neumann Laplacian}
Next we pass to consider the case of Neumann boundary conditions. The energy levels of the Neumann Laplacian on $\mathbb{S}^2_+$  are again the same
of the Laplacian on $\mathbb{S}^2$, that is  $\lambda_{(l)}=l(l+1)$, but with corresponding multiplicities $l+1$, where  $l\in\mathbb{N}$. 

As we have done for the Dirichlet Laplacian on $\mathbb{S}^2_+$, we show Weyl sharp upper and lower bound for the first Riesz-mean $R_1^N$ of the Neumann eigenvalues. The semiclassical expansion of $R_1^N$ is given by 
\begin{equation*}
  R_1^N(z)= L_{1,2}^{class}|\mathbb{S}_{+}^2|z^2+\frac{1}{4}L_{1,1}^{class}|\partial\mathbb{S}_{+}^2|+O(z)
  =\frac{1}{4}\,z^2+ \frac{1}{3}\,z^{3/2} +O(z) \qquad \mbox{ as } z \to+\infty,
\end{equation*}
where the $O(z)$ term is oscillatory and non-negative (see Theorem \ref{R1-d-hemi-NNeu}). We have the following two-sided bound with two sharp terms (see Figure \ref{f34} for an illustration of the result).
\begin{thm}\label{imp-S2-N-R1}
 For all $z \geq 0$ the following bounds hold for the first Riesz-mean $R_1^N$ of the Neumann Laplacian eigenvalues on $\mathbb{S}^2_+$:
 \begin{equation*}
  \frac{1}{4}\,z^2+\frac{1}{3}\,z\sqrt{z+\frac{1}{4}}\leq R_1^N(z)\leq  \frac{1}{4}\,z^2+\frac{1}{3}\,z\sqrt{z+\frac{1}{4}}+z.
\end{equation*}
Moreover, equality in the lower   bound  occurs if and only if $z=\lambda_{(l)}$ for some $l \in \mathbb{N}$.
\end{thm}
\begin{proof}
  As in the previous proofs, we derive the bounds for $z=w(w+1)$,  for all $w \geq 0$.
  By a direct computation one can verify that 
  \begin{equation}\label{HS-2-R-1-Neumann-eq-1}
    \begin{split}
     R_1^N(w(w+1))&-\frac{1}{4}\,w^2(w+1)^2-\frac{1}{6}\,w(w+1)(2w+1)\\
       & =(1-4\psi^2(w))\, \left(w^2+\frac{3-2\psi(w)}{8}\,w+\frac{(7-6\psi(w))(3-2\psi(w))}{192}\right).
  \end{split}
  \end{equation}
The right hand side of equation \eqref{HS-2-R-1-Neumann-eq-1} is clearly non-negative, and thus the lower bound holds. Moreover, the right hand side vanishes for those $w>0$ such that $\psi(w)=\pm 1/2$, that is when $w$ is a natural number and hence $w(w+1)$ is an energy level.

 For the upper bound we may assume $w\geq 1$ since for $w\leq 1$
  we have $R_1(w(w+1))=w(w+1)$ and the upper bound is trivially verified. Now we note that $\frac{3-2\psi(w)}{8}\leq \frac{1}{2}$ and  
$\frac{(7-6\psi)(3-2\psi)}{192}\leq \frac{5}{14}\leq \frac{w}{2}$. Hence 
\[
(1-4\psi^2(w))\, \left(w^2+\frac{3-2\psi(w)}{8}\,w+\frac{(7-6\psi(w))(3-2\psi(w))}{192}\right) \leq w^2+\frac{w}{2}+\frac{w}{2}=w(w+1).
\]
That is the right hand side of equation \eqref{HS-2-R-1-Neumann-eq-1} is bounded above by
$w(w+1)$, which concludes the proof.
\end{proof}


\begin{figure}[ht]
     \centering
     \begin{subfigure}[b]{0.45\textwidth}
         \centering
         \includegraphics[width=\textwidth]{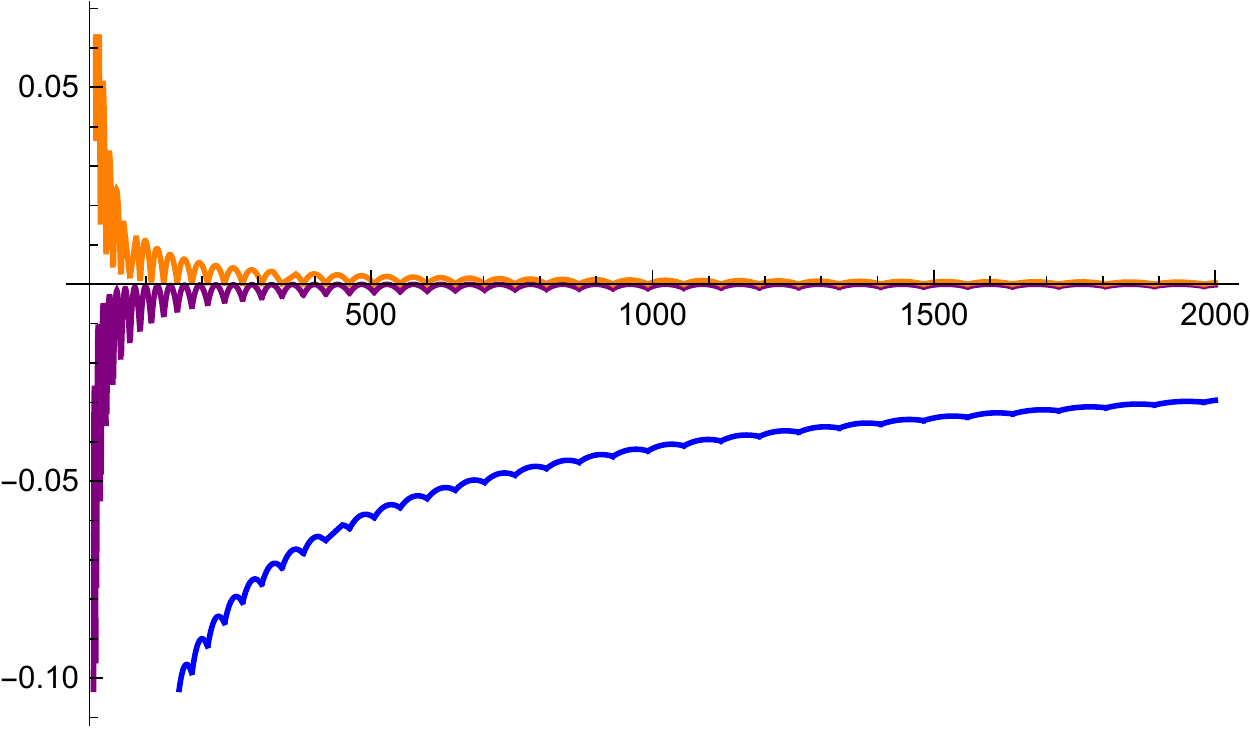}
     \end{subfigure}
     \hfill
     \begin{subfigure}[b]{0.45\textwidth}
         \centering
         \includegraphics[width=\textwidth]{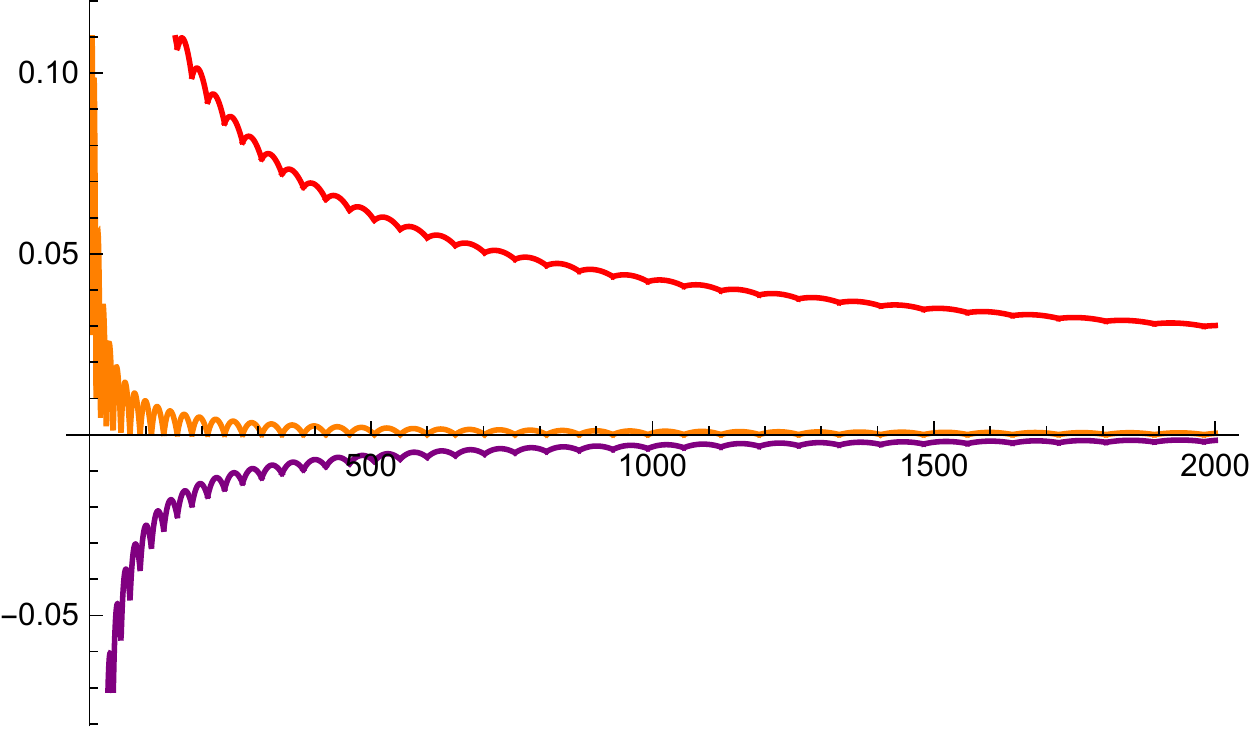}
     \end{subfigure}
    
        \caption{Left: in blue, the ratio (minus $1$) of $R_1^D$ with the leading term in Weyl's law $z^2/4$, in purple and orange the ratio (minus $1$) of $R_1^D$ and the improved upper and lower bounds of Theorem \ref{prop:R1S+}, respectively. Right: in red, the ratio (minus $1$) of $R_1^N$ with the leading term in Weyl's law $z^2/4$, in purple and orange the ratio (minus $1$) of $R_1^N$ and the improved upper and lower bounds of Theorem \ref{imp-S2-N-R1}, respectively.}\label{f34}
       \end{figure}

\subsubsection{Domains in $\mathbb{S}_{+}^2$}

Here we derive upper bounds in the spirit of Berezin-Li-Yau \cite{berezin,LY} for the first Riesz-mean of the eigenvalues of the Dirichlet Laplacian on domains $\Omega$ contained in  the hemisphere $\mathbb{S}_{+}^2$. We recall that, in terms of Riesz-means, Berezin-Li-Yau bounds amount to saying that the leading term in Weyl's law is an upper bound for $R_1^D$.


We recall that we denote by  
\[
0<\lambda_1(\Omega) <\lambda_2(\Omega)\leq  \ldots \leq 
\lambda_j(\Omega) \leq \ldots \nearrow +\infty
\]
 the eigenvalues of the Dirichlet Laplacian on $\Omega$, each repeated in accordance with its multiplicity, and by $\{u_j\}_{j\geq 1}$ the corresponding $L^2(\Omega)$-orthonormal sequence of eigenfunctions.
 
 
We note that Strichartz \cite{stri}  considered Berezin-Li-Yau-type inequalities for the eigenvalues of the Laplacian on domains of the sphere. For general domains of the sphere, an analogue of the Berezin-Li-Yau inequality cannot hold since  
the first eigenvalue on the whole sphere is zero, and actually, the opposite bound holds, see Proposition \ref{prop:R1S2bounds}. In \cite{stri} the author proves a Berezin-Li-Yau-type inequality with a first sharp term and with a lower order correction. The basic estimate relies on an observation of Colin de Verdière and Gallot, already contained in \cite{Ga}. We also refer to \cite{EHIS,ilyin_laptev} for equivalent approaches leading to analogous results.


The principal idea of this section in order to recover an analogue of the Berezin-Li-Yau inequality and improve the result of \cite{ilyin_laptev,stri}   is to get rid of the eigenvalue $0$ of the Laplacian on $\mathbb{S}^2$
by considering only domains $\Omega\subset \mathbb{S}_{+}^2$. Note that the Berezin-Li-Yau inequality (without correction) cannot hold in general as long as the domain is not contained in a hemisphere, even if it is close to it. In fact, for a spherical cap of radius $\pi/2+\epsilon$ in $\mathbb S^2$, P\'olya's conjecture already fails for $\lambda_1$.

We are now ready state our first  result.
Its proof is based on  the {\it averaged variational principle} (i.e., Theorem \ref{thm:AVP}) with the use of the eigenfunction $u_j$ extended to zero outside $\Omega$
as test functions for the Dirichlet Laplacian eigenvalues  on $\mathbb{S}_{+}^2$. For the sake of clarity we have postponed two technical lemmas used in the proof to the end of this subsection.

\begin{thm}\label{thm:blyS2+}
  Let $\Omega$ be a domain in  $\mathbb{S}_{+}^2$. Then for all $z \geq 0$ the following inequality for the first Riesz-mean $R_1^D$ of the eigenvalues of the Dirichlet Laplacian on $\Omega$ holds:
  \begin{equation*}
    R_1^D(z)= \sum_{j\geq 1} \left(z-\lambda_j(\Omega)\right)_{+}\leq \frac{1}{8\pi}|\Omega|z^2.
\end{equation*}

\end{thm}
\begin{proof}
The eigenfunctions of the Dirichlet Laplacian on $\mathbb{S}_{+}^2$  associated with the energy level
$\lambda_{(l)}$ are the spherical harmonics $Y_l^{-l-1+2h}$, where $h=1,\ldots, l$. We note one more time here that the index $l$ is
not the numbering of the eigenvalues counting multiplicities, as it is $j$ for $\lambda_j(\Omega)$, but it is the numbering of the energy levels.  

Let $z\geq 0$. We apply Theorem
\ref{thm:AVP} with $\mathcal{H}=L^2(\mathbb S^2_+)$, $H = -\Delta$, $\mathcal Q=H^1_0(\mathbb S^2_+)$, $Q(u,u) = \int_{\mathbb S^2_+} |\nabla u|^2$, $\mathfrak{M}=\mathbb{N}\setminus\{0\}$, $\mathfrak{M}_0=\{j\in\mathbb{N}\setminus\{0\}: z-\lambda_j(\Omega)\geq 0\}$, $f_p = u_j$.
We get
\begin{equation*}
  \sum_{l\geq 1}\sum_{h=1}^{l}\left(z-l(l+1)\right)_{+}
  \sum_{j \geq 1}\bigg|\int_{\Omega}Y_l^{-l-1+2h}u_j\, dS\bigg|^2\geq \sum_{j\geq 1} \left(z-\lambda_j(\Omega)\right)_{+},
\end{equation*}
which, since $\{u_j\}_k$ form a complete set in $L^2(\Omega)$, implies that 
\begin{equation*}
  \sum_{l\geq 1}\sum_{h=1}^{l}\big(z-l(l+1)\big)_{+}\int_{\Omega}|Y_l^{-l-1+2h}|^2\, dS\geq \sum_{j\geq 1} \left(z-\lambda_j(\Omega)\right)_{+}.
\end{equation*}
Now we note that 
\begin{equation*}
  \sum_{h=1}^{l}|Y_l^{-l-1+2h}|^2\leq \sum_{m=-l}^{l}|Y_l^{m}|^2=\frac{2l+1}{4\pi}
\end{equation*}
by the addition formula for spherical harmonics (see \cite[Chapter 2,\S H]{folland}, see also \cite{gine}), and, accordingly we get
\begin{equation}\label{AVP-Dirichlet-S-2-hemisphere-3}
 \frac{|\Omega|}{4\pi} \sum_{l\geq 1}(2l+1)\big(z-l(l+1)\big)_{+}\geq \sum_{j\geq 1} \left(z-\lambda_j(\Omega)\right)_{+}.
\end{equation}
Then the statement follows by Lemma \ref{lem:blys1}  below.
\end{proof}

\begin{rem}
  Alternatively, in order to recover the above Berezin-Li-Yau bound, we could have followed  a more physical idea.  Let $\Omega\subset \mathbb{S}_{+}^2$. Let $\widetilde{\Omega}$ be the set
  obtained by reflecting $\Omega$ at the equator.  Then we consider the Dirichlet eigenvalues of $\Omega\cup \widetilde{\Omega}$
  but restricted to functions antisymmetric with respect to the equator (and the same for the entire sphere).
  In this space the eigenvalues on $\Omega\cup \widetilde{\Omega}$ are of course $\lambda_j(\Omega)$ with the same multiplicities.  Finally we apply the averaged variational principle of Theorem \ref{thm:AVP} as above.
\end{rem}
\begin{rem}
  As already pointed out,  Theorem \ref{thm:blyS2+} cannot hold for large domains $\Omega$ approaching the entire sphere for which the reversed inequality of the theorem holds (see  Proposition \ref{prop:R1S2bounds}).
\end{rem}
We  also prove  another upper bound for $R_1^D$ containing lower order terms which improves Theorem \ref{thm:blyS2+} when $z>1$.

\begin{thm}\label{thm:blyS2+imp}
  Let $\Omega$ be a domain in  $\mathbb{S}_{+}^2$. Then for all $z\geq 0$ the following inequality for the first Riesz-mean $R_1^D$ of the eigenvalues of the Dirichlet Laplacian on $\Omega$ holds:
  \begin{equation*}
    R_1^D(z)= \sum_{j\geq 1} \left(z-\lambda_j(\Omega)\right)_{+}\leq \frac{1}{8\pi}|\Omega|\left(z-\frac{1}{2}\right)^2.
\end{equation*}

\end{thm}
\begin{proof}
 The proof can be performed following the same lines of that of Theorem \ref{thm:blyS2+} together with the use of Lemma \ref{lem:blys2} instead of Lemma \ref{lem:blys1}.
\end{proof}

We conclude with the two technical lemmas we used to prove the previous results.
\begin{lemma}\label{lem:blys1}
  For all $z\geq 0$ the following inequality holds: 
\begin{equation*}
 \sum_{l\geq 1}(2l+1)\big(z-l(l+1)\big)_{+}\leq \frac{z^2}{2}.
\end{equation*}
\end{lemma}

\begin{proof}
The proof can be performed by direct computations.
\end{proof}
It is possible to improve the previous lemma adding lower order terms (see \cite[Lemma 3.2]{ilyin_laptev}).
\begin{lemma}\label{lem:blys2}
  For all $z\geq0$ the following inequality holds:
\begin{equation*}
 \sum_{l\geq 1}(2l+1)\big(z-l(l+1)\big)_{+}\leq \frac{1}{2}\left(z-\frac{1}{2}\right)^2.
\end{equation*}
\end{lemma}
\begin{proof}
  We prove the inequality for $z=w(w+1)$, $w\geq 1$.  For the sake of simplicity we write $w=\lfloor w \rfloor +x$ where $x\in[0,1[$ denotes the fractional part of $w$. Then
 \begin{multline*}
\sum_{l=1}^{\lfloor w\rfloor}(2l+1)((\lfloor w\rfloor+x)(\lfloor w\rfloor+x+1)-l(l+1))_+-\frac{1}{2}\left((\lfloor w\rfloor+x)(\lfloor w\rfloor+x+1)-\frac{1}{2}\right)^2\\=-\frac{1}{8}\left(\lfloor w\rfloor(4x-2)+2x(x+1)-1\right)^2\leq 0.
\end{multline*}

\end{proof}

\section{The $d$-dimensional sphere $\mathbb S^d$ and the hemisphere $\mathbb S^d_+$}\label{sec:dsphere}

The general case $d\geq 3$ presents a few peculiar features: for example, P\'olya's conjecture does not hold for $\mathbb S^d_+$ as shown in \cite{FrMaSa22}. Actually, it should be remarked that the two-dimensional case is the special case. In what follows we shall treat $d\geq 2$.

\subsection{The sphere $\mathbb S^d$}\label{dsphere}

We recall that the eigenvalues of the Laplacian on $\mathbb S^d$  are given as energy levels by $\lambda_{(l)}=l(l+d-1)$ with multiplicities $m_{l, d} = H_{l, d}-H_{l-2,d}$ where
\begin{equation*}
  H_{l,d}=\binom{d+l}{l},
\end{equation*}
see e.g., \cite{berger}.

The first result of this subsection is a two-term expansion for $R_1(z)$. Note that the second term has a sign, though it contains an oscillating part. We stress the fact that, since the sphere has no boundary, the classical second term in $z^{\frac d 2}$ is not present in the expansion, and the term we obtain may be regarded as a ``third term'' in the semiclassical expansion. This asymptotic expansion improves the result in Strichartz \cite[Theorem  3.3 p. 168]{stri} on eigenvalue means, where the $\liminf$ and the $\limsup$ of the second term was given (see Figure \ref{f4} for an illustration of the result). Moreover, in Theorems \ref{lo-shift} and \ref{thm:ubR1dshift} below we shall prove lower and and upper bounds on $R_1$ corresponding to the lower and upper envelope of the asymptotic expansion (via the estimates $0\leq \frac{1}{4}-\psi^2\leq \frac{1}{4}$ for the fluctuation function, see also Remark \ref{rem_stri} below).

\begin{thm}\label{R1-d-sphere}
   As $z$ tends to infinity we have
   the following asymptotic expansion for the first Riesz  mean $R_1$   on $\mathbb{S}^d$:
  \begin{equation}\label{S-d-R-1-two-term-asymptotics}
   \frac{R_1(z)}{L_{1,d}^{class}|\mathbb{S}^d|\,z^{\frac{d}{2}+1}} = 1+ \frac{d(d+2)}{12}\left(d-2 + 6\left(\frac{1}{4}-\psi^2(w)\right)\right) z^{-1} +o(z^{-1}),
  \end{equation}
  where $w$ is defined by the relation $w(w+d-1)=z$.
\end{thm}
\begin{proof}

We first prove that
\begin{equation}\label{S-d-R-1}
  R_1(z)=\sum_{l=0}^{L}m_{l,d}\big(z-l(l+d-1)\big)
  =\frac{(2L+d)\Gamma(L+d)}{(d+2)\Gamma(L+1)\Gamma(d+1)}\,(-dL(L+d)+(d+2)z),
\end{equation}
where $L=\lfloor w\rfloor$. Note that \eqref{S-d-R-1} can be deduced by \cite[Appendix A]{ilyin_laptev} (see also \cite[Theorem 3.2]{stri}). We prove it here for the reader's convenience. We start by recalling the following well-known formula (see e.g. \cite{GR}, 0.15, p.3)
\begin{equation}\label{appendix-sum-binomial-coefficients}
  \sum_{k=0}^{m}\binom{n+k}{n}=\binom{n+m+1}{n+1}
\end{equation}
and we write
\begin{equation}\label{B1}
\sum_{l=0}^Lm_{l,d}z=z\left(\sum_{l=0}^L\binom{d+l}{l}-\sum_{l=0}^L\binom{d+l-2}{l-2}\right).
\end{equation}
Thus
\begin{multline}\label{B2}
\sum_{l=0}^Lm_{l,d}l(l+d-1)
=\sum_{l=0}^L\binom{d+l}{l}(l(l-1)+l d)-\sum_{l=0}^L\binom{d+l-2}{l-2}((l+d)(l+d-1)-d(l+d-1))\\
=\sum_{l=0}^Ld (d+1)\binom{d+l}{l-1}+\sum_{l=0}^L(d+1)(d+2)\binom{d+l}{l-2}
-\sum_{l=0}^L(d+1)(d+2)\binom{d+l}{l-2}+\sum_{l=0}^Ld(d+1)\binom{d+l-1}{l-2}.
\end{multline}
Using \eqref{appendix-sum-binomial-coefficients} to compute \eqref{B1} and \eqref{B2}, we get \eqref{S-d-R-1}.

Since
\begin{equation*}
  L_{1,d}^{class}|\mathbb{S}^d|=\frac{4}{(d+2)\Gamma(d+1)},
\end{equation*}
we get
\begin{equation}\label{S-d-R-1-ratio-eq-1}
  \begin{split}
     \frac{R_1(z)}{L_{1,d}^{class}|\mathbb{S}^d|z^{1+d/2}} &= \frac{(2L+d)\Gamma(L+d)}{4\Gamma(L+1)z^{1+d/2}}\,(-dL(L+d)+(d+2)z).\\
  \end{split}
\end{equation}
This proves \eqref{S-d-R-1}.
We apply the asymptotic expansions given in Appendix \ref{app_A} for the Gamma function (Lemma \ref{lem_3}), and the Taylor expansions of
the quadratic polynomials $P_{a,b}(x)=1+ax+bx^2$ (Lemma \ref{lem_2}) with $x=1/L$.
First we note that
\begin{multline*}
  \frac{(2L+d)\Gamma(L+d)}{4\Gamma(L+1)}
  =L^d\,\left(\frac{(2+dx)e^{-d}(1+dx)^{1/x}(1+dx)^{d-1/2}
  P_{\frac{1}{12},\frac{1}{288}}(\frac{x}{1+dx})}{4\,P_{\frac{1}{12},\frac{1}{288}}(x)}+O(x^3)\right)\\
  =L^d\,\left(\frac{1}{2}(1+\frac{dx}{2})P_{-\frac{d^2}{2},\frac{8d^3+3d^4}{24}}(x)P_{\frac{(2d-1)d}{2},\frac{d^2(2d-1)(2d-3)}{8}}(x)
   \cdot \frac{P_{\frac{1}{12},\frac{1}{288}}(\frac{x}{1+dx})}{P_{\frac{1}{12},\frac{1}{288}}(x)}+O(x^3)\right).
\end{multline*}
We rewrite the last term in \eqref{S-d-R-1-ratio-eq-1} as  $L^2(-d(d+x)+(d+2)\,\frac{z}{L^2})$. Since $z=w(w+d-1)$ and $L$ is the integer part of $w$,
we write $z$ using the fluctuation function, which is then given by
$$
\psi(w)=w-L-\frac{1}{2},
$$
as $z=(L+\psi(w)+\frac{1}{2})(L+d+\psi(w)-\frac{1}{2})$, and therefore
\begin{equation*}
  \frac{z}{L^2}=\left(1+\left(\psi(w)+\frac{1}{2}\right)x\right)\left(1+\left(d+\psi(w)-\frac{1}{2}\right)x\right).
\end{equation*}
According to \eqref{1-over-P-a-b-of-x} and \eqref{P-a-b-of-x-combined} of Lemma \ref{lem_2} we have
\begin{equation*}
  P_{\frac{1}{12},\frac{1}{288}}\left(\frac{x}{1+dx}\right)=P_{\frac{1}{12},\frac{1}{288}-\frac{d}{12}}(x)+O(x^3)
\end{equation*}
and
\begin{equation*}
  \frac{1}{P_{\frac{1}{12},\frac{1}{288}}(x)}=P_{-\frac{1}{12},\frac{1}{288}}(x)+O(x^3).
\end{equation*}
We compute the coefficients $A,B,C$ of the product in \eqref{S-d-R-1-ratio-eq-1} according to \eqref{products-of-P-a-b-of-x} of Lemma \ref{lem_2} as follows:
\begin{equation*}
  A=\frac{d}{2}-\frac{d^2}{2}+\frac{(2d-1)d}{2}+\frac{1}{12}-\frac{1}{12}=\frac{d^2}{2},
\end{equation*}
\begin{equation*}
  B=\frac{8d^3+3d^4}{24}+\frac{d^2(2d-1)(2d-3)}{8}+\frac{1}{288}-\frac{d}{12}+\frac{1}{288}=\frac{d(5d-2)(3d^2-2d+1)}{24}+\frac{1}{144},
\end{equation*}
and
\begin{equation*}
  C=\frac{d^4}{8}-\frac{d^2}{8}-\frac{d^4}{8}-\frac{(2d-1)^2d^2}{8}-\frac{1}{144}=-\,\frac{d^2(2d^2-2d+1)}{4}-\frac{1}{144}.
\end{equation*}
Hence the coefficient of $x^2$ is given by
\begin{equation*}
  B+C=\frac{d(d-1)(3d^2-d+2)}{24}.
\end{equation*}
Therefore we have
\begin{multline*}
     \frac{R_1(z)}{L_{1,d}^{class}|\mathbb{S}^d|z^{1+d/2}}= \frac{1}{2}\left(P_{A,B+C}(x)+O(x^3)\right)\\
     \cdot\left((d+2)\left(1+\left(\psi(w)+\frac{1}{2}\right)x\right)\left(1+\left(d+\psi(w)-\frac{1}{2}\right)x\right)-d(1+dx)\right)\cdot \left(\frac{L^2}{z}\right)^{1+d/2}\\
      =\left(P_{A,B+C}(x)+O(x^3)\right)\\
 \cdot\left(1+((d+2)\psi(w)+d)x+\left(\psi(w)+\frac{1}{2}\right)\left(1+\frac{d}{2}\right)\left(\psi(w)-\frac{1}{2}+d\right)x^2\right)\cdot \left(\frac{L^2}{z}\right)^{1+d/2}.
\end{multline*}
Next we expand
\begin{equation*}
  \left(\frac{L^2}{z}\right)^{1+d/2}=\left(1+\left(\psi(w)+\frac{1}{2}\right)x\right)\left(1+\left(d+\psi(w)-\frac{1}{2}\right)x\right)^{-1-d/2}.
\end{equation*}
Combining all terms as above we finally get
\begin{equation*}
  \frac{R_1(z)}{L_{1,d}^{class}|\mathbb{S}^d|z^{1+d/2}}=1+\frac{d(d+2)}{12}\bigg(d-2 + 6\left(\frac{1}{4}-\psi^2(w)\right)\bigg)x^2 +O(x^3).
\end{equation*}
Since $x^2=z^{-1} +O(z^{-3/2})$, the theorem is proven.
\end{proof}

\begin{figure}[ht]
\centering
\includegraphics[width=0.7\textwidth]{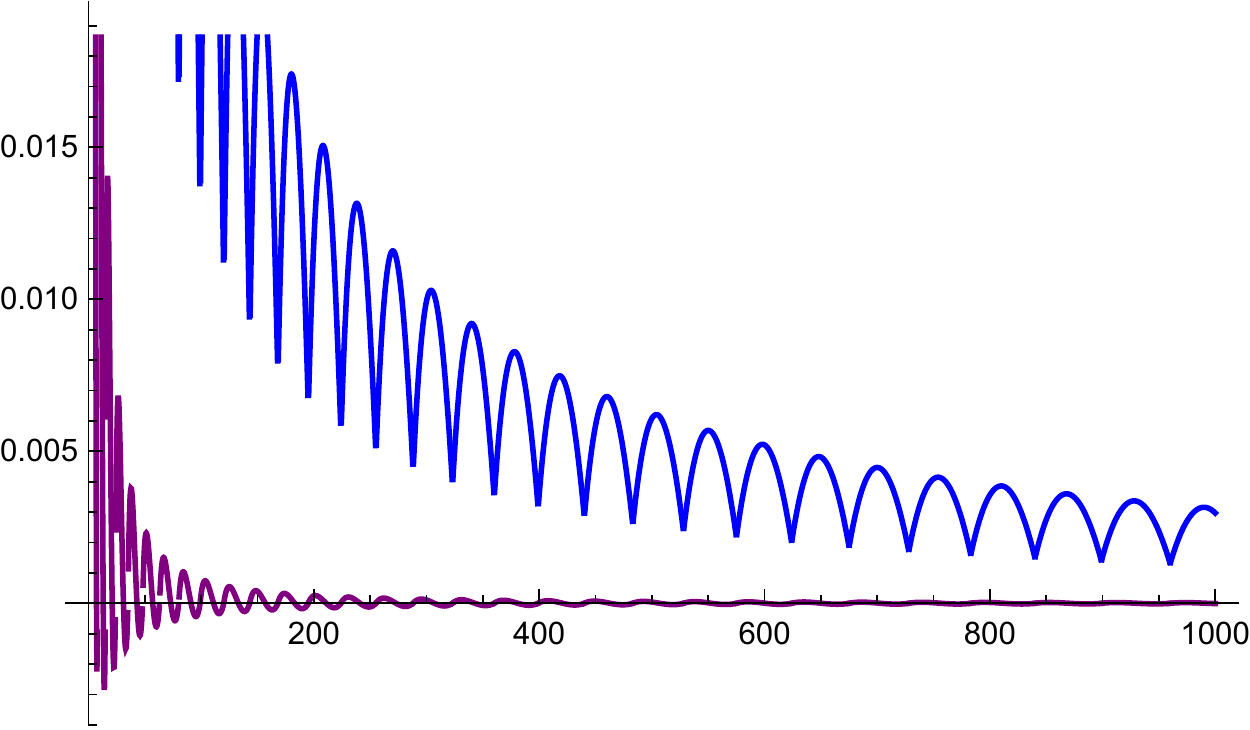}
\caption{In blue, ratio (minus $1$) of $R_1$ and the leading term in Weyl's law; in purple the ratio (minus $1$) of $R_1$ and the two-term expression of Theorem \ref{R1-d-sphere}. Here $d=3$.}\label{f4}
\end{figure}

In view of \eqref{S-d-R-1-two-term-asymptotics}, we now derive a Weyl sharp lower bound for $R_1(z)$. To do so, we first prove the following

\begin{lemma}\label{prop-R1-l-b}
The ratio
$$
\frac{R_1(z)}{L_{1,d}^{class}|\mathbb S^d|z^{1+\frac{d}{2}}}
$$
has a unique critical point which is a strict maximum in each interval $[\lambda_{(l)},\lambda_{(l+1)}]$.
\end{lemma}

\begin{proof}
As before we write $z=w(w+d-1)$ and put $L:=\lfloor w\rfloor$.
Since
\begin{equation*}
\begin{split}
  R_1(w(w+d-1))&=\sum_{l=0}^{L}m_{l,d}\big(w(w+d-1)-l(l+d-1)\big)\\
  &=\frac{(2L+d)\Gamma(L+d)}{(d+2)\Gamma(L+1)\Gamma(d+1)}\,(-dL^2-d^2L+(d+2)w(w+d-1))
\end{split}
\end{equation*}
we get
\begin{equation*}
     \frac{R_1(w(w+d-1))}{L_{1,d}^{class}|\mathbb{S}^d|w^{1+d/2}(w+d-1)^{1+d/2}}  
       = \frac{(2L+d)\Gamma(L+d)}{4\Gamma(L+1)w^{1+d/2}(w+d-1)^{1+d/2}}\,(-dL^2-d^2L+(d+2)w(w+d-1)).
\end{equation*}
The aim is then to show that in each interval $[L,L+1]$, $L\geq 1$, the ratio of $R_1$ and the leading term in Weyl's law has a unique maximum $w$.
When $L=0$ the ratio is a strictly decreasing function and singular at $w=0$.  For this we fix $L$ and put $w=w(x)=L+x$ with $x\in[0,1[$ the fractional part of $w$. Note that $\lambda_{(L)}=w(0)(w(0)+d-1)$, $\lambda_{(L+1)}=w(1)(w(1)+d-1)$. Therefore
\begin{equation}\label{S-d-R-1-ratio-eq-2}
\frac{R_1(w(w+d-1))}{L_{1,d}^{class}|\mathbb{S}^d|w^{1+d/2}(w+d-1)^{1+d/2}} 
       = \frac{(L+d/2)\Gamma(L+d)}{\Gamma(L+1)(L+x)^{1+d/2}(L+x+d-1)^{1+d/2}}\,A(x),\\
\end{equation}
with
\begin{equation}\label{A-def-eq}
  A(x)=\left(\frac{d}{2}+1\right)x^2+(d+2)\left(\frac{d-1}{2}+L\right)x+\left(L-1+\frac{d}{2}\right)L.
\end{equation}
We consider the logarithm of the quantities in equation \eqref{S-d-R-1-ratio-eq-2} that is
\begin{equation*}
  Q(x):=\log \left(\frac{(L+d/2)\Gamma(L+d)}{\Gamma(L+1)}\right)-\left(\frac{d}{2}+1\right)\log\big((L+x)(L+x+d-1)\big) +\log A(x)\,.
\end{equation*}
An easy computation shows that
\begin{equation*}
  Q'(x)=\frac{A'(x)}{A(x)}-\left(\frac{d}{2}+1\right)\frac{2L+d-1+2x}{(L+x)(L+d-1+x)}
\end{equation*}
and
\begin{equation*}
  Q''(x)=\frac{A''(x)}{A(x)}-\frac{A'(x)^2}{A(x)^2}
  -\frac{d+2}{(L+x)(L+d-1+x)}+\left(\frac{d}{2}+1\right)\frac{(2L+d-1+2x)^2}{(L+x)^2(L+d-1+x)^2}\,.
\end{equation*}
We compute the right derivatives of $A$ and $Q$ at $x=0$ and the left derivatives at $x=1$. By \eqref{A-def-eq} we have $A(0)=(L-1+\frac{d}{2})L>0$,
$A'(0)=(\frac{d}{2}+1)(2L+d-1)>0$
and therefore $ \displaystyle Q'(0)=\frac{d(d+2)(2L+d-1)}{2L(2L+d-2)(L+d-1)}>0$. Similarly, $A(1)=(L+1+\frac{d}{2})(L+d)>0$,
$A'(1)=(\frac{d}{2}+1)(2L+d+1)>0$ and therefore
$ \displaystyle Q'(1)=-\,\frac{d(d+2)(2L+d+1)}{2(L+1)(2L+d+2)(L+d+2)}<0$.
Therefore $Q(x)$ has (at least) one critical point in $]L,L+1[$. We show that it is unique. Suppose $Q'(x_0)=0$.
Since $A'(x)=(\frac{d}{2}+1)(2L+d-1+2x)>0$ the condition $Q'(x_0)=0$  is also equivalent to
\begin{equation*}
  A(x_0)=(L+x_0)(L+d-1+x_0)\,.
\end{equation*}
Then, since $A''(x)=d+2$ and $Q'(x_0)=0$ we get
\begin{equation*}
  Q''(x_0)=\frac{d+2}{A(x_0)}-\frac{d}{d+2}\frac{A'(x_0)^2}{A(x_0)^2}-\frac{d+2}{A(x_0)}
  =-\frac{d}{d+2}\frac{A'(x_0)^2}{A(x_0)^2}<0\,.
\end{equation*}
Hence any critical point is a strict local maximum and therefore $Q(x)$ has exactly one critical point in each interval $]L,L+1[$. This concludes the proof.
\end{proof}

Now we are ready to prove a Weyl-sharp lower bound for $R_1(z)$. This result can be found in Ilyin and Laptev \cite{ilyin_laptev}, however Lemma \ref{prop-R1-l-b} gives a new insight on the typical behavior of Riesz-means $R_1$ and opens the door to the improvement which we present in Theorem \ref{lo-shift}  below, confirming thereby the study of the asymptotics for eigenvalue sums done by Strichartz \cite{stri} (see Remark \ref{rem_stri} below).

\begin{thm}\label{R-1-lo-thm-sphere}
  For all $z\geq 0$ the following lower bound for the first Riesz-mean $R_1$ on $\mathbb{S}^d$ holds:
  \begin{equation}\label{S-d-R-1-lower-bound}
    R_1(z)\geq L_{1,d}^{class}|\mathbb{S}^d|\,z^{\frac{d}{2}+1}
  \end{equation}
\end{thm}
\begin{proof}
  From Lemma \ref{prop-R1-l-b}  we deduce that it is sufficient to prove the bound for each $z=\lambda_{(l)}$, $l\in\mathbb N$. Since the bound trivially holds for $\lambda_{(0)}=0$
  we consider $R_1(\lambda_{(l+1)})$. According to  \eqref{S-d-R-1-ratio-eq-1} we have
  \begin{equation}\label{S-d-R-1-Lambda}
    \frac{R_1(\lambda_{(l+1)})}{L_{1,d}^{class}|\mathbb{S}^d|\lambda_{(l+1)}^{1+d/2}}
    =\frac{(l+d/2)(l+1+d/2)\Gamma(l+d+1)}{\Gamma(l+1)(l+1)^{1+d/2}(l+d)^{1+d/2}}.
  \end{equation}
  We rewrite $(l+d/2)(l+1+d/2)=(l+1)(l+d)+\frac{d(d-2)}{4}$. Since $\frac{d(d-2)}{4}\geq 0$, we therefore have the lower bound
  \begin{equation*}
    \frac{R_1(\lambda_{(l+1)})}{L_{1,d}^{class}|\mathbb{S}^d|\lambda_{(l+1)}^{1+d/2}}
   \geq\frac{\Gamma(l+d+1)}{\Gamma(l+1)(l+1)^{d/2}(l+d)^{d/2}}.
  \end{equation*}
  Since
  \begin{equation*}
  \frac{\Gamma(l+d+1)}{\Gamma(l+1)}=\prod_{j=1}^{d}(l+j)
  =\left(\prod_{j=1}^{d}(l+j)(l+d+1-j)\right)^{1/2}
  =\left(\prod_{j=1}^{d}\left((l+1)(l+d)+(j-1)(d-j)\right)\right)^{1/2},
  \end{equation*}
  we finally obtain
  \begin{equation*}
    \frac{R_1(\lambda_{(l+1)})}{L_{1,d}^{class}|\mathbb{S}^d|\lambda_{(l+1)}^{1+d/2}}
   \geq\left(\prod_{j=1}^{d}\left(1+\frac{(j-1)(d-j)}{(l+1)(l+d)}\right)\right)^{1/2}\geq 1.
  \end{equation*}
\end{proof}
We note that, taking into account the term $d(d-2)/4$ in the proof of the above theorem (which we have dropped at the beginning of the estimate) we get the following estimates for $R_1(z)$ when $z=\lambda_{(l+1)}$, improving the result of \cite{ilyin_laptev}:
\begin{cor}\label{cor_lower}
  For all $l\geq 0$ and $d\geq 2$:
  \begin{equation*}
    R_1(\lambda_{(l+1)})\geq L_{1,d}^{class}|\mathbb{S}^d|\,\lambda_{(l+1)}^{\frac{d}{2}+1}\left(1+\frac{d(d-2)(d+2)}{12\lambda_{(l+1)}}\right).
  \end{equation*}
\end{cor}

\begin{proof}
 Since for $d=2$ the inequality has already been shown, we assume $d\geq 3$. We start from \eqref{S-d-R-1-Lambda}
and we rewrite $(l+d/2)(l+1+d/2)=(l+1)(l+d)+\frac{d(d-2)}{4}$. Then
  \begin{equation*}
     \frac{R_1(\lambda_{(l+1)})}{L_{1,d}^{class}|\mathbb{S}^d|\lambda_{(l+1)}^{1+d/2}}
    =\frac{\Gamma(l+d+1)}{\Gamma(l+1)(l+1)^{d/2}(l+d)^{d/2}}\left(1+\frac{d(d-2)}{4\lambda_{(l+1)}}\right).
  \end{equation*}
  Writing as before 
  \begin{equation*}
    \begin{split}
  \frac{\Gamma(l+d+1)}{\Gamma(l+1)} 
  &=\left(\prod_{j=1}^{d}\left((l+1)(l+d)+(j-1)(d-j)\right)\right)^{1/2},
  \end{split}
  \end{equation*}
  we finally obtain
  \begin{equation*}
     \frac{R_1(\lambda_{(l+1)})}{L_{1,d}^{class}|\mathbb{S}^d|\lambda_{(l+1)}^{1+d/2}}
    =\left(\prod_{j=1}^{d}\left(1+\frac{(j-1)(d-j)}{\lambda_{(l+1)}}\right)\right)^{1/2}\left(1+\frac{d(d-2)}{4\lambda_{(l+1)}}\right).
  \end{equation*}
  We consider the function $f(x)$ defined for $x\geq 0$ by
  \begin{equation}\label{f-auxiliary}
    f(x)= \left(1+\frac{d(d-2)}{4}\,x\right)^2\left(\prod_{j=1}^{d}\left(1+(j-1)(d-j)x\right)\right)-\left(1+\frac{d(d-2)(d+2)}{12}\,x\right)^2\,.
  \end{equation}
  We have $f(0)=0$. We will show $f'(0)=0$, $f''(0)>0$. Since obviously $f'''(x)\geq 0$ this implies $f(x)\geq 0$ for $x\geq 0$ and hence the claim.
  We note $f(x)=(1+Bx)^2P(x)-(1+Ax)^2$ where $\displaystyle P(x)=\prod_{j=1}^{d}(1+a_jx)$ denotes the polynomial given by the product. The coefficients $a_j,A,B$ are easily identified by \eqref{f-auxiliary}. We have
  \begin{equation*}
    P'(x)=P(x)\sum_{j=1}^{d}\frac{a_j}{1+a_jx},\quad P''(x)=P(x)\left(\sum_{j=1}^{d}\frac{a_j}{1+a_jx}\right)^2-P(x)\sum_{j=1}^{d}\frac{a_j^2}{(1+a_jx)^2}.
  \end{equation*}
  Hence
  \begin{equation*}
    f'(x)=2B(1+Bx)P(x)+(1+Bx)^2P'(x)-2A(1+Ax),
  \end{equation*}
  \begin{equation*}
    f''(x)=2B^2P(x)+4B(1+Bx)P'(x)+ (1+Bx)^2P''(x)-2A^2\,.
  \end{equation*}
  First of all, we see that
  \begin{equation*}
    f'(0)=2B+\sum_{j=1}^{d}a_j-2A=0.
  \end{equation*}
  The coefficient $A=\frac{d(d-2)(d+2)}{12}$ is indeed determined by this condition. Finally,
  \begin{multline*}
       f''(0)  =2B^2+4B \sum_{j=1}^{d}a_j+\left(\sum_{j=1}^{d}a_j\right)^2-\sum_{j=1}^{d}a_j^2-2A^2\\
          =2B^2+8B(A-B)+4(A-B)^2-2A^2-\sum_{j=1}^{d}a_j^2
          =2A^2-2B^2-\sum_{j=1}^{d}a_j^2.
  \end{multline*}
  Together with
  \begin{equation*}
    \sum_{j=1}^{d}a_j^2=\frac{d(d-1)(d-2)(d^2-2d+2)}{30}
  \end{equation*}
  we get
  \begin{equation*}
    f''(0)=\frac{d(d-1)(d-2)(d+1)(d+2)(5d-12)}{360}
  \end{equation*}
  which is non-negative for positive integers $d$ proving the assertion.
\end{proof}
From Corollary \ref{cor_lower} and a careful inspection of the proof of Lemma \ref{prop-R1-l-b} and Theorem \ref{R-1-lo-thm-sphere} we deduce the following improvement of \eqref{S-d-R-1-lower-bound}, which is optimal in a suitable sense, as we will explain in Remark \ref{optimal_lo} below (see also Figure \ref{f5}).
\begin{thm}\label{lo-shift}
For all $z\geq 0$ the following lower bound for the first Riesz-mean $R_1$ on $\mathbb{S}^d$ holds:
\begin{equation}\label{S-d-refined-Ries-mean-lower-bound}
    R_1(z)\geq L_{1,d}^{class}|\mathbb{S}^d|\,z^{\frac{d}{2}+1}\left(1+\frac{d(d-2)(d+2)}{12z}\right).
  \end{equation}
  \end{thm}
  \begin{proof}
   The proof follows the same lines as the proof of Lemma \ref{prop-R1-l-b},  showing that the ratio of the right-hand side and left-hand side of \eqref{S-d-refined-Ries-mean-lower-bound} as a function of $z$ has exactly one local maximum in each interval $[L(L+d-1),(L+1)(L+d)]$. Then we conclude by Corollary \ref{cor_lower}.

\end{proof}

We turn our attention to upper bounds for $R_1(z)$. The upper bound contains a shift term, which is again optimal in a suitable sense (see Remark \ref{optimal_up} below, see also Figure \ref{f5}). 

\begin{thm}\label{thm:ubR1dshift}
  For all $z\geq 0$ the following upper bound for the first Riesz-mean $R_1$ on $\mathbb{S}^d$ holds:
  \begin{equation}\label{S-d-R-1-bounds}
    R_1(z)\leq L_{1,d}^{class}|\mathbb{S}^d|\,(z+z_d)^{\frac{d}{2}+1}
  \end{equation}
  with
 \begin{equation}\label{S-d-R-1-upperbound shift}
  z_d=\frac{(2d-1)d}{12}.
  \end{equation}
\end{thm}
\begin{proof}
Again, let us set $z=w(w+d-1)$ and $L=\lfloor w\rfloor$. Let $b\geq 0$. We analyze the quantity
\begin{multline}\label{S-d-R-1-ratio-shift-eq-1}
     \frac{R_1(w(w+d-1))}{L_{1,d}^{class}|\mathbb{S}^d|\left(w(w+d-1)+b\right)^{1+d/2}} \\ 
       = \frac{(2L+d)\Gamma(L+d)}{4\Gamma(L+1)\left(w(w+d-1)+b\right)^{1+d/2}}\,\left(-dL(L+d)+(d+2)w(w+d-1)\right).
\end{multline}
We show that in each interval $[L,L+1]$, $L\geq 1$, the ratio in \eqref{S-d-R-1-ratio-shift-eq-1} has a unique maximum.
When $L=0$ the ratio is a strictly decreasing function and singular at $w=0$ if $b=0$.  For this we fix $L$ and put $w=L+x$ with $x\in[0,1[$ the fractional part of $w$. Note that $\lambda_{(L)}=w(0)(w(0)+d-1)$, $\lambda_{(L+1)}=w(1)(w(1)+d-1)$. Therefore
\begin{equation}\label{S-d-R-1-ratio-shift-eq-2}
     \frac{R_1(w(w+d-1))}{L_{1,d}^{class}|\mathbb{S}^d|\big(w(w+d-1)+b\big)^{1+d/2}}  
       = \frac{(L+d/2)\Gamma(L+d)}{\Gamma(L+1)\big((L+x)(L+x+d-1)+b\big)^{1+d/2}}\,A(x)
\end{equation}
with
\begin{equation*}
  A(x)=\left(\frac{d}{2}+1\right)(x+L)(x+L+d-1)-\frac{Ld(L+d)}{2}.
\end{equation*}
We also define
\begin{equation*}
  \rho(x):=(L+x)(L+x+d-1)+b.
\end{equation*}
Then $A(x)=\left(\frac{d}{2}+1\right)\rho(x)-\frac{Ld(L+d)}{2}-\left(\frac{d+2}{2}\right)b$, and \eqref{S-d-R-1-ratio-shift-eq-2} reads as follows
\begin{multline}\label{S-d-R-1-ratio-shift-eq-3}
     \frac{R_1(w(w+d-1))}{L_{1,d}^{class}|\mathbb{S}^d|\big(w(w+d-1)+b\big)^{1+d/2}}  \\
      = \frac{(L+d/2)\Gamma(L+d)}{\Gamma(L+1)}\left(\frac{d+2}{2}\,\rho(x)^{-d/2}-\left(\frac{Ld(L+d)}{2}+\frac{d+2}{2}\,b\right)\rho^{-1-d/2}\right).
\end{multline}
The right-hand side of \eqref{S-d-R-1-ratio-shift-eq-3} has a unique maximum at $\rho_b=L(L+d)+\frac{d+2}{d}\,b$.
It is easy to check that $\lambda_{(L)}+b\leq \rho_b\leq \lambda_{(L+1)}+b$ when $b\leq d^2/2$. Therefore we get the inequality
\begin{equation}\label{S-d-R-1-ratio-shift-ineq-1}
     \frac{R_1(w(w+d-1))}{L_{1,d}^{class}|\mathbb{S}^d|\big(w(w+d-1)+b\big)^{1+d/2}}  
      \leq \frac{(L+d/2)\Gamma(L+d)}{\Gamma(L+1)}\bigg(L(L+d)+\frac{d+2}{d}\,b\bigg)^{-d/2},
\end{equation}
which holds for all $w\in[L,L+1]$. Now, we note that
  \begin{equation*}
  \frac{\Gamma(L+d)}{\Gamma(L+1)}=\prod_{j=1}^{d-1}(L+j)
  =\left(\prod_{j=1}^{d-1}(L+j)(L+d-j)\right)^{1/2}
  =\left(\prod_{j=1}^{d-1}\left((L+d/2)^2-(j-d/2)^2\right)\right)^{1/2}.
  \end{equation*}
Therefore we may rewrite \eqref{S-d-R-1-ratio-shift-ineq-1} as follows:
\begin{equation}\label{S-d-R-1-ratio-shift-ineq-2}
     \frac{R_1(w(w+d-1))}{L_{1,d}^{class}|\mathbb{S}^d|\left(w(w+d-1)+b\right)^{1+d/2}}  
      \leq \left(\prod_{j=1}^{d-1} 1-\frac{(j-d/2)^2}{(L+d/2)^2}\right)^{1/2}
      \left(1+\frac{\frac{d+2}{d}\,b-\frac{d^2}{4}}{(L+d/2)^2}\right)^{-d/2}.
\end{equation}
We see that the right-hand side of \eqref{S-d-R-1-ratio-shift-ineq-2} is bounded above by $1$ if $b\geq \frac{d^3}{4(d+2)}$.
However, here we want to show a that a choice $b \leq \frac{d^3}{4(d+2)}$ also yields the upper bound $1$ in  \eqref{S-d-R-1-ratio-shift-ineq-2} .
For this we apply the arithmetic-geometric mean inequality to the product:
\begin{equation*}
    \left(\prod_{j=1}^{d-1} 1-\frac{(j-d/2)^2}{(L+d/2)^2}\right)^{1/2}  \\
      \leq \left(1- \frac{1}{d-1}\sum_{j=1}^{d-1} \frac{(j-d/2)^2}{(L+d/2)^2}\right)^{(d-1)/2}\\
      =\left(1-\frac{d(d-2)}{12(L+d/2)^2}\right)^{(d-1)/2}\,.
\end{equation*}
It is now sufficient to show that the function $f(t)$ defined by
\begin{equation*}
  f(t)=\frac{d-1}{2}\,\log\left(1-\frac{d(d-2)}{12}\,t\right)-\frac{d}{2}\,\log\left(1+\left(\frac{d+2}{d}\,b-\frac{d^2}{4}\right)\,t\right)
\end{equation*}
is decreasing for $t>0$ for $b$ suitably chosen (we will use this fact with $t=(L+d/2)^{-2}$). In particular, we want to show that this is the case for $b=z_d=\frac{(2d-1)d}{12}$ which will be the optimal choice. We easily compute
\begin{equation*}
  f'(t)=-\,\frac{\frac{d}{2}\left(\frac{d+2}{d}\,b- \frac{(d+2)(2d-1)}{12}-\frac{d-2}{12}\,(\frac{d+2}{d}\,b-\frac{d^2}{4})\,t\right)}
  {(1-\frac{d(d-2)}{12}\,t)(1+(\frac{d+2}{d}\,b-\frac{d^2}{4})\,t)}.
\end{equation*}
The best choice is obviously $b=z_d=\frac{(2d-1)d}{12}$ eliminating the constant term. With this choice
\begin{equation*}
  f'(t)=-\,\frac{d(d-1)(d-2)^2t}{24(1-\frac{d(d-2)}{12}\,t)(1+(\frac{d+2}{d}\,b-\frac{d^2}{4})\,t)}\leq 0.
\end{equation*}
The proof is now completed.
\end{proof}

\begin{rem}\label{optimal_up}
We remark that the shift $z_d$ in the upper bound \eqref{S-d-R-1-bounds} is, in a sense, optimal. We observe that for $d=2$ the upper bound coincides with the one found in Proposition \ref{prop:R1S2bounds} for $\mathbb S^2$, which we have already shown to be sharp. For $d\geq 3$, in general we cannot find $z\in[\lambda_{(l)},\lambda_{(l+1)}]$ such that the equality is attained in \eqref{S-d-R-1-bounds}. When $z\in[\lambda_{(l)},\lambda_{(l+1)}]$ one uses the explicit form of $R_1(z)$ (as in \eqref{S-d-R-1}) and considers the function $f(z)=R_1(z)-L^{class}_{1,d}|\mathbb S^d|(z+b)^{\frac{d}{2}+1}$. Computing $f'(z)$,  finding $z_0$ such that $f'(z_0)=0$, and substituting $z_0$ in $f(z)$, we find the minimum distance from $R_1(z)$ to $L^{class}_{1,d}|\mathbb S^d|(z+b)^{\frac{d}{2}+1}$, namely, $|f(z_0)|$. If we want this distance to be zero, then we must chose $b=b(l)$. If $d=2$, then $b=b(l)=1/2$ for all $l$ and this  corresponds to the optimal upper bound of Proposition \ref{prop:R1S2bounds} (see also \cite{ilyin_laptev,stri}). If $d\geq 3$, one has that in each interval $[\lambda_{(l)},\lambda_{(l+1)}]$ the optimal shift would be given by 
$$
b(l)=\frac{d}{d+2}(4^{-1/d}((d+2l)(d+l-1)!/l!)^{2/d}-l(l+d)).
$$
We highlight that $b(l)\to z_d$ as $l\to+\infty$, so in this sense the shift $z_d$ becomes sharp as $z\to\infty$.
\end{rem}

\begin{rem}\label{rem_stri}
In \cite{stri} the author estimates the liminf and limsup of the remainder of  Weyl's law for $R_1$ on $\mathbb S^d$ (Theorem 3.3). These expressions agree  with the two-term Weyl's law we have proved in Theorem \ref{R1-d-sphere} and with the corresponding upper and lower bounds. In fact, the bounds of Theorems \ref{lo-shift} and \ref{thm:ubR1dshift}  are optimal since they provide the   precise envelopes for the second term of the asymptotic expansion \eqref{S-d-R-1-two-term-asymptotics}. In particular, the upper bound is obtained when $\psi(w)=0$ (meaning $w=\lfloor w\rfloor$), and the lower bound  when $\psi^2(w)=\frac{1}{4}$ (meaning $w=\lfloor w\rfloor \pm \frac{1}{2}$). Here $w$ is defined by $w(w+d-1)=z$.
\end{rem}

\begin{rem}\label{optimal_lo}
A consequence of the upper bound \eqref{S-d-R-1-bounds} is that the average of $\lambda_j+z_d$ satisfies a Berezin-Li-Yau lower bound. One may wonder whether the lower bound \eqref{S-d-R-1-lower-bound} holds with a shift, namely, with $z$ replaced by $z+b_d$ where $b_d=\frac{d(d-2)}{6}$ which is the optimal choice (this is the liminf of Strichartz, like $z_d$ is for the upper bound). Clearly this is true for $d=2$ but it is already false for $d=3$. It is enough to observe that the corresponding inequality $R_1(z)\geq L^{class}_{1,d}|\mathbb S^d|(z+b_d)^{\frac{d}{2}+1}$ fails for $z\leq d$.
\end{rem}

\begin{figure}[ht]
    \centering
    \includegraphics[width=0.7\textwidth]{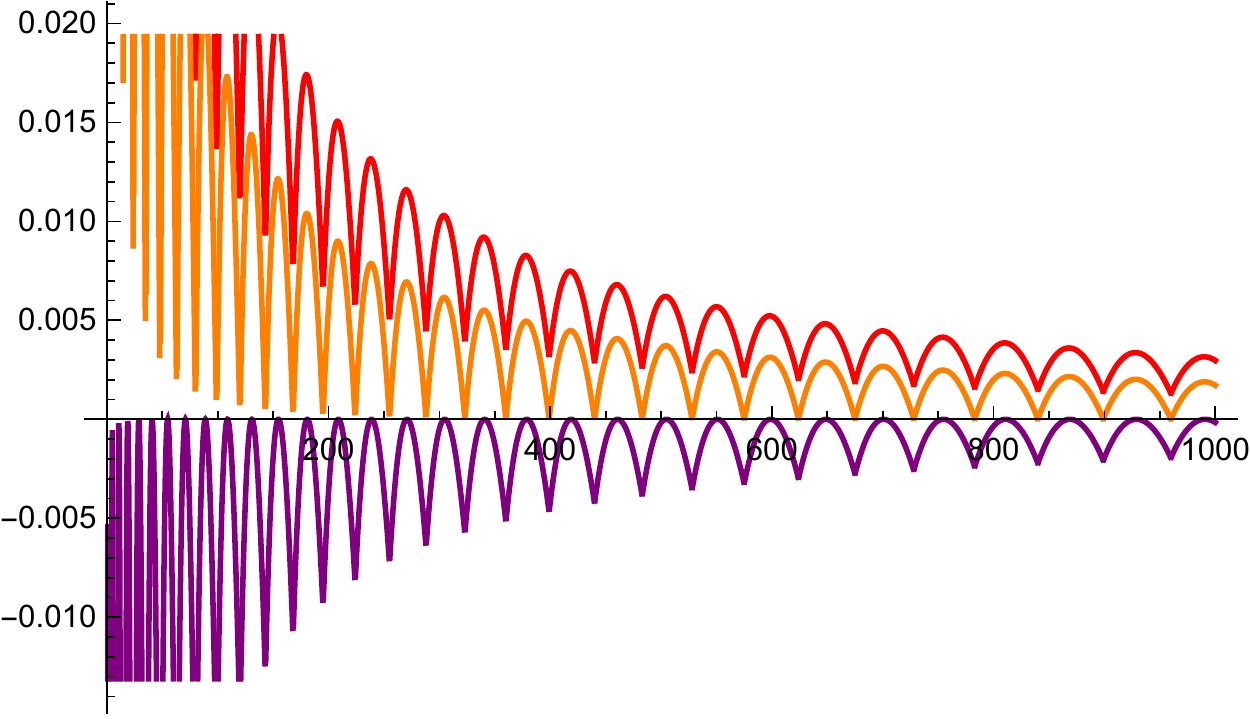}
   \caption{In red, the ratio (minus $1$) of $R_1$ and the leading term in Weyl's law; in purple and orange the ratio (minus $1$) of $R_1$ and the upper and lower bounds of Theorems \ref{thm:ubR1dshift} and  \ref{lo-shift}, respectively. Here $d=3$.}\label{f5}
\end{figure}

For the reader's convenience, we restate the results of Theorems \ref{R-1-lo-thm-sphere} and \ref{thm:ubR1dshift} in terms of inequalities for eigenvalues averages, i.e., in the form of \eqref{blyk}. For the equivalence between inequalities on Riesz means and averages we refer e.g., to \cite{HaHe2011}.

\begin{cor}\label{cor_new}
For all $k\geq 1$ the following inequalities hold:
$$
\frac{d}{d+2}\frac{4\pi^2}{\omega_d^{2/d}}\left(\frac{k}{|\mathbb S^d|}\right)^{2/d}-\frac{(2d-1)d}{12}\leq\frac{1}{k}\sum_{j=1}^k\lambda_j\leq\frac{d}{d+2}\frac{4\pi^2}{\omega_d^{2/d}}\left(\frac{k}{|\mathbb S^d|}\right)^{2/d}
$$
where $\lambda_j$ are the eigenvalues of the Laplacian on $\mathbb S^d$.
\end{cor}

We conclude this section with  a three-term asymptotic
expansion for the counting function $N(z)$. We illustrate the result in Figure \ref{f6}.

\begin{thm}\label{N-d-sphere}
As $z\to\infty$  we have the following asymptotic expansion for the counting function $N$ on $\mathbb{S}^d$:
\begin{equation}\label{N-d-sphere-eq}
\frac{N(z)}{L^{class}_{0,d}|\mathbb S^d|z^{\frac{d}{2}}}=1-\frac{L^{class}_{0,d-1}}{L^{class}_{0,d}}\frac{|\partial\mathbb S^d_+|}{|\mathbb S^d|}\psi(w)z^{-\frac{1}{2}}+\frac{d(d-1)(12\psi^2(w)+2d-1)}{24}z^{-1}+O(z^{-\frac{3}{2}}).
\end{equation}
Here $w$ is defined by $w(w+d-1)=z$ and $|\partial\mathbb S^d_+|$ denotes the measure of the boundary of the hemisphere. 
\end{thm}
\begin{proof}
The proof follows from the identity $N(z)=N^D(z)+N^N(z)$, where $N^D(z)$ and $N^N(z)$ are the counting functions for the Dirichlet and Neumann Laplacian on the hemisphere $\mathbb S^d_+$. We prove the corresponding three-term expansions in Theorem \ref{ND} and \ref{NN} in the next section.
\end{proof}

It is interesting to see that the second term is oscillatory, but it is not bounded: along suitable subsequences it behaves like $\pm z^{\frac{d}{2}-\frac{1}{2}}$. This is natural as this is the correct order of the remainder after the first term, see \cite{canzani,SV}. This also provides an interpretation of the results of \cite[Theorem F]{FrMaSa22}  for the eigenvalues on the whole sphere. Note that in \cite{SV}, the authors present a quasi-Weyl formula in the case of manifolds or domains not satisfying the geometric conditions ensuring the existence of a second term of the form $c_1z^{\frac{d-1}{2}}$. They present the explicit example of $-\Delta+\frac{(d-1)^2}{4}$ (\cite[Examples 1.2.5, 1.7.1 and 1.7.11]{SV}). The eigenvalues are given as energy levels $\left(l+\frac{d-1}{2}\right)^2$ (they are $\lambda_{(l)}+\frac{(d-1)^2}{4}$, with multiplicities $m_{l,d}$). For such eigenvalues, a two-term quasi-Weyl formula in the sense of \cite[Formula (1.7.5)]{SV} does hold, and the function $Q$ which describes the behavior of the second term in \cite[Formulas (1.7.4)-(1.7.5)]{SV} agrees with the second term of \eqref{N-d-sphere-eq}.

\begin{figure}[ht]
    \centering
    \includegraphics[width=0.7\textwidth]{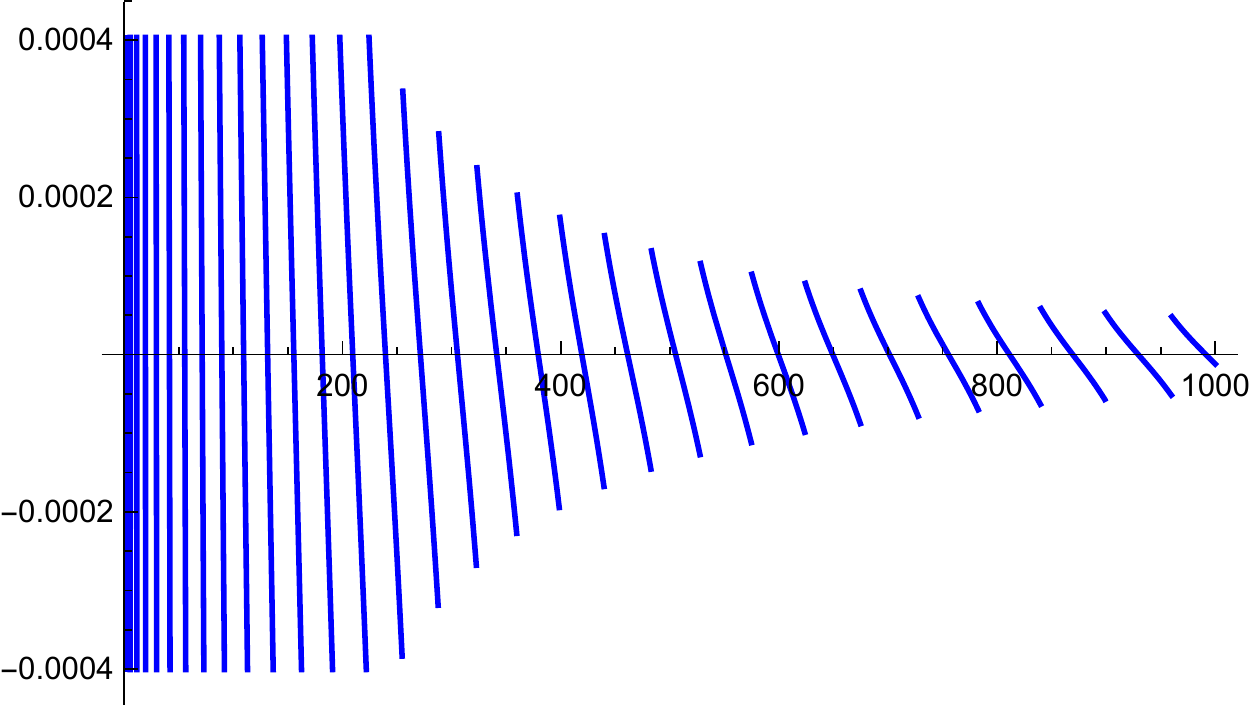}
    \caption{Ratio (minus $1$) of $N(z)$ and the three-term expansion of Theorem \ref{N-d-sphere}. Here $d=3$.}\label{f6}
\end{figure}

\subsubsection{Domains in $\mathbb S^d$}

Here we derive upper bounds in the spirit of Berezin-Li-Yau \cite{berezin,LY} for the first Riesz-mean of the eigenvalues of the Dirichlet Laplacian on domains $\Omega$ of $\mathbb S^d$. We denote by  
\[
0<\lambda_1(\Omega) <\lambda_2(\Omega)\leq  \ldots \leq 
\lambda_j(\Omega) \leq \ldots \nearrow +\infty
\]
 the eigenvalues of the Dirichlet Laplacian on $\Omega$, each repeated in accordance with its multiplicity, and by $\{u_j\}_{j\geq 1}$ the corresponding $L^2(\Omega)$-orthonormal sequence of eigenfunctions. In \cite{stri} the author establishes Berezin-Li-Yau-type inequalities for domains of $\mathbb S^2$ (see also \cite{ilyin_laptev}), and provides an expansion for $R_1^D$ in the higher dimensional case, highlighting the sharp behavior of the remainder. Here we establish a Berezin-Li-Yau inequality with a shift term in any dimension, which coincides with that proved in \cite{ilyin_laptev,stri} when $d=2$, and which contains a shift term which is asymptotically sharp when $\Omega=\mathbb S^d$, see Remark \ref{optimal_up}. In particular, the ``generalized conjecture of P\'olya'' stated in \cite[Formula (1.11)]{chengyang} holds for the sphere in a stronger form: the correct shift constant is $z_d$ and not $d^2/4$ as conjectured in \cite{chengyang} (this was clear for $d=2$ by \cite{stri}).

  The proof is based on  the {\it averaged variational principle} and is in the spirit of that of Theorem \ref{thm:blyS2+}.

\begin{thm}\label{bounds_domains_Sd}
Let $\Omega$ be a domain in $\mathbb S^d$. Then for all $z\geq 0$ the following inequality for the first Riesz-mean $R_1^D$ of the eigenvalues of the Dirichlet  Laplacian on $\Omega$ holds:
\begin{equation*}
R_1^D(z)=\sum_{j\geq 1}(z-\lambda_j(\Omega))_+\leq L_{1,d}^{class}|\Omega|(z+z_d)^{\frac{d}{2}+1}
\end{equation*}
with $z_d=\frac{(2d-1)d}{12}$. Equivalently, the following inequality holds for all $k\geq 1$:
$$
\frac{1}{k}\sum_{j=1}^k\lambda_j(\Omega)\geq\frac{d}{d+2}\frac{4\pi^2}{\omega_d^{2/d}}\left(\frac{k}{|\Omega|}\right)^{2/d}-z_d.
$$
\end{thm}
\begin{proof}
The proof follows the same lines as that of Theorem \ref{thm:blyS2+}. The eigenfunctions of the Laplacian on $\mathbb S^d$  associated with the energy level $\lambda_{(l)}=l(l+d-1)$ are the spherical harmonics $Y_l^m$, where $m=1,\ldots, m_{l,d}$. Let $z\geq 0$. We apply Theorem
\ref{thm:AVP} with $\mathcal{H}=L^2(\mathbb S^d)$, $H = -\Delta$, $\mathcal Q=H^1(\mathbb S^d)$, $Q(u,u) = \int_{\mathbb S^d} |\nabla u|^2$, $\mathfrak{M}=\mathbb{N}\setminus\{0\}$, $\mathfrak{M}_0=\{j\in\mathbb{N}\setminus\{0\}: z-\lambda_j(\Omega)\geq 0\}$, $f_p = u_j$. Following the proof of Theorem \ref{thm:blyS2+} we obtain
$$
\frac{|\Omega|}{|\mathbb S^d|}R_1(z)\geq R_1^D(z),
$$
where $R_1(z)=\sum_l m_{l,d}(z-l(l+d-1))_+$ is the first Riesz mean for the whole $\mathbb S^d$. The upper bound for $R_1^D$ follows then from Theorem \ref{thm:ubR1dshift}. By Legendre transforming the inequality for $R_1^D$ we get the inequality on the average (see Corollary \ref{cor_new}, see also \cite{HaHe2011}).
\end{proof}
\begin{rem}
Concerning the Neumann eigenvalues, it has been shown by Ilyin and Laptev \cite{ilyin_laptev} that any domain of $\mathbb S^d$ satisfies a Kr\"oger-type bound, namely, the leading term in Weyl's law is a lower bound for $R_1^N=\sum_{j\geq 1}(z-\mu_j(\Omega))_+$. This is a consequence of the fact that
$$
\sum_{j\geq 1}(z-\mu_j(\Omega))_+\geq\frac{|\Omega|}{|\mathbb S^d|}\sum_{l\geq 0} m_{l,d}(z-l(l+d-1))_+
$$
which is proved in \cite{ilyin_laptev} (or can be easily deduced as an application of the averaged variational principle as for \eqref{AVP-Dirichlet-S-2-hemisphere-3}), and from the inequality \eqref{S-d-R-1-lower-bound}. However, from our improved inequality \eqref{S-d-refined-Ries-mean-lower-bound}, we can improve the result for domains in $\mathbb S^d$. Namely, for any domain $\Omega$ in $\mathbb S^d$ we have
\begin{equation*}
R_1^N(z)=\sum_{j\geq 1}(z-\mu_j(\Omega))_+\geq L^{class}_{1,d}|\Omega|z^{\frac{d}{2}+1}\left(1+\frac{d(d-2)(d+2)}{12z}\right).
\end{equation*}

\end{rem}

\subsection{The hemisphere $\mathbb S^d_+$}\label{dhemi}

In this subsection we shall consider the eigenvalues of the Dirichlet and Neumann Laplacian on the hemisphere $\mathbb S^{d}_+$. In particular, we will compute three-term expansions for $N^D,R_1^D,N^N,R_1^N$.

\subsubsection{The Dirichlet Laplacian.} 

The eigenvalues are of the form $\lambda_{(l)}=l(l+d-1)$, $l\in\mathbb N\setminus\{0\}$, with multiplicities $m^D_{l,d}$ given by
\begin{equation}\label{hemisphere-d-multiplicity-coeff}
 m^D_{l,d}=\binom{d+l-2}{d-1}.
\end{equation}
The counting function $N^D(z)$ is easily computed. Again let $w$ be defined by the relation $z=w(w+d-1)$ and $L=\lfloor w\rfloor$ be the integer part of $w$. Then
\begin{equation}\label{hemisphere-d-N-of-z}
  N^D(z)=\sum_{l=1}^{L} m^D_{l,d}=\frac{\Gamma(d+L)}{\Gamma(L)\Gamma(d+1)}.
\end{equation}
Since the hemisphere does not satisfy the billiard condition, the counting function $N^D$ does not admit an expansion with just a power-like surface term of order $z^{\frac{d-1}{2}}$ after the leading term in Weyl's law as in \eqref{two-terms generic}. This is explained in \cite{SV}, and a major consequence is the failure of P\'olya's conjecture, as pointed out in \cite{FrMaSa22}.


We prove here a three-term asymptotic expansion for $N^D$ and we show that the second term contains oscillations, but, at any rate, it has a sign. In fact, it is non-positive. Moreover, the third term is oscillating but again, it has a sign and it is non-negative. Moreover, it is strictly positive along the sequences where the second term vanishes. This explains the failure of P\'olya's conjecture along certain sequences of eigenvalues, as pointed out in \cite[Theorem A]{FrMaSa22}. The second and third terms should be instead compared with the sharp corrections to the P\'olya's inequality proved in \cite[Theorems B, C, D]{FrMaSa22}.


\begin{thm}\label{ND}  As $z \to \infty$   we have the following asymptotic expansion for the counting function $N^D$ of the Dirichlet Laplacian eigenvalues on $\mathbb{S}^d_+$:
 \begin{equation}\label{hemisphere-d-N-of-z-three-term-asymptotics}
  \begin{split}
   \frac{N^D(z)}{L_{0,d}^{class}|\mathbb{S}_{+}^d|\,z^{\frac{d}{2}}} &= 1-
    \frac{1}{4} \frac{L_{0,d-1}^{class}}{L_{0,d}^{class}}\frac{|\partial\mathbb{S}_{+}^d|}{|\mathbb{S}_{+}^d|}(1+2\psi(w))\, z^{-1/2}\\
    &\quad\quad +\frac{d(d-1)}{2}\left(\left(\frac{1}{2}+\psi(w)\right)^2+\frac{d-2}{6}\right)z^{-1}+O(z^{-3/2})
    \end{split}
  \end{equation}
 
  or equivalently
  \begin{equation*}
  \begin{split}
   \frac{N^D(z)}{L_{0,d}^{class}|\mathbb{S}_+^d|\,z^{\frac{d}{2}}} &= 1-
    \frac{d(1+2\psi(w))}{2}\, z^{-1/2}\\
    &\quad\quad +\frac{d(d-1)}{2}\left(\left(\frac{1}{2}+\psi(w)\right)^2+\frac{d-2}{6}\right)z^{-1}+O(z^{-3/2}).
    \end{split}
  \end{equation*}
  
  Here $w$ is defined by the relation $w(w+d-1)=z$.
  
\end{thm}
\begin{proof}
As in the proof of Theorem \ref{R1-d-sphere} for $\mathbb S^d$, we set $L= \lfloor w \rfloor$ and we expand in $x=1/L$.
 From \eqref{hemisphere-d-N-of-z} we have
\begin{equation}\label{hemisphere-d-N-of-z-explicit-1}
   \frac{N^D(z)}{L_{0,d}^{class}|\mathbb{S}_+^d|\,z^{\frac{d}{2}}}
   =\frac{\Gamma(L+d)}{\Gamma(L)} \,z^{-\frac{d}{2}}.
  \end{equation}
  Moreover,
\begin{equation*}
  \frac{\Gamma(L+d)}{\Gamma(L)}=L^d\bigg(P_{-\frac{d^2}{2},\frac{8d^3+3d^4}{24}}(x)P_{\frac{(2d-1)d}{2},\frac{d^2(2d-1)(2d-3)}{8}}(x)
   \cdot \frac{P_{\frac{1}{12},\frac{1}{288}}(\frac{x}{1+dx})}{P_{\frac{1}{12},\frac{1}{288}}(x)}+O(x^3)\bigg)\,.
\end{equation*}
According to  \eqref{1-over-P-a-b-of-x} and \eqref{P-a-b-of-x-combined}  of Lemma \ref{lem_1}, we have
\begin{equation*}
  P_{\frac{1}{12},\frac{1}{288}}\left(\frac{x}{1+dx}\right)=P_{\frac{1}{12},\frac{1}{288}-\frac{d}{12}}(x)+O(x^3)
\end{equation*}
and
\begin{equation*}
  \frac{1}{P_{\frac{1}{12},\frac{1}{288}}(x)}=P_{-\frac{1}{12},\frac{1}{288}}(x)+O(x^3).
\end{equation*}
We compute the coefficients $A,B,C$ of the product in \eqref{hemisphere-d-N-of-z-explicit-1} according to \eqref{products-of-P-a-b-of-x} of Lemma \ref{lem_2} as follows:
\begin{equation*}
  A=-\frac{d^2}{2}+\frac{(2d-1)d}{2}+\frac{1}{12}-\frac{1}{12}=\frac{d(d-1)}{2},
\end{equation*}
\begin{equation*}
  B=\frac{8d^3+3d^4}{24}+\frac{d^2(2d-1)(2d-3)}{8}+\frac{1}{288}-\frac{d}{12}+\frac{1}{288}=\frac{d(5d-2)(3d^2-2d+1)}{24}+\frac{1}{144}
\end{equation*}
and
\begin{equation*}
  C=\frac{d^2(d-1)^2}{8}-\frac{d^4}{8}-\frac{(2d-1)^2d^2}{8}-\frac{1}{144}=-\,\frac{d^3(2d-1)}{4}-\frac{1}{144}.
\end{equation*}
Hence the coefficient of $x^2$ is given by
\begin{equation*}
  B+C=\frac{d(d-1)(d-2)(3d-1)}{24}.
\end{equation*}
Therefore
\begin{equation*}
  \frac{N^D(z)}{L_{0,d}^{class}|\mathbb{S}_{+}^d|\,z^{\frac{d}{2}}}
   =\left(P_{A,B+C}(x)+O(x^3)\right)\left(\frac{z}{L^2}\right)^{-d/2}.
\end{equation*}
Since $z=w(w+d-1)=(\psi(w)+\frac{1}{2}+L)(\psi(w)-\frac{1}{2}+d+L)$ we have
\begin{multline*}
\left(\frac{z}{L^2}\right)^{-d/2}=\left(1+(d+2\psi(w))x+\left(\psi(w)+\frac{1}{2}\right)\left(d+\frac{1}{2}-\psi(w)\right)x^2\right)^{-d/2}\\
=1-\frac{d}{2}(d+2\psi(w))x+\frac{d}{8}((d+1)(2\psi(w)+d)^2+(d-1)^2)x^2+O(x^3)
=:1+\alpha x+\beta x^2+O(x^3).
\end{multline*}
Now
\begin{equation*}
  P_{A,B+C}(x)P_{\alpha,\beta}(x)=P_{A',B'+C'}(x)
\end{equation*}
with
\begin{equation*}
  A'=A+\alpha,\quad B'=B+C+\beta, \quad C'=A\alpha\,.
\end{equation*}
We compute
\begin{equation*}
  A'=-\frac{d(1+2\psi(w))}{2},\quad C'=-\frac{d^2(d-1)(d+2\psi(w))}{4}
\end{equation*}
and
\begin{equation*}
  B'+C'=\frac{d}{12}\left(6(d+1)(\psi(w)+1/2)^2+(d-1)(d+6\psi(w)+1)\right)\,.
\end{equation*}
Finally, in order to reconvert $x=1/L$ into the variable $z$ we use
\begin{equation*}
  L=\sqrt{z+\left(\frac{d-1}{2}\right)^2}-\frac{d-1}{2}-\left(\psi(w)+\frac{1}{2}\right)
\end{equation*}
and therefore
\begin{equation*}
  \frac{1}{L}=z^{-\frac{1}{2}}+\frac{d+2\psi(w)}{2}\,z^{-1}+O(z^{-\frac{1}{2}})\,.
\end{equation*}
Inserting the first two terms $P_{A',B'+C'}(x)$ we obtain that the coefficient of $z^{-1}$ is given by
\begin{equation*}
  B'+C'+A'\frac{d+2\psi(w)}{2}=\frac{d(d-1)}{2}\left((\frac{1}{2}+\psi(w))^2+\frac{d-2}{6}\right),
\end{equation*}
proving the theorem.
\end{proof}

On the other hand one may expect that, similarly to the Euclidean setting, the more regular Riesz-mean $R_1^D(z)$ admits an expansion with a surface term after the leading term in Weyl's law as in \eqref{two-terms generic}, that is,
\begin{equation*}
  R_1^D(z)\sim L_{1,d}^{class}|\mathbb{S}_{+}^d|\,z^{\frac{d}{2}+1}-\frac{1}{4}\,L_{1,d-1}^{class}|\partial\mathbb{S}_{+}^d|\,z^{\frac{d}{2}+\frac{1}{2}}
\end{equation*}
as $z$ goes to infinity. Note that 
$$
L_{1,d}^{class}|\mathbb{S}_{+}^d|= \frac{2}{(d+2)\Gamma(d+1)}$$ and
\begin{equation*}
  L_{1,d-1}^{class}|\partial\mathbb{S}_{+}^d|=L_{1,d-1}^{class}|\mathbb{S}^{d-1}|=\frac{4}{(d+1)\Gamma(d)}\,.
\end{equation*}
We prove the following theorem stating that $R_1^D(z)$ has a second term of order $z^{\frac{d}{2}+\frac{1}{2}}$, and a third term, of negative sign, which includes an oscillatory part.
\begin{thm}\label{R1-d-hemi-N}
   As $z\to \infty$ we have  the following asymptotic expansion for the first Riesz-mean $R^D_1$ of the Dirichlet Laplacian eigenvalues on $\mathbb{S}^d_+$:
  \begin{multline*}
   \frac{R_1^D(z)}{L_{1,d}^{class}|\mathbb{S}_{+}^d|\,z^{\frac{d}{2}+1}} = 1-
    \frac{1}{4} \frac{L_{1,d-1}^{class}}{L_{1,d}^{class}}\frac{|\partial\mathbb{S}_{+}^d|}{|\mathbb{S}_{+}^d|}\, z^{-1/2}
    -\frac{d(d+2)}{2}\left(\frac{1}{4}-\psi^2(w)+\frac{d-2}{6}\right)z^{-1}+O(z^{-3/2})
  \end{multline*}
  or, equivalently,
  \begin{equation*}
   \frac{R_1^D(z)}{L_{1,d}^{class}|\mathbb{S}^d_+|\,z^{\frac{d}{2}+1}} = 1-
    \frac{d(d+2)}{2(d+1)}\, z^{-1/2}
     -\frac{d(d+2)}{2}\left(\frac{1}{4}-\psi^2(w)+\frac{d-2}{6}\right)z^{-1}+O(z^{-3/2})\,.
  \end{equation*}
  Here $w$ is defined by the relation $w(w+d-1)=z$.
\end{thm}
\begin{proof}
As before, we set $L= \lfloor w \rfloor$ and $x=1/L$. One easily computes the Riesz-mean as in the case of $\mathbb S^d$ (see Theorem \ref{R1-d-sphere})
\begin{equation}\label{hemisphere-d-R-1-explicit-1}
   \frac{R_1^D(z)}{L_{1,d}^{class}|\mathbb{S}_{+}^d|\,z^{\frac{d}{2}+1}}
   =\frac{d+2}{2}\,\frac{\Gamma(L+d)}{\Gamma(L)}\left(z-\frac{d(L+d)(L(d+1)+1)}{(d+1)(d+2)}\right)z^{-1-\frac{d}{2}}.
  \end{equation}
As in the proof of Theorem \ref{R1-d-sphere}, we expand
\begin{equation*}
  \frac{\Gamma(L+d)}{\Gamma(L)}=L^d\bigg(P_{-\frac{d^2}{2},\frac{8d^3+3d^4}{24}}(x)P_{\frac{(2d-1)d}{2},\frac{d^2(2d-1)(2d-3)}{8}}(x)
   \cdot \frac{P_{\frac{1}{12},\frac{1}{288}}(\frac{x}{1+dx})}{P_{\frac{1}{12},\frac{1}{288}}(x)}+O(x^3)\bigg)
\end{equation*}
as well as
\begin{equation*}
  z-\frac{d(L+d)(L(d+1)+1)}{(d+1)(d+2)}=L^2\bigg(\frac{z}{L^2}-\frac{d(1+dx)(1+\frac{x}{d+1})}{(d+2)}\bigg).
\end{equation*}
According to \eqref{1-over-P-a-b-of-x} and \eqref{P-a-b-of-x-combined} of lemma \ref{lem_1} we have
\begin{equation*}
  P_{\frac{1}{12},\frac{1}{288}}\left(\frac{x}{1+dx}\right)=P_{\frac{1}{12},\frac{1}{288}-\frac{d}{12}}(x)+O(x^3)
\end{equation*}
and
\begin{equation*}
  \frac{1}{P_{\frac{1}{12},\frac{1}{288}}(x)}=P_{-\frac{1}{12},\frac{1}{288}}(x)+O(x^3).
\end{equation*}
We compute the coefficients $A,B,C$ of the product in \eqref{hemisphere-d-R-1-explicit-1} according to \eqref{products-of-P-a-b-of-x} of Lemma \ref{lem_2} as follows:
\begin{equation*}
  A=-\frac{d^2}{2}+\frac{(2d-1)d}{2}+\frac{1}{12}-\frac{1}{12}=\frac{d(d-1)}{2},
\end{equation*}
\begin{equation*}
  B=\frac{8d^3+3d^4}{24}+\frac{d^2(2d-1)(2d-3)}{8}+\frac{1}{288}-\frac{d}{12}+\frac{1}{288}=\frac{d(5d-2)(3d^2-2d+1)}{24}+\frac{1}{144},
\end{equation*}
and
\begin{equation*}
  C=\frac{d^2(d-1)^2}{8}-\frac{d^4}{8}-\frac{(2d-1)^2d^2}{8}-\frac{1}{144}=-\,\frac{d^3(2d-1)}{4}-\frac{1}{144}.
\end{equation*}
Hence the coefficient of $x^2$ is given by
\begin{equation*}
  B+C=\frac{d(d-1)(d-2)(3d-1)}{24}.
\end{equation*}
Therefore we have
\begin{multline*}
     \frac{R_1(z)}{L_{1,d}^{class}|\mathbb{S}_{+}^d|z^{1+d/2}} = \frac{d+2}{2}\left(P_{A,B+C}(x)+O(x^3)\right)
     \cdot\left(\frac{z}{L^2}-\frac{d(1+dx)(1+\frac{x}{d+1})}{(d+2)}\right)\cdot \left(\frac{L^2}{z}\right)^{1+d/2}\\
      =\left(P_{A,B+C}(x)+O(x^3)\right)\left(\frac{L^2}{z}\right)^{1+d/2}\,.
\end{multline*}
\end{proof}

The results of Theorems \ref{ND} and \ref{R1-d-hemi-N} are illustrated in Figure \ref{f7}.

\begin{rem}
We remark that this result suggests that the leading term in Weyl's law could be an upper bound for $R_1^D(z)$ for all $d\geq 2$. In Subsection \ref{rem-li-yau} below we show that it is false for $d\geq 6$, and prove the Weyl upper bound for $d=3,4,5$ in Theorem \ref{BLY-345}.
\end{rem}

\begin{figure}[ht]
     \centering
     \begin{subfigure}[b]{0.45\textwidth}
         \centering
         \includegraphics[width=\textwidth]{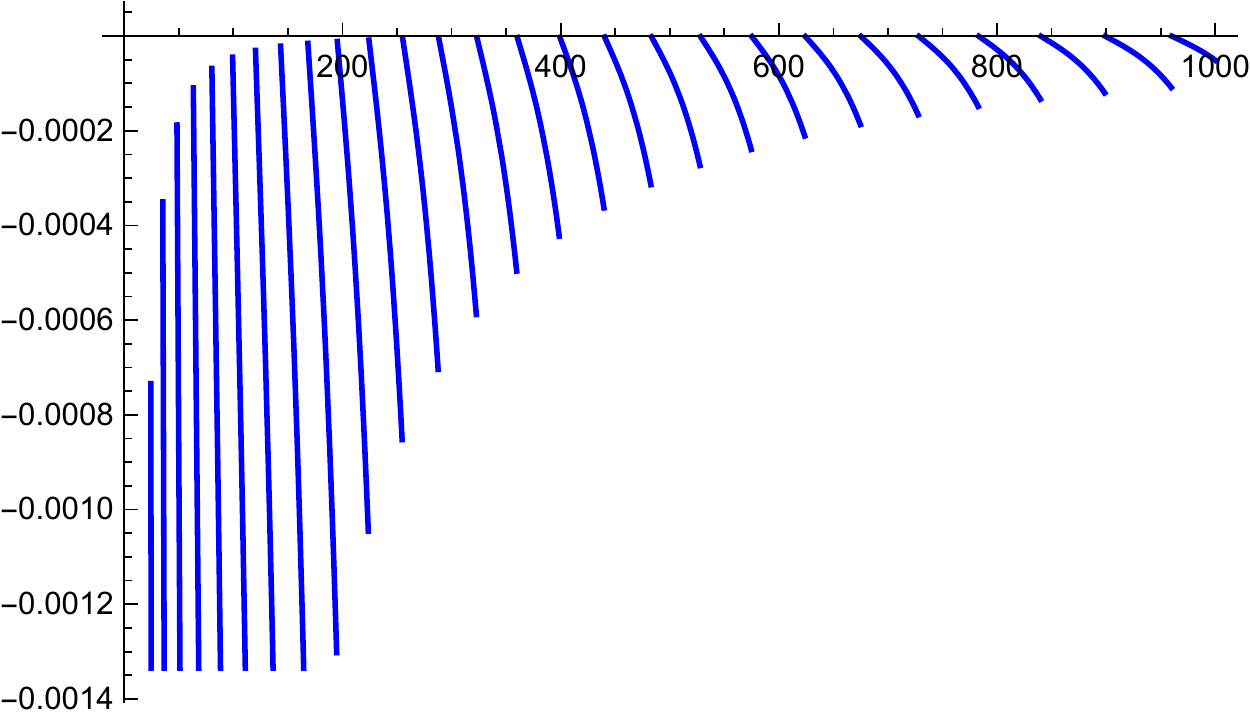}
     \end{subfigure}
     \hfill
     \begin{subfigure}[b]{0.45\textwidth}
         \centering
         \includegraphics[width=\textwidth]{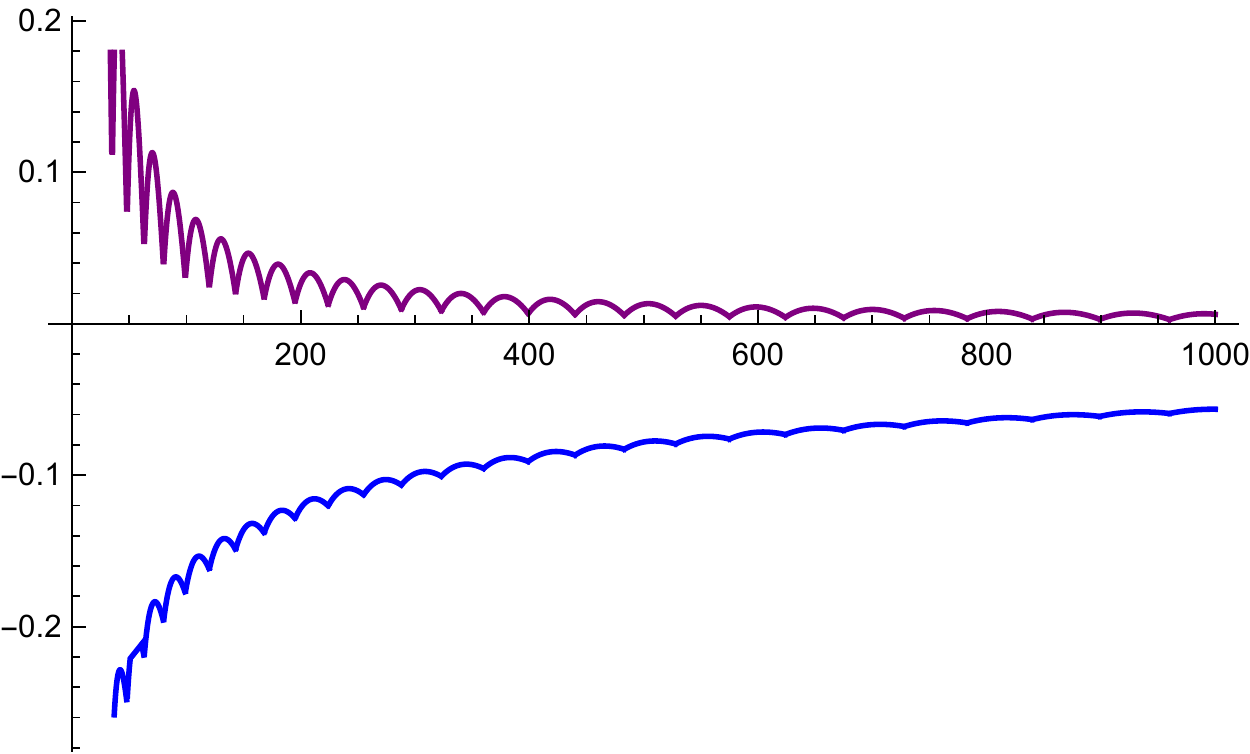}
     \end{subfigure}
     
        \caption{On the left, the ratio (minus $1$) of $N^D$ and the three-term expansion of Theorem \ref{ND}. On the right, in blue the ratio (minus $1$) of $R_1^D$ and the leading term in Weyl's law, and in purple the ratio (minus $1$) of $R_1^D$ and the three-term expansion of Theorem \ref{R1-d-hemi-N}. Here $d=3$.}\label{f7}
       
\end{figure}

\subsubsection{The Neumann Laplacian} 

When we consider the Laplacian on $\mathbb S^d_+$ with Neumann boundary conditions, the eigenvalues are of the form $\lambda_{(l)}=l(l+d-1)$, $l\in\mathbb N$, with multiplicities $m^N_{l,d}$ given by
\begin{equation}\label{hemisphere-d-multiplicity-coeff-Neumann}
 m^N_{l,d}=\binom{d+l-1}{d-1}.
\end{equation}
The counting function $N^N(z)$ is easily computed. Again let $w$ be defined by the relation $z=w(w+d-1)$ and $L$ be the integer part of $w$. Then
\begin{equation}\label{hemisphere-d-N-of-z-Neumann}
  N^N(z)=\sum_{l=1}^{L} m^N_{l, d}=\frac{\Gamma(d+L+1)}{\Gamma(L+1)\Gamma(d+1)}=\left(1+\frac{d}{L}\right)N^D(z),
\end{equation}
where $N^D(z)$ denotes the counting function for the Dirichlet Laplacian on the hemisphere.
In view of this relation the asymptotic expansion of $N^N(z)$ is easily determined form the expansion for $N^D(z)$. we have the following result.

\begin{thm}\label{NN}
   As $z\to \infty$  we have  the following asymptotic expansion for the counting function $N^N$ of the Neumann Laplacian eigenvalues on $\mathbb{S}^d_+$:
  \begin{equation*}
  \begin{split}
   \frac{N^N(z)}{L_{0,d}^{class}|\mathbb{S}_{+}^d|\,z^{\frac{d}{2}}} &= 1+
    \frac{1}{4} \frac{L_{0,d-1}^{class}}{L_{0,d}^{class}}\frac{|\partial\mathbb{S}_{+}^d|}{|\mathbb{S}_{+}^d|}(1-2\psi(w))\, z^{-1/2}\\
    &\quad\quad +\frac{d(d-1)}{2}\left(\left(\frac{1}{2}-\psi(w)\right)^2+\frac{d-2}{6}\right)z^{-1}+O(z^{-3/2})
    \end{split}
  \end{equation*}
  or, equivalently,
  \begin{equation*}
  \begin{split}
   \frac{N^N(z)}{L_{0,d}^{class}|\mathbb{S}^d|\,z^{\frac{d}{2}}} &= 1+
    \frac{d(1-2\psi(w))}{2}\, z^{-1/2}\\
    &\quad\quad +\frac{d(d-1)}{2}\left(\left(\frac{1}{2}-\psi(w)\right)^2+\frac{d-2}{6}\right)z^{-1}+O(z^{-3/2}).
    \end{split}
  \end{equation*}
 Here $w$ is defined by the relation $w(w+d-1)=z$.
\end{thm}
\begin{proof}
As usual, let $w$ be defined by $w(w+d-1)=z$ and let $L=\lfloor w\rfloor$. From \eqref{hemisphere-d-N-of-z-Neumann}, expanding $1/L$ in terms of $z$ we have
\begin{equation*}
  N^N(z)=\left(1+dz^{-1/2}+\frac{d+2\psi(w)}{2}\,z^{-1}+O(z^{-3/2})\right)N^D(z).
\end{equation*}
Hence 
\begin{multline*}
  \frac{N^N(z)}{L_{0,d}^{class}|\mathbb{S}^d|\,z^{\frac{d}{2}}}=(1+dz^{-1/2}+\frac{d+2\psi(w)}{2}\,z^{-1})\\
  \cdot\left(1-\frac{d(1+2\psi(w))}{2}\,z^{-1/2}+\frac{d(d-1)}{2}\left(\left(\frac{1}{2}-\psi(w)\right)^2+\frac{d-2}{6}\right)z^{-1})\right)+O(z^{-3/2})
\end{multline*}
as $z\to\infty$,
from which we easily compute the coefficients of $z^{-1/2}$ and $z^{-1}$, respectively.
\end{proof}

For the more regular Riesz-mean $R_1^N(z)$ we prove the following three-term expansion
\begin{thm}\label{R1-d-hemi-NNeu}
   As $z \to \infty$ we have  the following asymptotic expansion for the first Riesz-mean $R_1^N$ of the Neumann Laplacian eigenvalues on $\mathbb{S}^d_+$:
  \begin{equation*}
   \frac{R^N_1(z)}{L_{1,d}^{class}|\mathbb{S}_{+}^d|\,z^{\frac{d}{2}+1}} = 1+
    \frac{1}{4} \frac{L_{1,d-1}^{class}}{L_{1,d}^{class}}\frac{|\partial\mathbb{S}_{+}^d|}{|\mathbb{S}_{+}^d|}\, z^{-1/2}
    -\frac{d(d+2)}{2}\bigg(\frac{1}{4}-\psi^2(w)+\frac{d-2}{6}\bigg)z^{-1}+O(z^{-3/2})
  \end{equation*}
  or, equivalently,
  \begin{equation*}
   \frac{R^N_1(z)}{L_{1,d}^{class}|\mathbb{S}_{+}^d|\,z^{\frac{d}{2}+1}} = 1+
    \frac{d(d+2)}{2(d+1)}\, z^{-1/2} -\frac{d(d+2)}{2}\bigg(\frac{1}{4}-\psi^2(w)+\frac{d-2}{6}\bigg)z^{-1}+O(z^{-3/2})\,.
  \end{equation*}
  Here $w$ is defined by the relation $w(w+d-1)=z$.
\end{thm}
\begin{proof}
From explicit but long computations one can get
\begin{equation}\label{Hemisphere-R-1-Dirichlet-Neumann-expansion}
  R^N_1(z)=\left(1+\frac{d(d+2)}{d+1}\,z^{-1/2}+\frac{1}{2}\,\frac{d^2(d+2)^2}{(d+1)^2}\,z^{-1}+O(z^{-3/2})\right)R^D_1(z)
\end{equation}
from which the result easily follows. A simpler way of proving \eqref{Hemisphere-R-1-Dirichlet-Neumann-expansion} is to directly link the Riesz-mean for the Neumann Laplacian to the Riesz-mean for the Dirichlet Laplacian
via counting function $N^D(z)$ and to use of the explicit sum
\begin{equation*}
  \sum_{l=0}^{L}\bigg(\binom{d+l-2}{d-1}-\binom{d+l-1}{d-1}\bigg)l(l+d-1)=-\,\frac{d-1}{d+1}\cdot \frac{\Gamma(L+1+d)}{\Gamma(L)\Gamma(d)}.
\end{equation*}
This sum equals to the sum of the difference of Dirichlet and Neumann energy levels weighed by their multiplicities.
This quantity is negative since there are more Neumann eigenvalues for each energy level. Therefore we obtain the following expression
for the difference of the Riesz-means divided by the leading term in Weyl's law.
\begin{equation*}
  \frac{R^N_1(z)-R^D_1(z)}{L_{1,d}^{class}|\mathbb{S}_{+}^d|\,z^{\frac{d}{2}+1}}=\frac{N^N(z)-N^D(z)}{L_{1,d}^{class}|\mathbb{S}_{+}^d|\,z^{\frac{d}{2}}}
  -\,\frac{d-1}{d+1}\cdot \frac{\Gamma(L+1+d)}{\Gamma(L)\Gamma(d)}\cdot \frac{1}{L_{1,d}^{class}|\mathbb{S}_{+}^d|\,z^{\frac{d}{2}+1}}
\end{equation*}
where $L$ is the integer part of $w$ and $z=w(w+d-1)$.
We have already shown that $\displaystyle N^N(z)=\left(1+\frac{d}{L}\right)N^D(z)$ and
$$N^D(z)=\frac{\Gamma(L+d)}{\Gamma(L)\Gamma(d+1)}.$$
Since $\displaystyle L_{1,d}^{class}=\frac{2}{d+2}\,L_{0,d}^{class}$ we therefore have the relation

\begin{equation*}
  \frac{R^N_1(z)-R^D_1(z)}{L_{1,d}^{class}|\mathbb{S}_{+}^d|\,z^{\frac{d}{2}+1}}=\frac{d+2}{2}
  \left(\frac{d}{L}-\frac{d(d-1)(L+d)}{(d+1)z}\right)\frac{N^D(z)}{L_{0,d}^{class}|\mathbb{S}_{+}^d|\,z^{\frac{d}{2}}}.
\end{equation*}
We expand the term in parentheses using $\displaystyle L=\sqrt{z+\frac{(d-1)^2}{4}}-\frac{d}{2}-\psi(w)$ and therefore
\begin{equation*}
  L=z^{1/2}-\left(\frac{d}{2}+\psi(w)\right)+O(z^{-1/2}),\quad \frac{1}{L}=z^{-1/2}+\left(\frac{d}{2}+\psi(w)\right)z^{-1}O(z^{-3/2}).
\end{equation*}
For counting function $N^D(z)$  we have by the previous result
\begin{equation*}
  \frac{N^D(z)}{L_{0,d}^{class}|\mathbb{S}_{+}^d|\,z^{\frac{d}{2}}}=1-\frac{d}{2}\left(1+2\psi(w)\right)z^{-1/2}.
\end{equation*}
Therefore we finally obtain
\begin{equation*}
  \frac{R^N_1(z)-R^D_1(z)}{L_{1,d}^{class}|\mathbb{S}_{+}^d|\,z^{\frac{d}{2}+1}}=\frac{d(d+2)}{d+1}\,z^{-1/2}+O(z^{-3/2})
\end{equation*}
which in particular implies \eqref{Hemisphere-R-1-Dirichlet-Neumann-expansion}, concluding the proof.
\end{proof}

The results of Theorems \ref{NN} and \ref{R1-d-hemi-NNeu} are illustrated in Figure \ref{f8}.

\begin{rem}
We recall the following identities which, in fact, we have used in the proof of Theorems \ref{NN} and \ref{R1-d-hemi-NNeu}
 
\begin{equation}\label{hemisphere-counting-fct-relation-1}
N^N(w(w+d-1))=\frac{\lfloor w\rfloor+d}{\lfloor w\rfloor}N^D(w(w+d-1))
\end{equation}
or
\begin{equation}\label{hemisphere-counting-fct-relation-2}
N^N(w(w+d-1))=N^D((w+1)(w+d)).
\end{equation}

 Identity \eqref{hemisphere-counting-fct-relation-1} corresponds to \eqref{hemisphere-d-N-of-z-Neumann}.  Identity \eqref{hemisphere-counting-fct-relation-2} says that the two counting functions $N^D$, $N^N$, are equal when the $w$ variable is shifted by $1$. This fact is equivalent to a statement about the multiplicities (and clearly seen from these) defined in  \eqref{hemisphere-d-multiplicity-coeff} and \eqref{hemisphere-d-multiplicity-coeff-Neumann}.

\end{rem}

\begin{figure}[ht]
     \centering
     \begin{subfigure}[b]{0.45\textwidth}
         \centering
         \includegraphics[width=\textwidth]{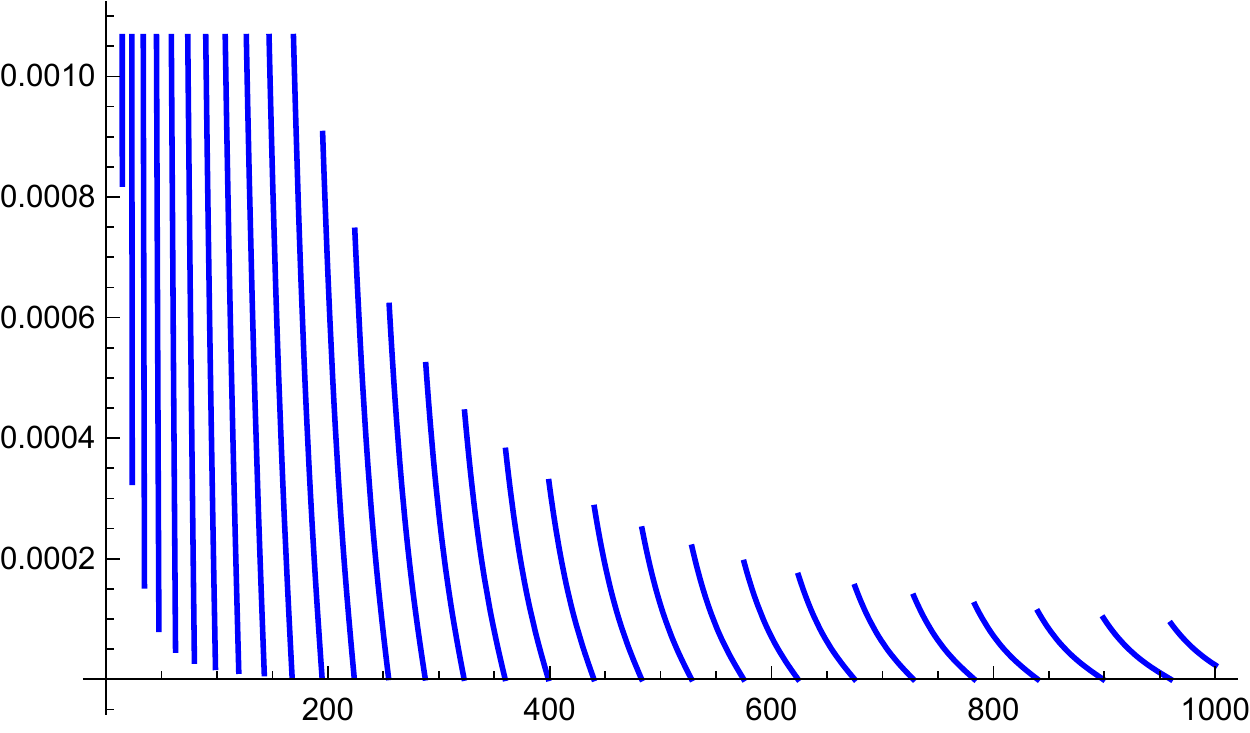}
     \end{subfigure}
     \hfill
     \begin{subfigure}[b]{0.45\textwidth}
         \centering
         \includegraphics[width=\textwidth]{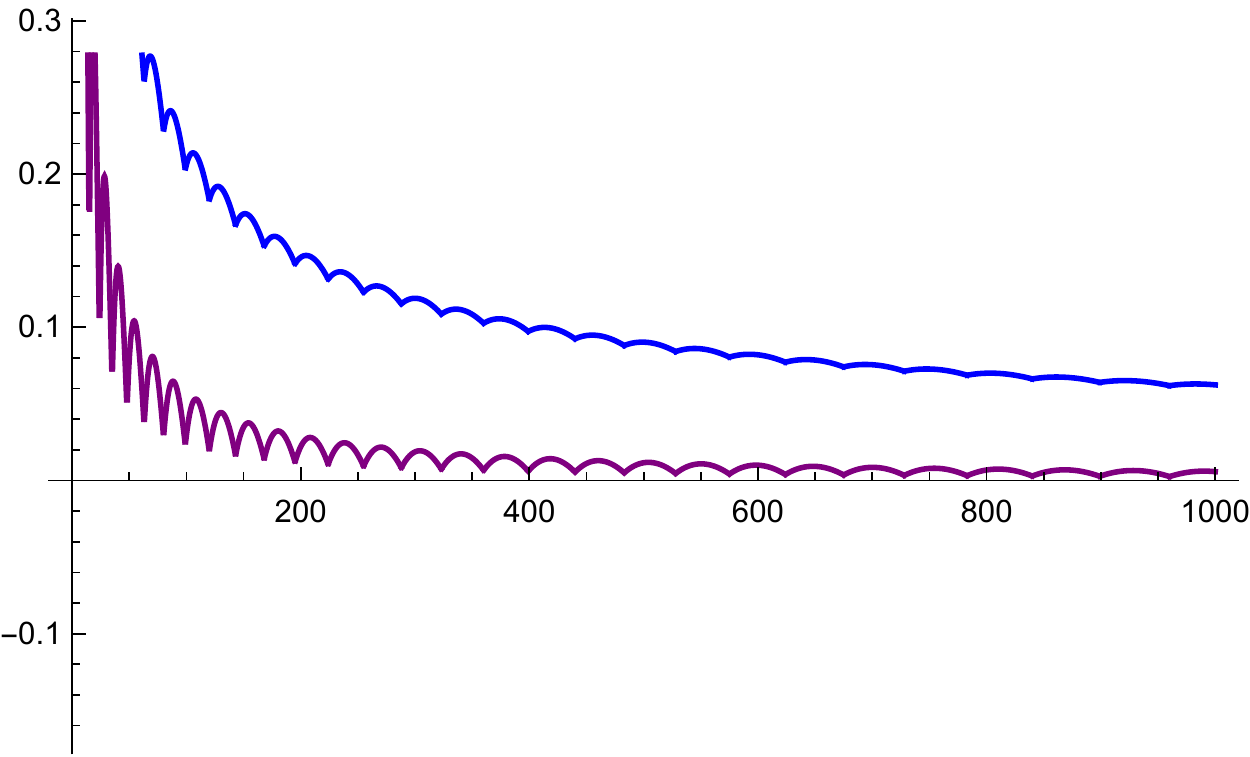}
     \end{subfigure}
     
        \caption{On the left, the ratio (minus $1$) of $N^N$ and the three-term expansion of Theorem \ref{NN}. On the right, in blue the ratio (minus $1$) of $R_1^N$ and the leading term in Weyl's law, and in purple the ratio (minus $1$) of $R_1^N$ and the three-term expansion of Theorem \ref{R1-d-hemi-NNeu}. Here $d=3$.}\label{f8}
       
\end{figure}

\subsubsection{P\'olya's conjecture}
It is well-known that P\'olya's conjecture in general fails for the Dirichlet eigenvalues of $\mathbb S^d_+$ when $d\geq 3$, while it is satisfied for $d=2$. As proved in \cite{FrMaSa22}, for the Dirichlet eigenvalues of the hemisphere, one can find subsequences of eigenvalues (corresponding to the last eigenvalue in a chain of multiple eigenvalues) which don't satisfy P\'olya's conjecture, as well as subsequences of eigenvalues which satisfy it (corresponding to the first eigenvalues in a chain of multiple eigenvalues, but starting from an energy level in general higher than $2$). We have already discussed the relation of our results, especially three-term asymptotic expansions, with those presented in \cite{FrMaSa22}.


For the sake of completeness, we briefly show here that P\'olya's conjecture does not hold in general for Dirichlet eigenvalues on  $\mathbb{S}_{+}^d$ when $d\geq 3$.


By P\'olya's conjecture we understand that the counting function $N^D$ is bounded above by
the leading term in Weyl's law, that is
\begin{equation*}
  N^D(z)\leq L_{0,d}^{class}|\mathbb{S}_{+}^d|\,z^{d/2}=\frac{1}{\Gamma(d+1)}\,z^{d/2}.
\end{equation*}
This inequality is equivalent to the eigenvalue bound
\begin{equation*}
  \lambda_j\geq \left(L_{0,d}^{class}|\mathbb{S}_{+}^d|\right)^{-2/d}\,j^{2/d}=\Gamma(d+1)^{2/d}\,j^{2/d}.
\end{equation*}
Here $\lambda_j$ are the Dirichlet eigenvalues on $\mathbb S^d_+$. We have shown in Subsection \ref{2hemi}  that this bounds hold when $d=2$. Let $d\geq 3$. We have $\lambda_1=d$ and
\begin{equation*}
  \frac{\Gamma(d+1)^{2}}{\lambda_1^d}= \prod_{j=1}^{d}\frac{j(d+1-j)}{d}=\prod_{j=1}^{d}\left(1+\frac{(j-1)(d-j)}{d}\right)>1.
\end{equation*}
On the other hand,  the counting function $N^N(z)$ for the Neumann eigenvalues on the hemisphere satisfies
\begin{equation*}
  N^N(z)\geq L_{0,d}^{class}|\mathbb{S}_{+}^d|\,z^{d/2}=\frac{1}{\Gamma(d+1)}\,z^{d/2}
\end{equation*}
as one easily sees from the identity
\begin{equation*}
  N^N(z)=L_{0,d}^{class}|\mathbb{S}_{+}^d|\,\frac{\Gamma(L+d+1)}{\Gamma(L+1)},
\end{equation*}
where, as usual, $L=\lfloor w\rfloor$ with $w(w+d-1)=z$.
However, the stronger version of P\'olya's inequality for the Neumann eigenvalues (i.e., taking into account the $0$ eigenvalue, see \cite[Corollary 1.4]{HaSt18} where it is proved for certain Euclidean domains), which reads
\begin{equation*}
  N^N(z)\geq L_{0,d}^{class}|\mathbb{S}_{+}^d|\,z^{d/2}+1=\frac{1}{\Gamma(d+1)}\,z^{d/2}+1.
\end{equation*}
does not hold.

\subsubsection{Li-Yau estimates}\label{rem-li-yau}

Usually, averaging the eigenvalues leads to a more regular behavior, as we have already seen in the previous sections.


The Weyl-sharp upper bound $R^D_1(z)\leq L_{1,d}^{class}|\mathbb{S}_{+}^d|z^{1+d/2}$ for all $z\geq 0$ for the first Riesz-mean of Dirichlet Laplacian eigenvalues on $\mathbb{S}^d_+$
is equivalent to the following estimate for averages of eigenvalues:
\begin{equation}\label{Li-Yau-estimate-d-hemisphere}
  \frac{1}{k}\sum_{j=1}^{k}\lambda_j\geq  \frac{d}{d+2}\Gamma(d+1)^{2/d}k^{2/d}
\end{equation}
for all positive integers $k$. Note that, in our notation, $\lambda_j$ denotes the $j-$th eigenvalue (and not the numbering of the energy level).
Since $\lambda_1=d$ the Li-Yau estimate \eqref{Li-Yau-estimate-d-hemisphere} for $k=1$ is equivalent to
\begin{equation*}
  (d+2)^d\geq \Gamma(d+1)^2
\end{equation*}
which only holds provided $d\leq 5$.

Therefore, an estimate on averages as \eqref{Li-Yau-estimate-d-hemisphere} cannot hold if $d\geq 6$. Clearly it holds for $d=2$ (as a consequence of the validity of P\'olya's conjecture). We actually are able to prove that  \eqref{Li-Yau-estimate-d-hemisphere} holds for $d=3,4,5$.

\begin{thm}\label{BLY-345}
For all $z\geq 0$ the following inequality for the first Riesz-mean $R_1^D$ of the eigenvalues of the Dirichlet Laplacian on $\mathbb S^d_+$, $d=3,4,5$, holds:
\begin{equation*}
    R_1^D(z)\leq L^{class}_{1,d}|\mathbb S^d_+|z^{1+\frac{d}{2}}.
\end{equation*}
\end{thm}
\begin{proof}
As usual, we write $z=w(w+d-1)$, $L=\lfloor w\rfloor$ and $w=L+x$ with $x\in[0,1[$ the fractional part of $w$. As in the proof of Theorem \ref{R1-d-hemi-N}, we write explicitly the quotient $\frac{R_1^D(z)}{L^{class}_{1,d}|\mathbb S^d_+|z^{1+\frac{d}{2}}}$ as a function of $x$. Here, however, we shall not expand in power series with respect to $x$.


When $L=0$ the claimed bound is clearly satisfied in any dimension. Hence let us consider $L\geq 1$. For any fixed $L$, the function $f_L(x)=\frac{R_1^D(z(x))}{L^{class}_{1,d}|\mathbb S^d_+|z(x)^{1+\frac{d}{2}}}$, with $z(x)=(L+x)(L+x+d-1)$, is smooth in $x\in[0,1]$. Computing its derivative, it vanishes in $(0,1)$ only at the point
$$
x=x_L=\frac{1-d-2L}{2}+\sqrt{\frac{(1-d-2L)}{4}+\frac{d+(d+2)L}{d+1}},
$$
(it is easily proven that $x_L\in(0,1)$ for any $L\geq 1$). Therefore it is sufficient to prove that $f_L(0)\leq 1$ and $f_L(x_L)\leq 1$ (since $f_L(1)=f_{L+1}(0)$). A standard computation shows that
$$
f_L(0)=\frac{(L-1)(L+1)\cdots(L+d-2)\left(L+\frac{d^2}{2(d+1)}\right)}{(L(L+d-1))^{\frac{d}{2}}}
$$
and
$$
f_L(x_L)=\frac{L(L+1)\cdots(L+d-1)}{\left((L+d)\left(L+\frac{1}{d+1}\right)\right)^{\frac{d}{2}}}.
$$
Let us prove that for $d=3,4,5$, $f_L(x_L)\leq 1$. The same proof allows to show that $f_L(0)\leq 1$ for all $d\geq 0$ (actually, $x_L$ is a local - and global -  maximum of $f_L(x)$ for $x\in[0,1]$ and all $L\geq 1$).

We write
$$
L(L+1)\cdots(L+d-1)=\left(\prod_{j=1}^d(L+j-1)(L+d-j)\right)^{\frac{1}{2}}.
$$
Applying the arithmetic-geometric inequality we get
$$
L(L+1)\cdots(L+d-1)\leq\left(\frac{1}{d}\sum_{j=1}^d(L+j-1)(L+d-j)\right)^{\frac{d}{2}}=\left(L^2+(d-1)L+\frac{(d-1)(d-2)}{6}\right)^{\frac{d}{2}}.
$$
Therefore $f_L(x_L)\leq 1$ if and only if
$$
L^2+(d-1)L+\frac{(d-1)(d-2)}{6}\leq(L+d)\left(L+\frac{1}{d+1}\right).
$$
An explicit computation shows that
$$
L^2+(d-1)L+\frac{(d-1)(d-2)}{6}-(L+d)\left(L+\frac{1}{d+1}\right)=-\frac{(d+2)(L-1)}{d+1}+\frac{(d-2)(d-5)}{6}
$$
and the right-hand side is negative for all $L\geq 1$ provided $d\leq 5$. This concludes the proof.

\end{proof}

Concerning the proof of Theorem \ref{BLY-345}, we remark that for any $d\geq 2$ there exists always $L_0>1$ such that $f_L(x_L)\leq 1$ for all $L\geq L_0$, so that Berezin-Li-Yau holds for all $z\geq z_0$, where $z_0$ depends on $d$. It doesn't hold for all $z\geq 0$, as already mentioned. In fact, for $d\geq 6$, we always have $f_1(x_1)>1$.

\section{The circle $\mathbb{S}^1$}\label{S1}
For the sake of completeness, in this brief section we consider the case of the one dimensional sphere $\mathbb{S}^1$, that is the circle.
We recall that the energy levels of the Laplacian on $\mathbb{S}^1$  are:
\[
 \lambda_{(l)}= l^2, \quad l\in\mathbb{N},
\]
with corresponding multiplicities $m_{0,1}=1$, $m_{l,1}=2$ for all $l \in \mathbb{N} \setminus\{0\}$. 
We also recall that $L_{1,1}^{class}=\frac{2}{3}\pi$ and then  $L_{1,1}^{class}|\mathbb S^1|=\frac{4}{3}$.

 As a first observation, we  show that the leading term in Weyl's law $\frac{4}{3}\,z^{3/2}$ cannot be a either lower or upper bound for the Riesz-mean $R_1(z)$. As already done several times in the previous sections, we use an auxiliary variable to simplify the computations. Namely we set $z=w^2$, $w \geq 0$. 
Clearly, for $0\leq w\leq 1$ we have $R_1(w^2)=w^2$ and then
$\frac{R_1(w^2)}{( L_{1,1}^{class}|S^1|w^3)}=\frac{3}{4} w$ which is strictly less than $1$. For $w>1$ we have
\begin{equation*}
  R_1(w^2)=w^2+\sum_{l=1}^{\lfloor w\rfloor}2(w^2-l^2)=\frac{4 w^3}{3}+\frac{w}{6}-2w\psi^2(w)-\frac{\psi(w)}{6}+\frac{2\psi^3(w)}{3}.
\end{equation*}
Clearly, in any interval between two integers there exist two $w_{\pm}$ such that $\psi(w_{\pm})=\pm \frac{\sqrt{3}}{6}$. Then
\begin{equation*}
  R_1(w_{\pm}^2)=\frac{4}{3}\, w_{\pm}^3\pm\frac{\sqrt{3}}{54}
\end{equation*}
proving that $\frac{4}{3}\,w^{3}$ is neither a lower bound nor an upper bound for For $R_1(w^2)$. 
However, if we introduce  a shift we are able to get the following Weyl sharp upper bound.
\begin{prop}\label{R1-S1}
  For all $z\geq 0$ the first Riesz-mean $R_1$ of the Laplacian eigenvalue on $\mathbb{S}^1$ satisfies the following inequality: 
  \begin{equation*}
    R_1(z)\leq \frac{4}{3}\,\left(z+\frac{1}{12}\right)^\frac{3}{2}.
  \end{equation*}
 Moreover, in  each interval $]l^2,(l+1)^2[$ with $l \in \mathbb{N}$  there exists a $z_l$ such that equality holds.
\end{prop}
\begin{proof}
We prove the inequality for $z= w^2$, $w \geq 0$.
  We start considering the difference of the squares of both sides of the claimed inequality. We have
  \begin{equation*}
    R_1(w^2)^2-\frac{16}{9}\,\left(w^2+\frac{1}{12}\right)^{3}=\frac{(12\psi^2(w)-24\psi(w) w-1)^2(3\psi^2(w)-6\psi(w) w-9w^2-1)}{972}\,.
  \end{equation*}
  We note that for $w \geq 1$ the right hand side of the above inequality is always   negative   since $-\frac{1}{2}\leq \psi\leq \frac{1}{2}$. Instead, for $0\leq w<1$ we have
  \begin{equation*}
     R_1(w^2)^2-\frac{16}{9}\,\left(w^2+\frac{1}{12}\right)^{3}=w^4-\frac{16}{9}\,\left(w^2+\frac{1}{12}\right)^{3}
     =- \frac{(48w^2+1)(6w^2-1)^2}{972} \leq 0.
  \end{equation*}
  Equality is attained when 
  $\psi(w)=w-\sqrt{w^2+\frac{1}{12}}=-\,\frac{1}{12\left(w+\sqrt{w+\frac{1}{12}}\right)}$
  which has a solution in each interval $]l,l+1[$ with $l \in \mathbb{N}$.
\end{proof}
Note that Proposition \ref{R1-S1} has been already proved (Theorem \ref{thm:ubR1dshift}). The new information of Proposition \ref{R1-S1} is that the bound is saturated. Note also that the shift $\frac{1}{12}$ corresponds exactly to $z_d$ with $d=1$ for the general Theorem \ref{thm:ubR1dshift}.


\begin{rem}
In Theorem \ref{thm:blyS2+imp} we have shown that Berezin-Li-Yau inequality holds for domains of $\mathbb S^2_+$, and that it cannot hold in general for domains invading the whole $\mathbb S^2$, since there is spectral convergence and on $\mathbb S^2$ Berezin-Li-Yau inequality does not hold. On the other hand, Kr\"oger inequality is proved for domains in $\mathbb S^d$ and this is a consequence of the fact that it holds on the whole sphere. In the case of $\mathbb S^1$ the leading term in Weyl's law is neither a lower nor an upper bound for $R_1$, but clearly, the Dirichlet and Neumann eigenvalues of any domain (i.e., each arc of length smaller than $2\pi$) satisfy Berezin-Li-Yau and Kr\"oger inequalities. This is not a contradiction, since in this case we do not have convergence of the spectrum of an arc with Dirichlet/Neumann conditions to the spectrum of $\mathbb S^1$ when the arc invades $\mathbb S^1$.
\end{rem}

\section{Higher order operators}\label{higher_order}

\subsection{Basic estimates for $\mathbb S^d$}

Let now $p\in\mathbb N$, $p\geq 2$. In this section we consider the spectrum of the polyharmonic operator $(-\Delta)^p$ on $\mathbb S^d$. It is well-known that its eigenvalues coincide exactly with the $p$-th powers of the Laplacian eigenvalues on $\mathbb S^d$, therefore the eigenvalues are given as energy levels as $\lambda_{(l)}^p$, with multiplicity $m_{l,d}$. If we want to enumerate them in increasing order, they are just given by
$$
0=\lambda_1^p<\lambda_2^p\leq\cdots\lambda_j^p\leq\cdots\nearrow+\infty\,.
$$


We first consider upper and lower bounds for $R_1^p(z)$ defined by
$$
R_1^p(z):=\sum_j(z-\lambda_j^p)_+\,.
$$
We introduce the corresponding semiclassical constant:
$$
L^{class}_{\gamma,d,p}:=(4\pi)^{-d/2}\frac{\Gamma(\gamma+1)\Gamma(1+\frac{d}{2p})}{\Gamma(1+\frac{d}{2})\Gamma(1+\gamma+\frac{d}{2p})}.
$$
We recall that, from the asymptotic expression for the counting functions (see e.g., \cite{SV}), it is immediate to deduce that, as $z$ tends to infinity 
$$
R_1^p(z)=L^{class}_{1,d,p}|\mathbb S^d|z^{1+\frac{d}{2p}}+o\left(z^{1+\frac{d}{2p}}\right).
$$
We will need the following integral transformation, which can be easily verified by a direct computation

\begin{equation}\label{Riesz-means-polyharmonic-integral-transform-1}
  \sum_{j}(z^p-\lambda_j^p)_{+}=-p(p-1)\int_{0}^{\infty}(z-t)^{p-2} \sum_{j}(z-\lambda_j-t)_{+}dt+pz^{p-1}\sum_{j}(z-\lambda_j)_{+}\,.
\end{equation}

We are ready to state upper and lower bounds for $R_1^p(z)$.

\begin{thm}\label{R1p-bounds-general}
For all $z\geq 0$ we have the following bounds for the first Riesz-mean $R_1^p$, $p\geq 2$, of the eigenvalues of $(-\Delta)^p$ on $\mathbb{S}^d$:
\begin{multline*}
  L^{class}_{1,d,p}|\mathbb S^d|z^{1+\frac{d}{2p}}-\frac{2(p-1)}{d+2}L^{class}_{1,d,p}|\mathbb S^d|\left(\left(z^{\frac{1}{p}}+z_d\right)^{\frac{d}{2}+p}-z^{1+\frac{d}{2p}}\right)\\
  \leq R_1^p(z)\\
  \leq L^{class}_{1,d,p}|\mathbb S^d|\left(z^{\frac{1}{p}}+z_d\right)^{1+\frac{d}{2p}}+\frac{2(p-1)}{d+2}L^{class}_{1,d,p}|\mathbb S^d|\left(\left(z^{\frac{1}{p}}+z_d\right)^{\frac{d}{2}+p}-z^{1+\frac{d}{2p}}\right),
\end{multline*}
where $z_d=\frac{d(2d-1)}{12}$.
\end{thm}
\begin{proof}
It is sufficient to exploit the integral transformation \eqref{Riesz-means-polyharmonic-integral-transform-1} and the bounds \eqref{S-d-R-1-lower-bound} and \eqref{S-d-R-1-bounds} in order to find upper and lower bounds for $R_1^p(z^p)$. Then, it is sufficient to replace $z$ by $z^{1/p}$ to obtain the bounds on $R_1^p(z)$. For the lower bound, use $(z-t)^{p-2}\leq (z-t+z_d)^{p-2}$ inside the integral in \eqref{Riesz-means-polyharmonic-integral-transform-1} and compute it  over $t\in(0,z+z_d)$. For the upper bound, estimate $z^{p-1}$ by $(z+z_d)^{p-1}$ in the second summand of \eqref{Riesz-means-polyharmonic-integral-transform-1}.

\end{proof}

In some specific cases we can find better lower bounds, for example, when $d=2$. In this case the integral in \eqref{Riesz-means-polyharmonic-integral-transform-1} can be computed without the estimate $(z-t)^{p-2}\leq(z-t+z_d)^{p-2}$.

\begin{cor}
Let $d=2$. For all $z\geq 0$
    \begin{equation*}
  \frac{2}{3}\,z^{3/2}-\frac{1}{2}\,z-\frac{1}{4}\,z^{1/2}\leq  R_{1}^2(z)\leq  \frac{2}{3}\,z^{3/2}+z+\frac{1}{4}\,z^{1/2},
\end{equation*}
and for all $p\geq 1$
\begin{equation*}
  \frac{p}{p+1}\,z^{1+1/p}-\frac{p-1}{2}\,z-\frac{p}{8}\,z^{1-1/p}\leq  R_{1}^p(z)\leq  \frac{p}{p+1}\,z^{1+1/p}+\frac{p}{2}\,z+\frac{p}{8}\,z^{1-1/p}.
\end{equation*}
\end{cor}

In some sense, Theorem \ref{R1p-bounds-general} is the analogous of Proposition \ref{R1-S1} for the unit circle. In general, the leading term in Weyl's law is not a lower bound for $R_1^p$, as discussed in the next remark. We refer to Figure \ref{f9} where this behavior is clearly depicted in the case $p=2,d=2$.

\begin{rem}\label{rem_pd}
As in the case of $\mathbb S^1$ (Section \ref{S1}), we note that when $p=2$, $d=2$, the leading term in Weyl's law, $\frac{2}{3}z^{3/2}$ is neither a lower nor an upper bound for $R_1^2(z)$. In fact, for all $l\in\mathbb N\setminus\{0\}$, one can verify that $R_1^2(l^2(l+1)^2)<\frac{2}{3}l^3(l+1)^3$. On the other hand, for all $l\in\mathbb N\setminus\{0\}$, one can verify that $R_1^2((1 + l)^2 (2 + l (2 + l)))>\frac{2}{3}(1 + l)^3 (2 + l (2 + l))^{3/2}$.


However, taking $p=2,d=3$ one verifies that, for $z=l^2(l+2)^2$, one has $R_1^2(l^2(l+2)^2)>\frac{4}{21}l^{7/2}(l+2)^{7/2}=L^{class}_{1,3,2}|\mathbb S^3|(l^2(l+2)^2)^{7/4}$, and such points are  the local minima of the function $R_1^2(z)-\frac{4}{21}z^{7/4}$ in each interval $[l^2(l+2)^2,(l+1)^2(l+2)^2]$. Therefore the basic asymptotically Weyl-sharp bounds of Theorem \ref{R1p-bounds-general} are not always optimal in the further terms.

\end{rem}


In view of this remark, it is natural to conjecture that the leading term in Weyl's law is a lower bound for $d>p$. For $p=1$ this is in fact true. We are able to prove this conjecture for $p=2$, i.e., for the biharmonic operator $\Delta^2$.

\begin{thm}\label{lower-R-1-bil-sphere}
Let $d\geq 3$. For all $z\geq 0$ we have the following 
inequality for the first Riesz-mean $R_1^2$ of the biharmonic eigenvalues on $\mathbb{S}^d$:
\begin{equation*}
R_1^2(z)\geq L^{class}_{1,d,2}|\mathbb S^d|z^{1+\frac{d}{4}}=\frac{8}{(d+4)\Gamma(d+1)}z^{1+\frac{d}{4}}.
\end{equation*}
\end{thm}

\begin{proof}
The proof follows the same lines as that of Theorem \ref{R-1-lo-thm-sphere}, and uses the analogous of Lemma \ref{prop-R1-l-b} which holds when we replace $R_1$ by $R_1^2$. As in the proof of Theorem \ref{R-1-lo-thm-sphere}, one has to show that $\frac{R_1^2(\lambda_{(l+1)}^2)}{L_{1,d,2}|\mathbb S^d|\lambda_{(l+1)}^{2+\frac{d}{2}}}$ is lower bounded by $1$ for all $l\in\mathbb N$. This expression equals
$$
\frac{\Gamma(l+d+1)}{\Gamma(l+1)(l+1)^{d/2}(l+d)^{d/2}}\cdot \left(1+Ah+Bh^2\right),
$$
where $A,B$ can be computed explicitly and depend only on $d$, and $h=\frac{1}{(l+1)(l+d)}$. The first factor is lower bounded by $1$ as shown in the proof of Theorem \ref{R-1-lo-thm-sphere}. We sketch how to prove that the whole expression is lower bounded by $1$. When $d\geq 4$, using $h\leq\frac{1}{d}$ and the explicit expression of $A,B$, one immediately sees that $(1+Ah+Bh^2)\geq 1+\frac{d(d-4)}{8} h\geq 1$. The proof for $d=3$ requires a more refined lower bound for the factor $\frac{\Gamma(l+3+1)}{\Gamma(l+1)(l+1)^{3/2}(l+3)^{3/2}}$ which is given by $(1+h)^{1/2}$. Since $(1+h)^{1/2}\left(1-\frac{3}{8} h\right)\geq 1$ in the range $h\leq\frac{1}{3}$, we deduce the bound also for $d=3$.
\end{proof}

\begin{figure}[ht]
    \centering
    \includegraphics[width=0.7\textwidth]{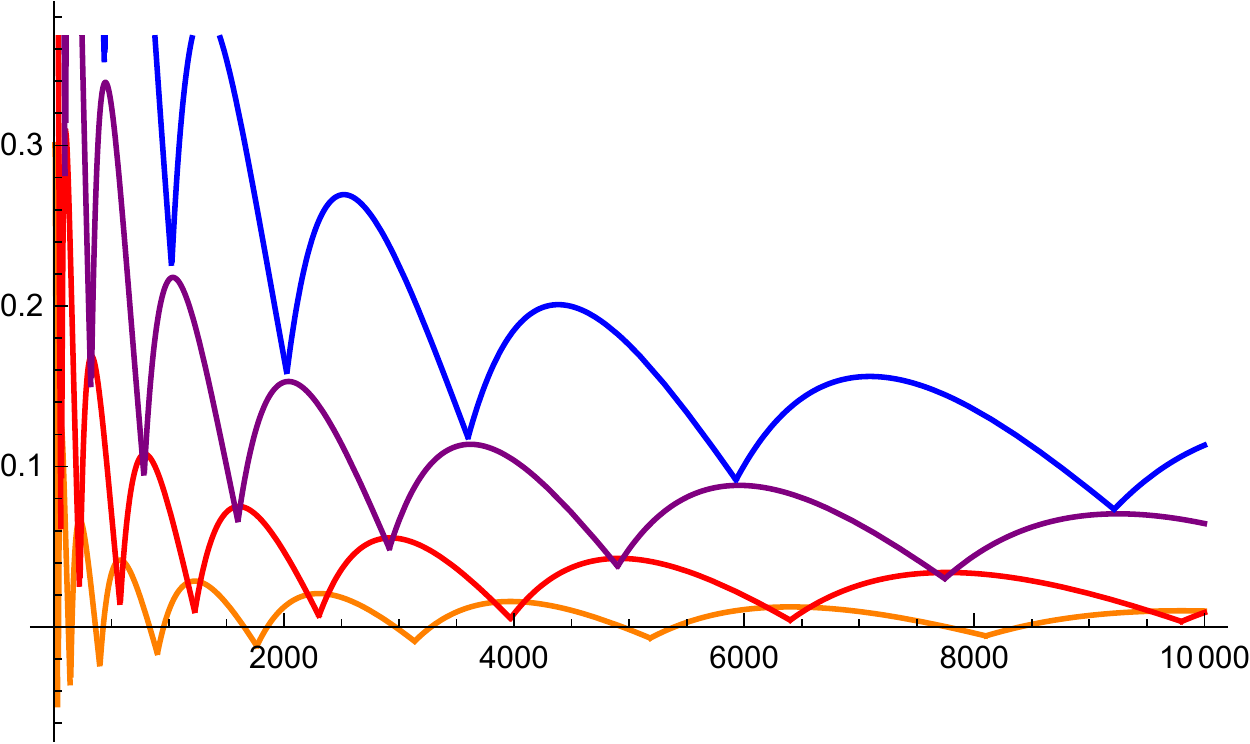}
    \caption{The ratio (minus $1$) of $R_1^2$ and the leading term in Weyl's law for $d=2$ (orange), $d=3$ (red), $d=4$ (purple) and $d=5$ (blue).}\label{f9}
\end{figure}

\subsection{The Dirichlet problem on domains of $\mathbb S^2_+$}

In this subsection we consider the Dirichlet problem for the polyharmonic operator $(-\Delta)^p$ on domains  $\Omega \subseteq \mathbb S^2_+$, and we will provide Berezin-Li-Yau-type bounds when $p=2,3$ (i.e., for the biharmonic and the triharmonic operator), extending the results of \cite{Lap97}. The Dirichlet problem for $(-\Delta)^p$ reads 
\begin{equation}\label{poly-c}
    \begin{cases}
    (-\Delta)^pu=\Lambda u\,, & {\rm in\ }\Omega\\
    u=\partial_{\nu}u=\cdots=\partial^{p-1}_{\nu^{p-1}}u=0\,, & {\rm on\ }\partial\Omega,
    \end{cases}
\end{equation}
where $\partial^m_{\nu^m}$ denotes the $m$-th partial derivative in the direction of the outer unit normal. As customary, we understand problem \eqref{poly-c} in its weak formulation:
\begin{equation}\label{poly-w-even}
\int_{\Omega}\Delta^{p/2}u\Delta^{p/2}\phi= \Lambda\int_{\Omega}u\phi\,,\ \ \ \forall\phi\in H^p_0(\Omega),   
\end{equation}
if $p$ is even, and 
\begin{equation}\label{poly-w-odd}
\int_{\Omega}\nabla\Delta^{(p-1)/2}u\cdot\nabla \Delta^{(p-1)/2}\phi= \Lambda\int_{\Omega}u\phi\,,\ \ \ \forall\phi\in H^p_0(\Omega),   
\end{equation}
if $p$ is odd.


It is standard to prove that problems \eqref{poly-w-even}-\eqref{poly-w-odd} admit an increasing sequence of non-negative eigenvalues
$$
0<\Lambda_1(\Omega)\leq\Lambda_2(\Omega)\leq\cdots\leq\Lambda_j(\Omega)\leq\cdots\nearrow+\infty.
$$

First, we note that Weyl-sharp upper bounds for the first Riesz-mean hold for all $p$ when $\Omega=\mathbb S^2_+$. For this purpose, we recall the following integral transform

\begin{equation}\label{Riesz-means-polyharmonic-integral-transform-2}
  \sum_{j}(z^p-\lambda_j^p(\Omega))_{+}=p\int_{0}^{\infty}(z-t)^{p-1} N^D(z-t)\,dt.
\end{equation}
Here $N^D(z)$ denotes the counting function for the eigenvalues $\lambda_j(\Omega)$ of the Dirichlet Laplacian on a domain. We denote by $R_1^{p,D}(z)$ the first Riesz-mean for problem \eqref{poly-c}:
$$
R_1^{p,D}(z):=\sum_{j}(z-\Lambda_j(\Omega))_+.
$$
Noting that, for any domain $\Omega$, $\Lambda_j(\Omega)\geq\lambda_j^p(\Omega)$ (it follows just by the min-max formulation of the two sequences of eigenvalues, see e.g., \cite{bps} for the case $p=2$), we deduce that
$$
R_1^{p,D}(z)\leq\sum_j(z-\lambda_j^p(\Omega))_+.
$$
Therefore, any upper bound for $N^D(z)$ translates into an upper bound for $\sum_j(z^p-\lambda_j^p(\Omega))_+$ and consequently for $R_1^{p,D}(z)$. A Weyl-sharp upper bound for $N^D(z)$ when $\Omega=\mathbb S^2_+$ is given by \eqref{HS-2-N-bound}, which, combined with  \eqref{Riesz-means-polyharmonic-integral-transform-2},  allows to prove the following

\begin{thm}\label{b-l-y-p-hemi}
Let $\Omega=\mathbb S^2_+$. For all $z\geq 0$ we have
the following inequality for the first Riesz-mean of  Dirichlet eigenvalues of $(-\Delta)^p$ on $\mathbb{S}^2_+$:
\begin{equation*}
R_1^{p,D}(z)\leq\frac{p}{2(p+1)}z^{1+\frac{1}{p}}=L^{class}_{1,2,p}|\mathbb S^2_+|z^{1+\frac{1}{p}}.    
\end{equation*}
\end{thm}

We point out \cite[Example 1.7.13]{SV} where the authors derive a two-term formula for the Dirichlet eigenvalues of $\Delta^2$ on $\mathbb S^2_+$, where the second term contains the surface measure term, and an additional oscillating part. As for Laplacian eigenvalues, this is different from the Euclidean case, where only the surface term contributes \cite[Section 6.2]{SV}, which is also the case for the biharmonic operator with other boundary conditions, see \cite[Formulas (3.14) - (3.18) ]{bps}.

We can prove Theorem \ref{b-l-y-p-hemi} for any domain contained in $\mathbb S^2_+$ when $p=2,3$. To do so, we prove the following result (cf.  \cite{EHIS,Ga,ilyin_laptev,stri}).

\begin{prop}
For any $\Omega\subseteq\mathbb S^2_+$ and any $z\geq 0$ we have
\begin{equation*}
R_1^{p,D}\leq\frac{|\Omega|}{4\pi}\sum_{l\geq 1}(2l+1)(z-l^p(l+1)^p)_+.
\end{equation*}
\end{prop}
\begin{proof}
The proof is in the spirit of Theorem \ref{thm:blyS2+}. Namely, We apply the averaged variational principle to the  eigenvalues $\gamma_j$ of the following intermediate problem:
\begin{equation}\label{intermediate}
\begin{cases}
(-\Delta)^pu=\gamma u\,, &  {\rm in\ }\mathbb S^2_+,\\
u=\Delta u=\cdots=\Delta^{p-1}u=0\,, &  {\rm on\ }\partial\mathbb S^2_+\,.
\end{cases}
\end{equation}
The weak formulation of this problem reads exactly as \eqref{poly-w-even}-\eqref{poly-w-odd}, except that the energy space $V(\mathbb S^2_+)$ is given by
$$
V(\mathbb S^2_+):=\{f\in H^p(\mathbb S^2_+):f=\Delta f=\cdots=\Delta^{(p-2)/2}f=0\},
$$
if $p$ is even, and 
$$
V(\mathbb S^2_+):=\{f\in H^p(\mathbb S^2_+):f=\Delta f=\cdots=\Delta^{(p-1)/2}f=0\},
$$
if $p$ is odd. The equalities are intended in the sense of traces. Problem \eqref{intermediate} admits a sequence of positive eigenvalues
$$
0<\gamma_1\leq\gamma_2\leq\cdots\leq\gamma_j\leq\cdots\nearrow+\infty.
$$
We don't highlight for $\gamma_j$ the dependence on the domain, being it fixed and equal $\mathbb S^2_+$.

It is not difficult to show (see e.g., \cite{bps} for the case $p=2$) that in the case of a smooth domain the eigenvalues of \eqref{intermediate} are exactly the $p$-th powers of the Dirichlet eigenvalues of the Laplacian. This is clearly the case of $\mathbb S^2_+$. 
Moreover, the eigenfunctions are the same. In particular, the eigenvalues $\gamma_j$ are given as energy levels by $\lambda_{(l)}^p=l^p(l+1)^p$, $l\in\mathbb N\setminus\{0\}$, with associated eigenfunctions $Y_l^{-l-1+2h}$, $h=1,...,l$.


As test functions for the averaged variational principle for $\gamma_j$ we use the Dirichlet eigenfunctions $u_j$ associated with the eigenvalues $\Lambda_j(\Omega)$ of \eqref{poly-c} on a domain $\Omega\subseteq\mathbb S^2_+$, extended by zero to $\mathbb S^2_+$. Clearly these extensions belong to $V(\mathbb S^2_+)$.
 We have then
$$
\sum_{l\geq1}\sum_{h=1}^{l}(z-\lambda_{(l)}^p)_+\sum_{j\geq 1}\left|\int_{\Omega} u_j Y_l^{-l-1+2h}\right|^2\geq\sum_{j\in J}\int_{\Omega}z|u_j|^2-(\Delta^{p/2} u_j)^2=\sum_{j\in J}(z-\Lambda_j(\Omega)),
$$
for any $J\subset\mathbb N$, if $p$ is even. If $p$ is odd, just replace $(\Delta^{p/2} u_j)^2$ by $|\nabla \Delta^{(p-1)/2}u|^2$ in the right-hand side. Then we can replace the right-hand side by $\sum_j(z-\Lambda_j(\Omega))_+$. As for the left-hand side, note that

\begin{multline*}
\sum_{l\geq 1}\sum_{h=1}^{l}(z-\lambda_{(l)}^p)_+\sum_{j\geq 1}\left|\int_{\Omega}u_j Y_l^{-l-1+2h}\right|^2=\sum_{l\geq 1}(z-\lambda_{(l)}^p)_+\sum_{h=1}^{l}\sum_{j\geq 1}\left|\int_{\Omega}u_jY_l^{-l-1+2h}\right|^2\\
\leq \sum_{l\geq 1}(z-\lambda_{(l)}^p)_+\int_{\Omega}\sum_{m=-l}^{l}|Y_l^m|^2=\frac{|\Omega|}{|\mathbb S^2|}\sum_{l\geq 1}(2l+1)(z-\lambda_{(l)}^p)_+.
\end{multline*}
This concludes the proof.
\end{proof}

In order to find an upper bound on $R_1^{p,D}(z)$ we need to find an upper bound for $\sum_{l\geq 1}(2l+1)(z-l^p(l+1)^p)_+$.


We prove a Weyl-sharp upper bound when $p=2,3$, i.e., for the biharmonic and the triharmonic problems. 

\begin{lemma}
For all $z\geq 0$ we have, when $p=2$
$$
\sum_{l\geq 1}(2l+1)(z-l^2(l+1)^2)_+\leq L^{class}_{1,2,2}|\mathbb S^2|z^{\frac{3}{2}}=\frac{2}{3}z^{\frac{3}{2}},
$$
and, for $p=3$
$$
\sum_{l\geq 1}(2l+1)(z-l^3(l+1)^3)_+\leq L^{class}_{1,2,3}|\mathbb S^2|z^{\frac{4}{3}}=\frac{3}{4}z^{\frac{4}{3}}.
$$
\end{lemma}
The proof is very similar to that of Lemma \ref{lem:blys1} and is accordingly omitted. We remark that the inequality $\sum_{l\geq 1}(2l+1)(z-l^p(l+1)^p)_+\leq L^{class}_{1,2,p}|\mathbb S^2|z^{1+\frac{1}{p}}$ is no longer true for $p\geq 4$. For example, take $z=81$. It is easily shown that $R_1^{4,D}(81)=195$. On the other hand $L^{class}_{1,2,4}|\mathbb S^2|81^{\frac{4}{3}}=972/5=194.4$. This is clearly visible in Figure \ref{f10}.

We are ready to state a Berezin-Li-Yau bound for the biharmonic and triharmonic operator for domains in $\mathbb S^2_+$. 

\begin{thm}\label{b-l-y-poly-hemi}
Let $\Omega$ be a domain in $\mathbb S^2_+$. Then for all $z\geq 0$ the following inequalities for the first Riesz-mean $R_1^{p,D}(z)$, $p=2,3$, of the  Dirichlet eigenvalues on $\Omega$ of the biharmonic and triharmonic operators hold:
\begin{equation*}
R_1^{2,D}(z)\leq\frac{|\Omega|}{6\pi}z^{\frac{3}{2}}
\end{equation*}
\begin{equation*}
R_1^{3,D}(z)\leq\frac{3|\Omega|}{16\pi}z^{\frac{4}{3}}
\end{equation*}

\end{thm}

We remark that our argument does not allow to prove Berezin-Li-Yau bounds for $p\geq 4$, however we conjecture that the bounds hold for any $p$ (this is the case when $\Omega=\mathbb S^2_+$).

\begin{figure}[ht]
    \centering
    \includegraphics[width=0.7\textwidth]{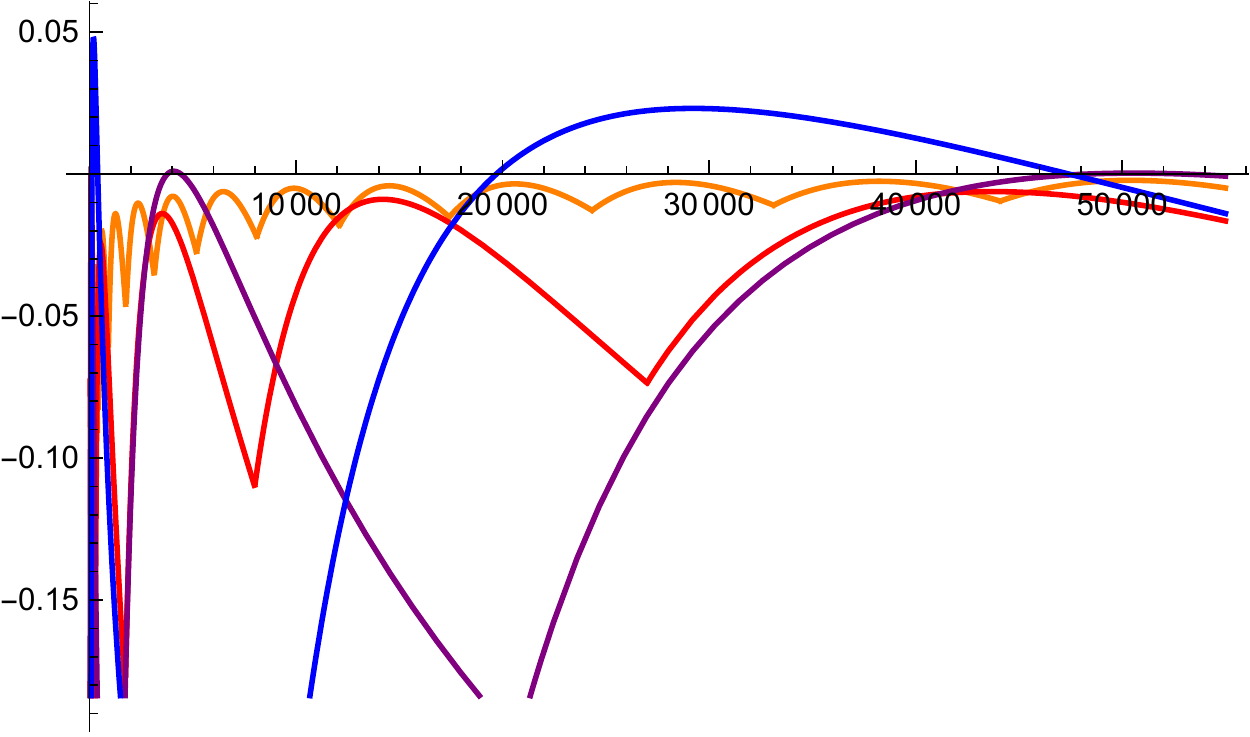}
    \caption{The ratio (minus $1$) of $R_1^p-z$ (that is, neglecting the zero eigenvalue) and the leading term in Weyl's law for $p=2$ (orange) $p=3$ (red), $p=4$ (purple) and $p=5$ (blue). Here $d=2$.}\label{f10}
\end{figure}

\subsection{The buckling problem on domains of $\mathbb S^2$}

In this subsection we consider the buckling problem on domains of $\mathbb S^2$, and we provide Berezin-Li-Yau-type upper bounds for the first Riesz-mean. The buckling problem reads:

\begin{equation}\label{buckling}
\begin{cases}
-\Delta^2u=\sigma \Delta u \,, & {\rm in\ }\Omega,\\
u=\partial_{\nu}u=0\,, & {\rm on\ }\partial\Omega,
\end{cases}
\end{equation}
in the unknowns $u$ (the eigenfunction) and $\sigma$ (the eigenvalue). Clearly this problem makes sense on any domain of $\mathbb S^d$, $d\geq 2$. However, in this subsection we will mainly concentrate on $d=2$.


As customary, we understand problem \eqref{buckling} in its weak formulation:

\begin{equation}\label{buckling_weak}
\int_{\Omega}  \Delta u\Delta \phi=\sigma\int_{\Omega} \nabla u\cdot \nabla \phi \,,\ \ \ \forall\phi\in H^2_0(\Omega).
\end{equation}
It is standard to prove that problem \eqref{buckling_weak} admits an increasing sequence of non-negative eigenvalues given by
$$
0<\sigma_1(\Omega)\leq\sigma_2(\Omega)\leq\cdots\leq \sigma_j(\Omega)\leq\cdots\nearrow+\infty.
$$
We refer to \cite{bpls_buckling} for more information on spectral asymptotics for the buckling problem.


An immediate consequence of the variational characterization of buckling eigenvalues is that
$$
\Lambda_j(\Omega)\leq\sigma_j^2(\Omega),
$$
where $\Lambda_j(\Omega)$ are the Dirichlet eigenvalues of the biharmonic operator (problem \eqref{poly-c} with $p=2$). Therefore we can deduce Weyl-sharp upper bounds for the first Riesz-mean of $\sigma_j(\Omega)$ from Theorem \ref{b-l-y-poly-hemi}. However we show here that a Berezin-Li-Yau-type upper bound for buckling eigenvalues holds for any domain of $\mathbb S^2$. For the Euclidean case Berezin-Li-Yau-type upper bounds for the first Riesz-mean are proved in \cite{levine_protter} (see also \cite{bpls_buckling}).

A preliminary observation is that, when $\Omega=\mathbb S^2$, we can relate  the eigenvalues of \eqref{buckling} with the eigenvalues of the Laplacian on $\mathbb S^2$. In fact, for all dimensions $d$, we have the following

\begin{lemma}
 The eigenvalues of \eqref{buckling} on $\mathbb S^d$ are exactly the eigenvalues $\lambda_{(l)}$ of the Laplacian on $\mathbb S^d$, except the zero eigenvalue, with the same multiplicities, namely $\lambda_{(l)}=l(l+d-1)$, $l\in\mathbb N\setminus\{0\}$, with multiplicity $m_{l,d}$.
\end{lemma}
\begin{proof}
First of all, note that $H^2(\mathbb S^d)=H^2_0(\mathbb S^d)$. Problem \eqref{buckling_weak}  is well-defined once we quotient out the constant functions. Namely, we consider
\begin{equation}\label{buckling_M}
\int_{\mathbb S^d}\Delta u\Delta\phi=\sigma\int_{\mathbb S^d}\nabla u\cdot\nabla \phi\,,\ \ \ \forall\phi\in H^2(\mathbb S^d):\int_{\mathbb S^d}\phi=0
\end{equation}
in the unknown $u\in H^2(\mathbb S^d)$ with $\int_{\mathbb S^d}u=0$. In fact, identity \eqref{buckling_weak} holds true if we replace $u,\phi$ by $u+a$, $\phi+b$, with $a,b\in\mathbb R$. Hence we rather consider problem \eqref{buckling_weak} in the subspace $\{u\in H^2(\mathbb S^d):\int_{\mathbb S^d}u=0\}$. From Bochner's formula
$$
\int_{\mathbb S^d}(\Delta u)^2=\int_{\mathbb S^d}|D^2u|^2+(d-1)|\nabla u|^2\geq\int_{\mathbb S^d}\frac{(\Delta u)^2}{d}+(d-1)|\nabla u|^2
$$
we deduce $\int_{\mathbb S^d}|\nabla u|^2\leq \frac{1}{d}\int_{\mathbb S^d}|\Delta u|^2$.
Therefore the eigenvalues are strictly positive and the quadratic form is coercive.


Let then $\sigma_j$ be an eigenvalue associated with some eigenfunction $u_j$ of \eqref{buckling_M} on $\mathbb S^d$. Then, taking $f_j=\Delta u_j$, we see that $-\Delta u_j=\sigma_j f_j$ and $\int_{\mathbb S^d}f_j=0$, hence $f_j$ is an eigenfunction of the Laplacian on $\mathbb S^d$ with eigenvalue $\sigma_j\ne 0$. On the other hand, if $\lambda_j\ne 0$ is an eigenvalue of the Laplacian on $\mathbb S^d$ with eigenfunction $u_j$, then $\int_{\mathbb S^d}u_j=0$ and clearly satisfies \eqref{buckling_M} with $\sigma=\lambda_j$.
\end{proof}

We have then that the eigenvalues of the buckling problem on the whole sphere coincide with the eigenvalues of the Laplacian, zero excluded. In $d=2$ the possibility to get rid of the constant eigenfunction was an essential point in order to prove Berezin-Li-Yau bounds for domains on $\mathbb S^2_+$. In this sense, for the buckling problem we got rid in a natural way of the constant eigenfunction on the whole space, so we can prove Berezin-Li-Yau-type bounds for domains on the whole $\mathbb S^2$. 


In this section we denote by $R^B_1(z)$ the Riesz-mean for buckling eigenvalues on a domain $\Omega$
$$
R_1^B(z):=\sum_j(z-\sigma_j(\Omega))_+.
$$
We prove  the following result, in the spirit of \cite{EHIS,Ga,ilyin_laptev,stri}.
\begin{prop}
For any $\Omega\subseteq\mathbb S^2$, and any $z\geq 0$ we have
\begin{equation*}
R_1^B(z)\leq \frac{|\Omega|}{|\mathbb S^2|}\sum_{l\geq 1}(2l+1)(z-l(l+1))_+.
\end{equation*}
\end{prop}
\begin{proof}
We start as in the proof of Theorem \ref{thm:blyS2+}. Namely, We apply the averaged variational principle to the positive eigenvalues $\lambda_j$ of the Laplacian on $\mathbb S^2$. We interpret them as the eigenvalues of \eqref{buckling_M} on $\mathbb S^d$.

 Note that, if $Y_l^m$, $m=-l,...,l$, are the $L^2(\mathbb S^2)$-normalized eigenfunctions corresponding to the energy level $\lambda_{(l)}$ of the  Laplacian on $\mathbb S^2$, then $\left\{\frac{Y_l^m}{\sqrt{\lambda_{(l)}}}\right\}$ form  a orthonormal family with respect to the product of the gradients, which is what we need for the averaged variational principle for buckling eigenvalues on $\mathbb S^d$. As test functions for the averaged variational principle we use the following functions:
 $$
 \tilde u_j:=u_j-\frac{1}{|\mathbb S^d|}\int_{\Omega}u_j,
 $$
 where $u_j$ are buckling eigenfunctions on $\Omega$, extended by zero to $\mathbb S^d$. Namely, we extend by zero the buckling eigenfunctions $u_j$, and add a constant in such a way that $\int_{\mathbb S^d}\tilde u_j=0$. However, $\nabla\tilde u_j=\nabla u_j$ and $\Delta\tilde u_j=\Delta u_j$.
 We have then
$$
\sum_{l\geq1}\sum_{m=-l}^{l}(z-\lambda_{(l)})_+\sum_{j\geq 1}\left|\int_{\Omega} \nabla u_j\cdot\frac{1}{\sqrt{\lambda_l}}\nabla Y_l^m\right|^2\geq\sum_{j\in J}\int_{\Omega}z|\nabla u_j|^2-(\Delta u_j)^2=\sum_{j\in J}(z-\sigma_j(\Omega)),
$$
for any $J\subset\mathbb N$. Then we can replace the right-hand side by $\sum_j(z-\sigma_j(\Omega))_+$. As for the left-hand side, note that

\begin{multline*}
\sum_{l\geq 1}\sum_{m=-l}^l(z-\lambda_{(l)})_+\sum_{j\geq 1}\left|\int_{\Omega} \nabla u_j\cdot\frac{1}{\sqrt{\lambda_{(l)}}}\nabla Y_l^m\right|^2=\sum_{l\geq 1}(z-\lambda_{(l)})_+\sum_{m=-l}^l\sum_{j\geq 1}\left|\int_{\Omega} \nabla u_j\cdot\frac{1}{\sqrt{\lambda_{(l)}}}\nabla Y_l^m\right|^2\\
\leq \sum_{l\geq 1}(z-\lambda_{(l)})_+\int_{\Omega}\sum_{m=-l}^l\frac{1}{\lambda_{(l)}}|\nabla Y_l^m|^2=\frac{|\Omega|}{|\mathbb S^2|}\sum_{l\geq 1}(2l+1)(z-\lambda_{(l)})_+,
\end{multline*}
where in the last passage we used the identity
$$
\sum_{m=-l}^l\frac{1}{\lambda_{(l)}}|\nabla Y_l^m|^2=\frac{2l+1}{|\mathbb S^2|}
$$
This relation is well-known. It follows, for example, from the addition formula (see \cite{folland,gine}) 
$$
\sum_{m=-l}^l| Y_l^m|^2=\frac{2l+1}{|\mathbb S^2|}
$$
by taking the Laplacian at both sides and using $-\Delta Y_l^m=\lambda_{(l)}Y_l^m$ (see also Lemma \ref{add} below).
\end{proof}

We have an explicit upper bound on $\sum_{l\geq 1}(2l+1)(z-l(l+1))_+$,  from Lemma \ref{lem:blys1}, or an improved upper bound from Lemma \ref{lem:blys2}. This yields the following

\begin{thm}\label{BLY-buckling}
Let $\Omega\subseteq\mathbb S^2$. Then for any $z\geq 0$
the following bounds for the first Riesz-mean $R_1^B$ of buckling eigenvalues on $\Omega$ holds: 
\begin{equation*}
R_1^B(z)\leq\frac{|\Omega|}{8\pi}\left(z-\frac{1}{2}\right)^2
\end{equation*}
and
\begin{equation*}
R_1^B(z)\leq\frac{|\Omega|}{8\pi}z^2.
\end{equation*}
\end{thm}

Therefore, for any domain of the sphere $\mathbb S^2$, Berezin-Li-Yau upper bound holds for buckling eigenvalues, and even for the sphere itself.


\subsection{The Neumann problem for the biharmonic operator on domains of $\mathbb S^d$}

In this subsection we consider the biharmonic operator on domains of $\mathbb S^d$, $d\geq 3$ and establish Kr\"oger-type lower bounds for the first Riesz-mean, thus extending the results of \cite{ilyin_laptev,Lap97}.


Let $\Omega$ be a bounded and smooth domain in $ \mathbb S^d$. The Neumann problem for the biharmonic operator $\Delta^2$ reads
\begin{equation}\label{Neumann-bi}
\begin{cases}
\Delta^2u=Mu \,, & {\rm in\ }\Omega,\\
\partial^2_{\nu^2}u=0\,, & {\rm on\ }\partial\Omega,\\
{\rm div}_{\partial\Omega}(\nabla_{\nu}\nabla u)_{\partial\Omega}+\partial_{\nu}\Delta u=0\,, & {\rm on\ }\partial\Omega.
\end{cases}
\end{equation}
in the unknowns $u$ (the eigenfunction) and $M$ (the eigenvalue). Here ${\rm div}_{\partial\Omega}$ is the divergence on $\partial\Omega$ with respect to the induced metric, and $F_{\partial\Omega}$ denotes the projection of $F\in TM$ on $T{\partial\Omega}$. We understand problem \eqref{Neumann-bi} in its weak formulation:

\begin{equation}\label{Neumann_bi_weak}
\int_{\Omega} \langle D^2u,D^2\phi\rangle+(d-1)\nabla u\cdot\nabla\phi=M\int_{\Omega} u\phi \,,\ \ \ \forall\phi\in H^2(\Omega).
\end{equation}
Here $D^2u$ denotes the Hessian of $u$. Note that here, differently from the Euclidean case, the gradient term in the left-hand side of \eqref{Neumann_bi_weak} is not associated with a tension term, but is rather an intrinsic part of the operator due to the non-flat geometry of the ambient space (see e.g., \cite{buoso16,bcp} for the Euclidean case, \cite{colbois_provenzano} for the general Riemannian case). In general, the quadratic form associated with the Neumann problem is given by $\int_{\Omega}|D^2u|^2+{\rm Ric}(\nabla u,\nabla u)$; in the case of the sphere, ${\rm Ric}(X,Y)=(d-1)  X\cdot Y$ for all $X,Y\in T\mathbb S^d$.

It is standard to prove that problem \eqref{Neumann_bi_weak} admits an increasing sequence of non-negative eigenvalues given by
$$
0=M_1(\Omega)<M_2(\Omega)\leq\cdots\leq M_j\leq\cdots\nearrow+\infty.
$$
The first eigenvalue is $M_1(\Omega)=0$ with corresponding constant eigenfunctions. We refer to \cite{colbois_provenzano} for more details on the biharmonic Neumann problem on manifolds.

In particular, on the whole sphere $\mathbb S^d$, the eigenvalues are exactly the squares of the Laplacian eigenvalues.


We recall a preliminary lemma (an addition formula for higher derivative), which follows from the addition formula for the eigenfunctions of the Laplacian on $\mathbb S^d$ (see \cite{folland,gine}).

\begin{lemma}\label{add}
Let $\lambda_{(l)}$, $l\in\mathbb N$, denote an energy level of the Laplacian on $\mathbb S^d$. Let $\left\{Y_l^m\right\}_{m=1}^{m_{l,d}}$ be a $L^2(\mathbb S^d)$-orthonormal basis of the corresponding eigenspace. Then
\begin{equation}\label{sum_grad_hess}
\sum_{m=1}^{m_{l,d}}|\nabla Y_l^m|^2=\frac{m_{l,d}\lambda_{(l)}}{|\mathbb S^d|}\ \ \ {\rm and\ \ \ }\sum_{m=1}^{m_{l,d}}|D^2Y_l^m|^2+(d-1)|\nabla Y^m_l|^2=\frac{m_{l,d}\lambda_{(l)}^2}{|\mathbb S^d|}.
\end{equation}
\end{lemma}
\begin{proof}
Let us consider the addition formula for spherical harmonics in $d$ dimensions \cite{folland,gine}
$$
\sum_{m=1}^{m_{l,d}}|Y_l^m|^2=\frac{m_{l,d}}{|\mathbb S^d|}
$$
and take the Laplacian at both sides. We obtain
$2\sum_{m=1}^{m_{l,d}}|\nabla Y_l^m|^2+Y_l^m\Delta Y_l^m=0$, 
and since $-\Delta Y_l^m=\lambda_{(l)} Y_l^m$, the first identity follows (this is a well-known identity, see e.g., \cite{ilyin2}). As for the second formula, recall that from Bochner's formula we have
$$
|D^2u|^2+(d-1)|\nabla u|^2=|D^2u|^2+{\rm Ric}(\nabla u,\nabla u)=\frac{1}{2}\Delta(|\nabla u|^2)-\nabla\Delta u\cdot\nabla u.
$$
Then 
\begin{multline*}
\sum_{m=1}^{m_{l,d}}|D^2Y_l^m|^2+(d-1)|\nabla Y_l^m|^2=\frac{1}{2}\Delta\left(\sum_{m=1}^{m_{l,d}}|\nabla Y_l^m|^2\right)-\sum_{m=1}^{m_{l,d}}\nabla\Delta Y_l^m\cdot\nabla Y_l^m\\
=\lambda_{(l)}\sum_{m=1}^{m_{l,d}}|\nabla Y_l^m|^2=\frac{m_{l,d}\lambda_{(l)}^2}{|\mathbb S^d|}.
\end{multline*}
\end{proof}

We prove now a lower bound for the eigenvalues $M_j(\Omega)$ of \eqref{Neumann-bi} on $\Omega$ by means of the averaged variational principle, in the spirit of \cite{EHIS} (see also \cite{ilyin_laptev,stri}).

\begin{prop}\label{prop_EIHS_bil}
For all $z\geq 0$
$$
\sum_j(z-M_j(\Omega))_+\geq\frac{|\Omega|}{|\mathbb S^d|}\sum_l m_{l,d}(z-\lambda_{(l)}^2)_+.
$$
\end{prop}
\begin{proof}
As done several times through the paper, but somehow ``in a reversed way'', we apply the averaged variational principle to the eigenvalues $M_j(\Omega)$, with associated eigenfunctions $u_j$, using as test function the (restrictions to $\Omega$ of the) Laplacian eigenfunctions $Y_l^m$ on $\mathbb S^d$, which trivially belong to $H^2(\Omega)$.   From \eqref{sum_grad_hess} we deduce
\begin{multline*}
\sum_j(z-M_j(\Omega))_+\sum_l\sum_{m=1}^{m_{l,d}}\left|\int_{\Omega}Y_l^m u_j\right|^2\geq\sum_{l\in L}\sum_{m=1}^{m_{l,d}}\int_{\Omega}z|Y_l^m|^2-|D^2Y_l^m|^2-(d-1)|\nabla Y_l^m|^2\\
=\sum_{l\in J}\frac{m_{l,d}|\Omega|}{|\mathbb S^d|}\left(z-\lambda_{(l)}^2\right),
\end{multline*}
where $J$ is an arbitrary subset of $\mathbb N$. We can use Parseval's identity at the left-hand side of the inequality, while we can take the sum over $l\in\mathbb N$ such that $z-\lambda_{(l)}^2\geq 0$ on the right-hand side. This concludes the proof.
\end{proof}

We denote by $R_1^{2,N}$ the first Riesz-mean for problem \eqref{Neumann-bi}:
$$
R_1^{2,N}(z):=\sum_j (z-M_j(\Omega))_+.
$$

Combining Proposition \ref{prop_EIHS_bil} with the lower bounds on the first Riesz-mean $R_1^2$ for $\lambda_{(l)}^2$ obtained in Theorem \ref{lower-R-1-bil-sphere} we deduce the following

\begin{thm}\label{kr-gen-poly}
Let $\Omega$ be a domain in $\mathbb S^d$, $d\geq 3$. Then for all $z\geq 0$ we have the following inequality for the first Riesz-mean $R^{2,N}_1$ of Neumann biharmonic eigenvalues on $\Omega$:
$$
R_1^{2,N}(z)\geq L^{class}_{1,d,2}|\Omega|z^{1+  \frac{d}{4}}.
$$
\end{thm}

\section{Sum rules for Laplacian eigenvalues on compact symmetric spaces of rank one}\label{compact-two-points}

In this section we present an approach, which is not variational, to obtain bounds on Riesz-means, based on identities for spectral quantities, which we call sum rules referring to early quantum mechanics (see e.g., \cite{BJ,HaSt0}). In \cite{HaSt0} such identities were first derived for Dirichlet eigenvalues and Schr\"{o}dinger operators in the Euclidean setting, leading to asymptotically sharp universal eigenvalue inequalities and, by the discovery made in \cite{HaHe0}, to sharp Weyl-type bounds for Riesz-means and to sharp Lieb-Thirring inequalities for Schr\"{o}dinger operators \cite{St0}. Sum rules for manifolds immersed in the Euclidean space were first derived in \cite{EHI,Ha0}  and later in \cite{HaSt0}, where also an algebraic identity for the spectrum of an abstract selfadjoint operator $H$ defined on a Hilbert space $\mathcal{H}$ with scalar product $\langle ,\rangle$ was shown. For the sake of completeness, we state it for the situation when the spectrum of $H$  consists of eigenvalues $\lambda_j$, with an associated
orthonormal basis of eigenfunctions
$\left\{\phi_j\right\}$. If $J$ is a subset of the spectrum, $G$ is a linear operator (satifying suitable domain hypotheses, see \cite{HaSt1}) and 
$[H,G]=HG-GH$ is the commutator of $H$ and $G$, then
\begin{multline*}
    \frac1{2}\sum_{\lambda_j\in J}  (z - \lambda_j)^2\,\left(\langle[G^*,[H,G]]\phi_j,\phi_j\rangle+\langle[G, [H,G^*]]\phi_j,\phi_j\rangle\right)\\
    -\sum_{\lambda_j\in
    J}(z-\lambda_j)\,\left(\langle[H,G]\phi_j,[H,G]\phi_j\rangle+\langle[H,G^*]\phi_j,[H,G^*]\phi_j\rangle\right)\\
    =
    \sum_{\lambda_j\in J}\sum_{\lambda_k\notin J}
    (z-\lambda_j)(z-\lambda_k)(\lambda_k-\lambda_j)\left(|\langle G\phi_j,\phi_k\rangle|^2+|\langle G^*\phi_j,\phi_k\rangle|^2\right),
\end{multline*}
If $J=\{\lambda_1,\ldots,\lambda_N\}$ then the right-hand side of the above identity has a sign for all $z\in [\lambda_N,\lambda_{N+1}]$. More precisely, an upper bound of the right-hand side for $z$ in this interval is given by
\begin{equation*}
  \frac1{2}\,(z - \lambda_N)(z - \lambda_{N+1})\sum_{\lambda_j\in J} \,\left(\langle[G^*,[H,G]]\phi_j,\phi_j\rangle+\langle[G, [H,G^*]]\phi_j,\phi_j\rangle\right).
\end{equation*}

Therefore we arrive at an inequality between two quadratic polynomials in $z$ involving spectral quantities. This type of inequalities, for suitable choices of $G$, turns out to be equivalent to certain  properties for  Riesz-means. We shall present a few instances in the following subsections, where we will deal with the sphere, and with the other compact two-point homogeneous spaces, for which results in the spirit of the previous sections are much more difficult to get by variational techniques.

\subsection{The sphere $\mathbb S^d$}
Let us consider the eigenvalues $\lambda_j(\Omega)$ of the Dirichlet Laplacian on domains $\Omega \subset \mathbb{S}^d$ (hence in particular for $\Omega = \mathbb{S}^d$). The following result was shown in \cite{EHI,HaSt0}. Let 
Let $P_N(z)$ and $Q_N(z)$ be the quadratic polynomials defined by
\begin{equation}\label{P-N-of-z-sphere}
\begin{split}
    P_N(z) & =\sum_{j=1}^{N}(z-\lambda_j(\Omega))\left(z-d-\frac{d+4}{d}\,\lambda_j(\Omega)\right) \\
     & =Nz^2-2\,\frac{d+2}{d}\left(\sum_{j=1}^{N}\lambda_j(\Omega)\right)z-dNz+ \frac{d+4}{d}\,\sum_{j=1}^{N}\lambda^2_j(\Omega) +d\sum_{j=1}^{N}\lambda_j(\Omega),\\
\end{split}
\end{equation}
and
\begin{equation}\label{Q-N-of-z-sphere}
  Q_N(z)=N(z-\lambda_N(\Omega))(z-\lambda_{N+1}(\Omega)).
\end{equation}
Then, for all $z\in [\lambda_N,\lambda_{N+1}]$ the inequality
\begin{equation}\label{Polynomial-inequality-sphere}
  P_N(z)\leq Q_N(z)
\end{equation}
holds. When $\Omega=\mathbb S^d$,  in \cite{EHI} it is conjectured on the basis of computations using a computer algebra system that
\begin{equation}\label{Polynomial-identity-sphere}
  P_N(z)= Q_N(z)
\end{equation}
for all $N\geq 1$ such that $\lambda_{N}<\lambda_{N+1}$ (when $\Omega=\mathbb S^d$ we omit the dependence of the eigenvalues on the domain). In the following, we discuss some consequences of  \eqref{Polynomial-inequality-sphere}-\eqref{Polynomial-identity-sphere} and give a proof of  \eqref{Polynomial-identity-sphere}. Note that \eqref{Polynomial-identity-sphere} is already proved in \cite{R}. However, our proof of \eqref{Polynomial-identity-sphere} can be extended  to cover the case of compact two-point homogeneous spaces, and therefore we get also the consequences of the resulting identity. We will do this in the next subsection.

We start by discussing the consequences of \eqref{Polynomial-inequality-sphere}-\eqref{Polynomial-identity-sphere}. It is proved in \cite{HaSt1} that inequality \eqref{Polynomial-inequality-sphere} is equivalent to the following property of the Riesz-mean $R_2(z)=\sum_{j\geq 1}(z-\lambda_j(\Omega))^2_+$.

\begin{prop}
  Let $R_1,R_2$ be the first and second Riesz-means of the Dirichlet Laplacian on $\Omega\subset \mathbb{S}^d$. Then
  \begin{equation}\label{R-2-R-1-inequality-sphere}
  \frac{d+4}{4}\,R_2(z)-\left(z+\frac{d^2}{4}   \right)\,R_1(z)\leq 0
\end{equation}
or, equivalently
\begin{equation}\label{R-2-diff-inequality-sphere}
  \frac{d}{dz}\frac{R_2(z)}{(z+\frac{d^2}{4})^{2+d/2}}\geq 0.
\end{equation}
In view of Weyl's law, this implies the upper bound
\begin{equation}\label{R-2-Weyl-shifted-upper-bound-sphere}
  R_2(z)\leq L_{2,d}^{class}|\Omega|\left(z+\frac{d^2}{4}\right)^{2+\frac{d}{2}}.
\end{equation}
\end{prop}
\begin{rem}
  The shift $\frac{d^2}{4}$ was interpreted as one quarter of the mean curvature squared since this geometrical quantity arises naturally from the commutators of the Laplacian with an appropriate multiplication operator for immersed manifolds with nonconstant curvature. Below we give another interpretation of this quantity.
\end{rem}
\begin{rem}
The fact that $R_2(z)z^{-2-d/2}$ is increasing is a much stronger statement than just that the leading term in Weyl's law is an upper bound. We put this property in relation to the bounds obtained in this paper for $R_1(z)$. First we note that for $R_1(z)$ the quantity $R_1(z)(z+b)^{-1-d/2}$, which converges to $L_{1,d}^{class}|\Omega|$ as $z$ tends to infinity, is not increasing, independently of the choice of $b$. For the shifts $b=0$ and $b=1/2$ (from Proposition \ref{prop:R1S2bounds} for $\Omega=\mathbb{S}^2$, or more generally the shift $b=z_d<d^2/4$ from Theorem \ref{thm:ubR1dshift} when $\Omega=\mathbb{S}^d$) the quantity $R_1(z)(z+b)^{-1-d/2}$ is oscillating and has a local maximum between two energy levels. On the other hand, bounds on $R_1(z)$ imply bounds on $R_2(z)$ by integrating $R_1(z)$ so that for $\Omega=\mathbb{S}^d$ we also have
\begin{equation*}
  R_2(z)\leq L_{2,d}^{class}|\Omega|(z+z_d)^{2+\frac{d}{2}},\quad z_d=\frac{(2d-1)d}{12}
\end{equation*}
which is better than  the bound \eqref{R-2-Weyl-shifted-upper-bound-sphere}. However, $R_2(z)(z+z_d)^{-2-\frac{d}{2}}$ is not increasing anymore as it is easily verified. More generally it will not be increasing for any shift $b<d^2/4$. On the other hand, for $\mathbb{S}^2$ it is easy to verify that $R_2(z)z^{-3}$ is decreasing (to the leading term in Weyl's law). This behavior may change when passing from the entire sphere to subdomains. Indeed, for the Dirichlet Laplacian on the hemisphere the quantity $R_2(z)z^{-3}$ is increasing.
\end{rem}
\begin{rem}
  Inequality \eqref{R-2-R-1-inequality-sphere} for $R_2(z)$ can be transformed in various ways. For example, for $\mathbb{S}^d$ the following lower bound is useful
\begin{equation*}
  \sum_{j}(z^2-\lambda_j^2)_{+}\geq \frac{2d+4}{d+4}\,zR_1(z)-\frac{d^2}{d+4}\,R_1(z)
\end{equation*}
which shows that the first Riesz-mean for the biharmonic operator on $\mathbb{S}^d$ satisfies a Weyl-type lower bound up to a term of lower order.
\end{rem}

Another consequence of \eqref{Polynomial-identity-sphere} is the following lower bound on $R_2(z)$.

\begin{prop}

  The quantity $\displaystyle R_2(z)z^{-2-d/2}$ is decreasing to the leading term in Weyl's law and therefore
  \begin{equation*}
    R_2(z)\geq L_{2,d}^{class}|\mathbb{S}^d|z^{2+d/2}
  \end{equation*}
\end{prop}
\begin{proof}
  Since
  \begin{equation*}
    \frac{d}{dz}R_2(z)z^{-2-d/2}=2z^{-3-d/2}\left(zR_1(z)-\left(1+\frac{d}{4}\right)R_2(z)\right)
  \end{equation*}
  and for all $z\in[\lambda_N,\lambda_{N+1}]$
  \begin{equation*}
    Q_N(z)=P_N(z)=\left(1+\frac{4}{d}\right)R_2(z)-\frac{d}{4}\,z R_1(z)-d R_1(z)
  \end{equation*}
  we have, for all  $z\in[\lambda_N,\lambda_{N+1}]$,
  \begin{equation*}
     \frac{d}{dz}R_2(z)z^{-2-d/2}=-\,\frac{d}{2} z^{-3-d/2}\left( Q_N(z) +d R_1(z)\right).
  \end{equation*}
  We show that the quadratic polynomial $Q_N(z) +d R_1(z)$ is always positive which proves the claim. In fact, it has a critical point, which is a local minimum, 
at $\displaystyle z_0=\frac{\lambda_N+\lambda_{N+1}-d}{2}$.
Comparing the coefficients of $Q_N(z)$ and $P_N(z)$ we note that
$$
\frac{d}{d+2}\sum_{j=1}^N\lambda_j=Nz_0.
$$	
Therefore
$$
\frac{1}{N}\left(Q_N(z_0)+dR_1(z_0)\right)=-z_0^2+\lambda_N\lambda_{N+1}-\frac{d^2}{d+2}z_0.
$$
When there is an eigenvalue gap, that is, $\lambda_N<\lambda_{N+1}$, then $\lambda_N=L(L+d-1)$ and $\lambda_{N+1}=(L+1)(L+d)$ for some natural number $L$. Hence
$$
\frac{1}{N}\left(Q_N(z_0)+dR_1(z_0)\right)=\frac{d-2}{d+2}L(L+d)\geq 0.
$$
\end{proof}

Now we turn again to the Laplacian eigenvalues on $\mathbb{S}^d$ and prove the identity \eqref{Polynomial-identity-sphere}. 
\begin{thm}\label{P=Q}
  Let $P_N,Q_N$ be defined by \eqref{P-N-of-z-sphere} and \eqref{Q-N-of-z-sphere}, respectively, with $\lambda_j(\Omega)=\lambda_j$, the eigenvalues of the  Laplacian  on $\mathbb{S}^d$. Let $N\geq 1$ such that $\lambda_{N}<\lambda_{N+1}$. Then
  \begin{equation*}
  P_N(z)= Q_N(z).
\end{equation*}

\end{thm}
\begin{proof}
First we note that
 \begin{equation}\label{P-N-difference-sphere}
  P_{N+1}(z)-P_N(z)=(z-\lambda_{N+1})\left(z-d-\frac{d+4}{d}\,\lambda_{N+1}\right)
\end{equation}
We prove the claim by induction. Since $\lambda_1=0$ and $\lambda_2=d$ and from the definition $P_1(z)= z(z-d)$ the assertion is true for $N=1$.
It is now sufficient to prove that $P_N(\lambda_N)=0$ implies $P_N(\lambda_{N+1})=0$ since then by  \eqref{P-N-difference-sphere}
we also have $P_{N+1}(\lambda_{N+1})=0$ and we conclude. Let $N\geq 1$ and suppose $P_N(\lambda_N)=0$. Then
\begin{equation*}
  P_{N}(\lambda_{N+1})=P_{N}(\lambda_{N+1})-P_{N}(\lambda_{N})
  =(\lambda_{N+1}-\lambda_{N})
  \left(N(\lambda_{N+1}+\lambda_{N}-d)-2\,\frac{d+2}{d}\sum_{j=1}^{N}\lambda_j\right).
\end{equation*}
If $\lambda_{N+1}=\lambda_{N}$ the claim holds. If $\lambda_{N+1}>\lambda_{N}$ then
\begin{equation*}
  N=\sum_{l=0}^{L}m_{l,d}
\end{equation*}
for some positive integer $L$ and $\lambda_{N}=L(L+d-1)$, $\lambda_{N+1}=(L+1)(L+d)$ (here $m_{l,d}$ denotes the multiplicity of the eigenvalue corresponding to the angular momentum $l$ on $\mathbb{S}^d$). We now prove that in this case
\begin{equation*}
  N(\lambda_{N+1}+\lambda_{N}-d)-2\,\frac{d+2}{d}\sum_{j=1}^{N}\lambda_j=0,
\end{equation*}
that is
\begin{equation*}
  NL(L+d)=\frac{d+2}{d}\sum_{l=0}^{L}l(l+d-1)m_{l,d}
\end{equation*}
or equivalently
\begin{equation}\label{sphere-polynomial-reccurence-formula}
 \sum_{l=0}^{L}\bigg((d+2)l(l+d-1)-dL(L+d)\bigg)m_{l,d}=0.
\end{equation}
It is immediate to check that \eqref{sphere-polynomial-reccurence-formula} follows from \eqref{S-d-R-1}.
\end{proof}
\begin{rem}
  Note that the identity $P_N(z)=Q_N(z)$ implies the Riesz-mean inequalities \eqref{R-2-R-1-inequality-sphere} and \eqref{R-2-diff-inequality-sphere}. We will prove this fact for other homogeneous spaces and derive a Riesz-mean inequality by means of a sum rule,  and therefore get sharp Weyl type bounds without applying a variational principle.
\end{rem}

Another application of $P_N(z)=Q_N(z)$ is the existence of a trace identity for the spectrum.

\begin{thm}\label{invariant_S}
 Let $\displaystyle \tilde{\lambda}_{(l)}=l(l+d-1)+\frac{d^2}{4}$ be the shifted energy levels of the Laplacian on $\mathbb{S}^d$. Then
 \begin{equation}\label{Sum-rule-S-d-2}
      \,L_{0,d}^{class}|\mathbb{S}^d| =
       \,\sum_{l\geq 0}\frac{2l+d}{d}\,\binom{d+l-1}{d-1}\left(\tilde{\lambda}_{(l+1)}^{-d/2}-\tilde{\lambda}_{(l)}^{-d/2}
+\frac{d}{4}(\tilde{\lambda}_{(l+1)}^{-1-d/2}+\tilde{\lambda}_{(l)}^{-1-d/2})(\lambda_{(l+1)}-{\lambda}_{(l)})\right).
 \end{equation}
\end{thm}
\begin{proof}
For all $z\in [\lambda_N,\lambda_{N+1}]$ from Theorem \ref{P=Q}, and from the identity 
$$
P(z)=-\frac{2}{d}(z+b)^{3+d/2}\frac{d}{dz}\frac{R_2(z)}{(z+b)^{2+d/2}},
$$
we get
\begin{equation}\label{diff-identity-sphere-1}
  \frac{d}{dz}\frac{R_2(z)}{(z+b)^{2+d/2}}=-\frac{d}{2}\frac{Q_N(z)}{(z+b)^{3+d/2}},
\end{equation}
where $b=\frac{d^2}{4}$. Integrating the right-hand side for $z\in[\lambda_N,\lambda_{N+1}]$ we get
\begin{equation}\label{identity}
      -\frac{d}{2}\int_{\lambda_{N}}^{\lambda_{N+1}}\frac{Q_N(z)}{(z+b)^{3+d/2}}\,dz 
      =\frac{8N}{(d+2)(d+4)}\left(\tilde{\lambda}_{N+1}^{-d/2}-\tilde{\lambda}_{N}^{-d/2}
+\frac{d}{4}(\tilde{\lambda}_{N+1}^{-1-d/2}+\tilde{\lambda}_{N}^{-1-d/2})({\lambda}_{N+1}-{\lambda}_{N})\right).
\end{equation}
where
$$
\tilde{\lambda}_N=\lambda_N+b=\lambda_N+\frac{d^2}{4}.
$$
Summing identity \eqref{identity} over all $N\in\mathbb N$, and using the fact the the integral of the left-hand side of \eqref{diff-identity-sphere-1} between $0$ and infinity equals $L_{2,d}^{class}|\mathbb{S}^d|$, we get
\begin{equation*}
  L_{2,d}^{class}|\mathbb{S}^d|=\sum_{N=1}^{\infty }\frac{8N}{(d+2)(d+4)}\bigg(\tilde{\lambda}_{N+1}^{-d/2}-\tilde{\lambda}_{N}^{-d/2}
+\frac{d}{4}(\tilde{\lambda}_{N+1}^{-1-d/2}+\tilde{\lambda}_{N}^{-1-d/2})({\lambda}_{N+1}-{\lambda}_{N})\bigg).
\end{equation*}
Since non-zero contribution only appear when $\displaystyle N=\sum_{j=0}^{l}m_{j,d}=\frac{2l+d}{d}\binom{d+l-1}{d-1}$ for all $l\geq 0$, and since $\frac{(d+2)(d+4)}{8} \,L_{2,d}^{class}=L_{0,d}^{class}$, the result follows.
\end{proof}
\begin{rem}
When the dimension $d$ is even the expression in \eqref{Sum-rule-S-d-2} is easily simplified. For example when $d=2$,  \eqref{Sum-rule-S-d-2} yields
\begin{equation*}
  1=\sum_{l\geq 0}\frac{(l+1)^2}{2}\,\frac{(\lambda_{(l+1)}-\lambda_{(l)})^3}{\tilde{\lambda}_{(l)}^2\tilde{\lambda}_{(l+1)}^2}=
  \sum_{l\geq 0}\frac{4(l+1)^5}{(l^2+l+1)^2(l^2+3l+3)^2}.
\end{equation*}
For $d=1$ we get the identity
\begin{equation*}
  1=\sum_{l\geq 0}\frac{2l+1}{8}\,\frac{(\lambda_{(l+1)}^{1/2}-\lambda_{(l)}^{1/2})^3
  (\lambda_{(l)}+\lambda_{(l+1)}+3\lambda_{(l)}^{1/2}\lambda_{(l+1)}^{1/2})}{\tilde{\lambda}_{(l)}^{3/2}\tilde{\lambda}_{(l+1)}^{3/2}},
\end{equation*}
which can be made more explicit being $\lambda_{(l)}=l^2$.
\end{rem}

\subsection{The other compact symmetric spaces of rank one}

We consider now sum rules for the Laplacian eigenvalues on  compact two-point homogeneous spaces. This will allow us to prove Weyl sharp bounds on Riesz-means without applying any variational principle.


The classification of compact two-point homogeneous spaces is well-known. In particular, by the classical result \cite{wang_two_points_52} these spaces coincide with the compact symmetric spaces of rank one, and can be listed as follows,
\begin{enumerate}[i)]
\item $M^d=\mathbb S^d$, $d=1,2,3,...$ (i.e., the unit sphere);
\item $M^d=\mathbb P^d(\mathbb R)$, $d=2,3,4,...$ (i.e., the real projective space);
\item $M^d=\mathbb P^d(\mathbb C)$, $d=4,6,8,...$ (i.e., the complex projective space);
\item $M^d=\mathbb P^d(\mathbb H)$, $d=8,12,16,...$ (i.e., the quaternion projective space);
\item $M^d=\mathbb P^d(Cay)$, $d=16$ (i.e., the Cayley projective space).
\end{enumerate}

We resume here a few properties of the spectrum of the Laplacian on $M^d$. For the proofs and for more details we refer to \cite{he2,he1}.

\begin{prop}\label{prop:EigHS}
Let $M^d$ a compact two-point homogeneous space of dimension $d$, that is, one of the spaces {\rm i)-v)}. The energy levels of the Laplacian on $M^d$ are given by the numbers
\begin{equation*}
\lambda_{(l)}=
\begin{cases}
l(l+d-1) & {\rm for\ } M^d=\mathbb S^d\\
2l(2l+d-1) & {\rm for\ } M^d=\mathbb P^d(\mathbb R)\\
\frac{l(2l+d)}{2} & {\rm for\ } M^d=\mathbb P^d(\mathbb C)\\
\frac{l(2l+d+2)}{2} & {\rm for\ } M^d=\mathbb P^d(\mathbb H)\\
\frac{l(2l+d+6)}{2} & {\rm for\ } M^d=\mathbb P^d(Cay)
\end{cases}
\end{equation*}
for $l\in\mathbb N$. Each eigenvalue corresponding to an energy level $\lambda_{(l)}$ has multiplicity $m_{l,d}$ given by
\begin{equation*}
m_{l,d}=
\begin{cases}
\frac{2l+d-1}{l}\binom{d+l-2}{d-1} & {\rm for\ } M^d=\mathbb S^d\\
\frac{4l+d-1}{2l}\binom{d+2l-2}{d-1} & {\rm for\ }  M^d=\mathbb P^d(\mathbb R)\\
\frac{d+4l}{d}\binom{\frac{d}{2}+l-1}{\frac{d}{2}-1}^2 & {\rm for\ } M^d=\mathbb P^d(\mathbb C)\\
\frac{4l+d+2}{2l(l+1)}\binom{\frac{d}{2}+l-1}{\frac{d}{2}-1}\binom{\frac{d}{2}+l}{\frac{d}{2}+1} & {\rm for\ } M^d=\mathbb P^d(\mathbb H)\\
\frac{3(4l+d+6)}{l(l+1)(l+2)(l+3)}\binom{\frac{d}{2}+l-1}{\frac{d}{2}-1}\binom{\frac{d}{2}+l+2}{\frac{d}{2}+3} & {\rm for\ }  M^d=\mathbb P^d(Cay)
\end{cases}
\end{equation*}
\end{prop}

The identity $P_N(z)=Q_N(z)$
holds for all compact homogeneous spaces listed above in the following sense. Let $\lambda_j$ denote the eigenvalues of the Laplacian on $M^d$ (enumerated in increasing order and counted with their multiplicity). Let
\begin{equation*}
    P_N(z)  =\sum_{j=1}^{N}(z-\lambda_j)\left(z-\lambda-\frac{d+4}{d}\,\lambda_j\right) 
      =Nz^2-2\,\frac{d+2}{d}\left(\sum_{j=1}^{N}\lambda_j\right)z-\lambda z+ \frac{d+4}{d}\,\sum_{j=1}^{N}\lambda^2_j +\lambda\sum_{j=1}^{N}\lambda_j,.
\end{equation*}
Here $\lambda$ denotes the first non-trivial Laplacian eigenvalue of the homogeneous space (that is $\lambda_{(1)}=\lambda_2$), and, as before
\begin{equation*}
  Q_N(z)=N(z-\lambda_N)(z-\lambda_{N+1}).
\end{equation*}
Then one can prove $P_N(z)=Q_N(z)$ for all $\lambda_N\leq z \leq \lambda_{N+1}$ by induction exactly as it has been done in the previous subsection, and all the consequences, such as Weyl sharp bounds on $R_2$, hold for any of the spaces i)-v).

\begin{thm}
Let $R_2$ be the second Riesz-mean of the Laplacian on $M^d$, where $M^d$ is one of the spaces i)-v). Then, for all $z\geq 0$
$$
L_{2,d}^{class}|M^d|z^{2+\frac{d}{2}}\leq R_2(z)\leq L_{2,d}^{class}|M^d|\left(z+\frac{d}{4}\lambda\right)^{2+\frac{d}{2}},
$$
where $\lambda=\lambda_{(1)}$ is the first positive eigenvalue of the Laplacian on $M^d$. 
\end{thm}

Since in the proof of Theorem \ref{invariant_S} only the relation $P_N(z)=Q_N(z)$ is used, an analogous result holds for any of the space i)-v).

\begin{thm}
 Let $\displaystyle \tilde{\lambda}_{(l)}=\lambda_{(l)}+\frac{d\lambda}{4}$ be the shifted energy levels of the Laplacian on $M^d$, with multiplicities $m_{l,d}$, where $M^d$ is one of the spaces i)-v). Then
 \begin{equation*}
      L_{0,d}^{class}|M^d| =
       \sum_{l\geq 0}\sum_{j=0}^lm_{j,d}\left(\tilde{\lambda}_{(l+1)}^{-d/2}-\tilde{\lambda}_{(l)}^{-d/2}
+\frac{d}{4}(\tilde{\lambda}_{(l+1)}^{-1-d/2}+\tilde{\lambda}_{(l)}^{-1-d/2})(\lambda_{(l+1)}-{\lambda}_{(l)})\right).
 \end{equation*}
\end{thm}


\appendix
\section{Taylor and asymptotic expansions}\label{app_A}

We collect in this appendix a few Taylor expansions for real-valued functions and also asymptotic expansions for Gamma functions, which are used in the proofs of Theorems \ref{R1-d-sphere}, \ref{ND}, and \ref{R1-d-hemi-N}. We will also recall a formula for sums of binomial coefficients and an example of a related computation.


For real $a,b,x$ let $P_{a,b}$ the quadratic polynomial in $x$ defined by
\begin{equation*}
  P_{a,b}(x)=1+ax+bx^2
\end{equation*}
We have the following
\begin{lemma}\label{lem_1}
As $x\to 0$ we have
\begin{equation}\label{1-over-P-a-b-of-x}
  \frac{1}{ P_{a,b}(x)}= P_{-a,a^2-b}(x)+O(x^3).
\end{equation}
For $c\in\mathbb R$, as $x\to 0$ we have
\begin{equation}\label{P-a-b-of-x-combined}
   P_{a,b}\left(\frac{x}{1+cx}\right)= P_{a,b-ac}(x)+O(x^3).
\end{equation}

\end{lemma}
Note that the above expansions remain valid if we add an $O(x^3)$-term to $P_{a,b}(x)$. Moreover, we have also the following
\begin{lemma}\label{lem_2}
For any positive integer $n$ and for $a_j,b_j\in\mathbb R$, $j=1,\ldots,n$, let $A=\sum_{j=1}^{n} a_j$, $B=\sum_{j=1}^{n} b_j$ and $C=\sum_{j=1}^{n}\sum_{i=1}^{j-1} a_i a_j=\frac{1}{2}\left(A^2-\sum_{j=1}^{n}a_j^2\right)$.
Then
\begin{equation}\label{products-of-P-a-b-of-x}
  \prod_{j=1}^{n}P_{a_j,b_j}(x)= P_{A,B+C}(x)+O(x^3)
\end{equation}
\end{lemma}

We shall need the following expansions of Gamma, power-type and exponential functions (see e.g., \cite[Chapter 5]{nist}).
\begin{lemma}\label{lem_3}
The following asymptotic expansions hold:
\begin{enumerate}[i)]
\item
\begin{equation*}
  \Gamma(x)=\sqrt{2\pi}\,x^{x-1/2}e^{-x}\left(1+\frac{1}{12x}+\frac{1}{288x^2}+O(x^{-3})\right)
  =\sqrt{2\pi}\,x^{x-1/2}e^{-x}\left(P_{\frac{1}{12},\frac{1}{288}}(x^{-1})+O(x^{-3})\right)
\end{equation*}
as $x\to\infty$.
\item 
\begin{equation*}
  e^{-a}(1+ax)^{1/x}=1-\frac{a^2}{2}\,x+a^3(\frac{1}{3}+\frac{a}{8})\,x^2+O(x^3).
\end{equation*}
as $x\to 0$.
\item
For any $p>0$ 
\begin{equation*}
  (\sqrt{x+a^2}-b)^{p}=x^{p/2}\left(1-pbx^{-1/2}+p\,\frac{(p-1)b^2+a^2}{2}\,x^{-1}+O(x^{-3/2})\right)
\end{equation*}
as $x\to\infty$.
\end{enumerate}
\end{lemma}


\section{Duality in the averaged variational principle for estimating averages of increasing sequences}\label{app_C}

In this appendix we discuss a duality aspect in the averaged variational principle, which reflects in a duality principle for Berezin-Li-Yau and Kr\"oger bounds on sums.


Let  $(a_j)_j,(b_j)_j$ be two sequences of non-negative increasing numbers. When applying the averaged variational principle (see Theorem \ref{thm:AVP}) we show typically an inequality of the following form: for all $z\in[a_N,a_{N+1}]$
\begin{equation}\label{AVP-inequality-1}
\sum_{k=1}^N(z-a_k)\geq p \sum_{j\in J}(z-b_j)
\end{equation}
where $N\in\mathbb N\setminus\{0\}$, $J\subset\mathbb N\setminus\{0\}$ are arbitrary, and $p>0$ is some positive constant. Let
\begin{equation*}
  R_1^{(a)}(z)=\sum_{k}(z-a_k)_{+},\quad R_1^{(b)}(z)=\sum_{k}(z-b_k)_{+}
\end{equation*}
be the Riesz-means of the sequences $(a_j)_j,(b_j)_j$. Choosing $J$ such that the sum on the right-hand side of  \eqref{AVP-inequality-1} equals $R_1^{(b)}(z)$ one has, for all $z\geq 0$, the Riesz-mean inequality
\begin{equation*}
  R_1^{(a)}(z)\geq p\,R_1^{(b)}(z).
\end{equation*}
Moreover, for all positive integers $N$ and $z\in [a_N,a_{N+1}]$ one also has
\begin{equation*}
  \sum_{k=1}^{N}a_k\leq Nz-p\,R_1^{(b)}(z)
\end{equation*}
which trivially implies
\begin{equation*}
  \sum_{k=1}^{N}a_k\leq \underset{z\geq 0}{\max}\left(Nz-p\,R_1^{(b)}(z)\right)
\end{equation*}
On the other hand, choosing $J=\{1,\ldots, N\}$ and isolating  $\sum_{j\in J}b_j$ in \eqref{AVP-inequality-1} we get for all $z\geq 0$ the inequality
\begin{equation*}
  \sum_{j=1}^{N}b_j\geq Nz-p^{-1}\,R_1^{(a)}(z).
\end{equation*}
In particular, the above inequality holds at the maximum of the r.h.s. is attained, that is
\begin{equation*}
  \sum_{j=1}^{N}b_j\geq \underset{z\geq 0}{\max}\left(Nz-p^{-1}\,R_1^{(a)}(z)\right).
\end{equation*}
In the applications, we typically have simple lower bounds on $R_1^{(b)}(z)$ and upper bounds on $R_1^{(a)}(z)$ so that the maxima can be computed explicitly.

\section*{Acknowledgements}

The authors are grateful to the anonymous Referees for valuable comments helping to improve the paper. The first and the second authors are members of the Gruppo Nazionale per l'Analisi Ma\-te\-ma\-ti\-ca, la Probabilit\`a e le loro Applicazioni (GNAMPA) of the I\-sti\-tuto Naziona\-le di Alta Matematica (INdAM). The third author is member of the Gruppo Nazionale per le Strutture Algebriche, Geometriche e le loro Applicazioni (GNSAGA) of the I\-sti\-tuto Naziona\-le di Alta Matematica (INdAM).

\section*{Funding and competing interests} The fourth author acknowledges support of the SNSF project ``Bounds for the Neumann and Steklov eigenvalues of the biharmonic operator'', grant number 200021\_178736. The authors have no competing interests to declare that are relevant to the content of this article. 
 
 \bibliography{bibliography}{}
\bibliographystyle{abbrv}

\bigskip

$^\dagger$ Dipartimento per lo Sviluppo Sostenibile e la Transizione Ecologica, Università degli Studi del Piemonte Orientale ``A. Avogadro'', Piazza Sant'Eusebio 5, 13100 Vercelli (ITALY). \\E-mail: {\tt davide.buoso@uniupo.it}\\

$^\star$Dipartimento di Matematica, Università degli Studi di Padova, Via Trieste 63, 35121 Padova, (ITALY). \\E-mail: {\tt paolo.luzzini@unipd.it} \\

$^\ddag$Dipartimento di Scienze di Base e Applicate per l'Ingegneria, Sapienza Universit\`a di Roma, Via Antonio Scarpa 16, 00161 Roma (ITALY). \\E-mail: {\tt luigi.provenzano@uniroma1.it} \\

$^{\star\star}$Institute of Mathematics, EPFL, SB MATH SCI-SB-JS, Station 8, CH-1015 Lausanne (SWITZERLAND). \\E-mail: {\tt joachim.stubbe@epfl.ch}

\end{document}